\documentclass[a4paper, reqno]{amsart}

\usepackage{latexsym}
\usepackage{amsmath}
\usepackage{amsfonts}
\usepackage{amssymb}
\usepackage{amsxtra}
\usepackage{mathrsfs}
\usepackage{eucal}
\usepackage{tikz} \usetikzlibrary{calc} \tikzset{>=latex} \usetikzlibrary{backgrounds}
\usepackage[all]{xy}
\usepackage{color}
\usepackage{braket}
\usepackage{graphics, setspace}
\usepackage{etoolbox, textcomp}
\usepackage{comment}
\usepackage{changepage}
\usepackage{xcolor}
\definecolor{rouge}{rgb}{0.85,0.1,.4}
\definecolor{bleu}{rgb}{0.1,0.2,0.9}
\definecolor{violet}{rgb}{0.7,0,0.8}
\usepackage[colorlinks=true,linkcolor=bleu,urlcolor=violet,citecolor=rouge]{hyperref}
\usepackage{cleveref}
\usepackage{lscape}
\usepackage{subcaption}
\usepackage{hyperref}

\usepackage{tikz, tikz-cd}

\usepackage{enumerate} 

\usepackage{mathtools}

\DeclareFontFamily{U}{wncy}{}
\DeclareFontShape{U}{wncy}{m}{n}{<->wncyr10}{}
\DeclareSymbolFont{mcy}{U}{wncy}{m}{n}
\DeclareMathSymbol{\sha}{\mathord}{mcy}{"58} 

\usepackage{stmaryrd} 

\usepackage{bbm} 

\usepackage[abbrev,msc-links, alphabetic, backrefs]{amsrefs}
\usepackage{amsrefs}

\BibSpec{arXiv}{%
  +{}{\PrintAuthors}{author}
  +{,}{ \textit}{title}
  +{,}{ arXiv }{eprint}
}
\BibSpec{other}{%
  +{}{\PrintAuthors}{author}
  +{,}{ \textit}{title}
}

\usepackage{thmtools}
\newtheoremstyle{introTheorems}
  {\topsep}
  {\topsep}
  {\itshape}
  {0pt}
  {\bfseries}
  {}
  { }
  {\thmname{#1}
  \textnormal{\bf\thmnote{#3}.\!\!}
  }

\newtheorem{theorem}{Theorem}[section]
\newtheorem{assumption}[theorem]{Assumption}

\newtheorem{corollary}[theorem]{Corollary}
\newtheorem{definition}[theorem]{Definition}
\newtheorem{example}[theorem]{Example}
\newtheorem{lemma}[theorem]{Lemma}
\newtheorem{problem}[theorem]{Problem}

\newtheorem{remark}[theorem]{Remark}

\theoremstyle{introTheorems}
\newtheorem{introTheorem}{Theorem}

\def\CC{\mathbb{C}}

\def\KK{\mathbb{K}}

\def\ZZ{\mathbb{Z}}

\newcommand\cA{\mathcal{A}}
\newcommand\cB{\mathcal{B}}
\newcommand\cC{\mathcal{C}}
\newcommand\cD{\mathcal{D}}
\newcommand\cE{\mathcal{E}}

\newcommand\cH{\mathcal{H}}

\newcommand\cL{\mathcal{L}}
\newcommand\cM{\mathcal{M}}

\newcommand\cP{\mathcal{P}}

\newcommand\cS{\mathcal{S}}

\newcommand\cU{\mathcal{U}}

\newcommand\cZ{\mathcal{Z}}

\newcommand{\coker}{\textup{coker}}

\newcommand{\gr}{\textup{gr}}

\newcommand\id{\textup{id}}

\newcommand{\Ind}{\textup{Ind}}

\newcommand\Rep{\textup{Rep}}
\newcommand{\Res}{\textup{Res}}

\newcommand\triv{\textup{triv}}

\newcommand{\Vect}{\textup{Vect}}

\newcommand\Hom{\textup{Hom}}

\newcommand{\Ext}{\textup{Ext}}

\renewcommand\sl{\mathfrak{sl}}
\newcommand\Vir{\textup{Vir}}
\newcommand\loc{\textup{loc}}

\newcommand\quash[1]{}

\newcommand\vac{\mathbf{1}}

\renewcommand\a\alpha
\renewcommand\b\beta
\newcommand\g\gamma
\renewcommand\d\delta
\newcommand\D\Delta

\newcommand{\y}{\eta}

\DeclareSymbolFont{bbold}{U}{bbold}{m}{n}
\DeclareSymbolFontAlphabet{\mathbbold}{bbold}
\renewcommand{\1}{\mathbbold{1}}

\title{Affine $\mathfrak{sl}_2$ at admissible levels and  quantum $\mathfrak{sl}_{2|1}$}
\author{Thomas Creutzig and Simon Lentner}

\begin{document}

\begin{abstract}
We prove braided tensor equivalences between the categories of weight modules of the affine vertex algebra of $\mathfrak{sl}_2$ at any admissible level and categories associated to partially semisimplified versions of quantum $\mathfrak{sl}_{2|1}$ at roots of unity.
\end{abstract}

\maketitle

\newcommand{\catC}{\mathcal{C}}
\newcommand{\catB}{\mathcal{B}}
\newcommand{\catD}{\mathcal{D}}
\newcommand{\catZ}{\mathcal{Z}}

\newcommand{\C}{\mathbb{C}}
\newcommand{\Z}{\mathbb{Z}}

\newcommand{\Nichols}{H}
\newcommand{\NicholsOf}{\mathfrak{B}}

\newcommand{\YD}{\mathcal{YD}}

\renewcommand{\i}{\mathrm{i}}
\newcommand{\Comod}{\mathrm{Comod}}
\renewcommand{\split}{\mathrm{split}}
\renewcommand{\H}{\mathrm{H}} 
\newcommand{\Laugwitz}{\mathbb{V}} \newcommand{\coLaugwitz}{\mathrm{co}\mathbb{V}}

\newcommand{\NicholsAlgebraCounit}{\varepsilon}

\newcommand{\R}{\mathbb{R}}
\newcommand{\h}{\mathfrak{h}}
\renewcommand{\k}{\mathsf{k}}
\newcommand{\VirPhi}{L^\k}
\renewcommand{\g}{\mathfrak{g}}
\renewcommand{\sl}{\mathfrak{sl}}
\newcommand{\Uq}{\mathrm{U}_q}
\newcommand{\slhat}{\smash{\widehat{\sl}}}

\renewcommand{\YD}[1]{{^{#1}_{#1}}\mathcal{YD}}

\newcommand{\integral}{\Lambda}
\newcommand{\ellNichols}{\mathsf{a}}
\newcommand{\lambdaNichols}{\alpha}
\newcommand{\tillhere}{\newpage --------------- {\bf till here}}

\newcommand{\ra}{\mathrm{ra}}
\newcommand{\la}{\mathrm{la}}
\newcommand{\Fra}{{G}}

\newcommand{\Split}{\mathrm{split}}
\newcommand{\compare}{\mathrm{comp}}
\newcommand{\Schauenburg}{\mathcal{Sch}}

\newcommand{\monT}{{\mathsf{T}}}
\newcommand{\funT}{\Ind_\monT}
\newcommand{\monR}{{\mathsf{R}}}
\newcommand{\funR}{\Ind_{\monR}}
\newcommand{\comonC}{\mathsf{C}}
\newcommand{\comonZ}{\mathsf{Z}}
\newcommand{\funC}{\Ind_\comonC}
\newcommand{\forget}{f}

\newcommand{\funE}{E}
\newcommand{\funD}{D}
\newcommand{\funDD}{\tilde{D}}
\newcommand{\catDD}{\tilde{\catD}}

\newcommand{\fpullback}{h}

\newcommand{\cat}{\catB}

\frenchspacing

\newpage
\tableofcontents

\newpage

\section{Introduction}

A major source of braided tensor categories are categories of modules of quantum groups as well as certain semisimplifications. Another major source are vertex tensor categories, which are particular categories of modules over vertex algebras. Key vertex algebras are those associated to affine Lie (super)algebras and their $W$-algebras. Ordinary modules of affine vertex algebras are those whose top space is an integrable module for the horizontal subalgebra. Beyond the category of ordinary modules it is already  very difficult to establish the mere existence of a vertex tensor category structure. 
Building on the work of many over many years \cites{A, AM, ACK, CR1, CR2, CRW, KR, RW, Ri1, Ri2, Ri3} the existence of vertex tensor category for the category of weight modules of the simple affine vertex algebra $L_\k(\sl_2)$ of $\sl_2$ at any admissible level $\k$ has only recently been established \cite{Cr24}.

    As a main result of this paper we prove that this braided tensor category is equivalent to a partially semisimplified version of the category of modules of the quantum super group $\Uq(\sl(2|1))$ at roots of unity times a semisimplified quantum group $\Uq(\sl_2)$ (the latter however plays a rather minor role).  More precisely, we construct a novel nonsemisimple modular tensor category as a relative Drinfeld center of a Nichols algebra of the standard representation $M$ in the semisimplified category of modules of $\Uq(\sl_2)$. In view of $\sl_{2|1}$, the $\sl_2$ corresponds to the bosonic root and the standard representation $M$ corresponds to the two fermionic roots. 

    This equivalence is part of a larger program  that is motivated by  the Kazhdan-Lusztig correspondence \cites{KL93a, KL93b, KL93c} between affine Lie algebras $\hat{\g}$ at generic level and quantum groups $\Uq(\g)$ \cites{CLR21, CLR23, CN25, Len25} and which was developed by the authors to prove the logarithmic Kazhdan-Lusztig conjecture \cites{FGST05, FT10} between the Feigin-Tipunin vertex algebra and the small quantum group $u_q(\g)$ at roots of unity; with the main result that indeed the category of the singlet and triplet vertex algebras are braided equivalent to the category of the unrolled small quantum group and small quantum group of $\sl_2$ at corresponding root of unity.  Our proof technique from \cite{CLR23} can be summarized to be a reconstruction of an arbitrary braided tensor category from a known category of modules over a  commutative algebra, and the result is a relative Drinfeld center of a Nichols algebra. The functor to the relative Drinfeld center is called the Schauenburg functor and the desired equivalence amounts to showing that the Schauenburg functor is an equivalence. Criteria for being fully faithful have been derived in \cite{CLR23}  and in particular \cite[Theorem 3.20]{CMSY24}. Surjectivity however could only be verified via an explicit computation in \cite{CLR23}.
    The main technical results in the present article are significant improvements of this strategy by using monadic arguments and proving the surjectivity of the Schauenburg functor, as discussed below, and the application of this result to cases much beyond the original setup. 

    In the previously treated cases \cite{CLR23}, this reconstruction produced the category of representations of a vertex algebra from its free field realization and a Nichols algebra of screenings. The present article relies on the inverse quantum Hamiltonian reduction \cite{A}, which gives a realization of affine $\sl_2$ in the Virasoro algebra,   which corresponds to $\Uq(\sl_2) \otimes \mathrm{U}_{q'}(\sl_2)$, with additional free fields, together with an $\sl_2$-doublet  of screening operators. At admissible level, these quantum groups have to be replaced by their semisimplification and modularization. Nichols algebras in this modular tensor category were not studied systematically, to our knowledge, but in the situation before semisimplification compare \cite{CL17, AA20}. 
    
    The case of affine $\sl_2$ is of course only the smallest case of a general program to determine the braided tensor category of representations of affine Lie algebras and their $W$-algebras, which we discuss in the subsequent section as an outlook.

\subsection{Acknowledgments}

We thank M. Stroiński for valuable comments on monads and ideas from \cite{SZ24}, V. Ostrik for exposing us to the ideas in \cite{Ost97,CEO25}, see Problem \ref{prob:Ostrik}, and E. Mukhin for discussing \cite{FJM25,FJM26}, see Remark \ref{rem:Mukhin}. \\  SL thanks the CRC 1624 at the University Hamburg and the KAKHENI grant of the Japan Society for the Promotion of Science for partial support.

\section{Overview of results and outlook}

\subsection{Summary of results}

Our construction depends on the choice of an admissible level $\k=-2+\frac{u}{v}$ with $u,v$ relatively prime positive integers, both greater than one.

\bigskip

In Section \ref{sec:Cartan} we set up a semisimple category $\cC^\k$. From the perspective of the quantum group, this will replace the Cartan algebra, or rather its category or representations, the category of weight spaces. From the perspective of $\slhat_2$ at admissible level $\k$, this is a category of representations of the free field realization over the Virasoro algebra times a half lattice vertex algebra.

In Section \ref{subsec:CartanAn} we first recall certain semisimple braided tensor categories $\cA(p,p+q)$ depending on a parameter $s=e^{\pi\i\frac{p+q}{q}}$ with simple objects $X_0,\ldots,X_{p-2}$, whose fusion rules are called to be of type $A_n$ and are essentially truncated $\sl_2$ fusion rules. In Example~\ref{exm_tripleProduct} we in particular compute for later use the braiding and associator of the following double and triple product 
$$X_1\otimes X_1\cong X_0\oplus X_2,\qquad X_1\otimes X_1\otimes X_1\cong2X_1\oplus X_3$$

In Section \ref{subsec:CartanVir} we review the semisimple braided tensor category $\cC(u,v)$ of vertex algebra modules of 
the simple Virasoro algebra at central charge $c_{u,v}$ with simple objects $\VirPhi_{r,s}$ for $1 \leq  r \leq  u-1$ and $1 \leq s \leq v-1$. 

It has two tensor subcategories $\cC^L(u,v)$ and $\cC^R(u,v)$ with simple objects $\VirPhi_{r,1}=X^L_{r-1,1}$ and $\VirPhi_{1,s}=X^R_{1,s-1}$ and 
\[
\cC^L(u, v) \cong \cA(u, \tilde v), \qquad \cC^R(u, v) \cong \cA(v, \tilde u)
\]
with $\tilde v$ the smallest positive integer in $v + 4u \ZZ$ and $\tilde u $ the smallest positive integer in $u + 4v \ZZ$. 

In Section \ref{subsec:CartanLattice} we set up the semisimple braided tensor category $\Vect_{\Gamma}^Q$ for $\Gamma=\R/\Z\oplus \Z$ and a certain quadratic form, with simple objects $\Pi_\ell(\lambda)$ for $\ell\in\Z$ and $\lambda\in \R/\Z$, in a parametrization that will later be useful. It is the category of representations of the half lattice vertex algebra $\Pi(0)$, which is a rank 2 Heisenberg vertex algebra with generators $c,d$ and a hyperbolic inner product, extended by a rank 1 lattice along one of the isotropic directions $c$.

In Section \ref{subsec:Cartan} we combine the previous notions to define the \emph{Cartan category} for our article
\[\cC^\k = \cC(u,v)\boxtimes \Pi(0)\] with simple objects $\VirPhi_{1,s}\boxtimes\Pi_\ell(\lambda)$. We can also take the right Virasoro category giving
\[\cC^\k_R = \cC^R(u,v)\boxtimes \Pi(0)\]
and/or the Heisenberg vertex algebra or a full rank lattice algebra, this would lead  to slight variants of the quantum group, similar to unrolled or uprolled quantum groups, see subsection \ref{sec:variations}.

\bigskip

In Section \ref{sec:QuantumGroup} we construct a particular nonsemisimple nondegenerately braided tensor category $\cU^\k$ as a \emph{generalized quantum group} with \emph{Cartan part} $\cC^\k$ from the previous section.

In Section \ref{subsec:introNichols} we briefly recall the Nichols algebra $\NicholsOf(M)$ of an object $M$ in a braided tensor category $\cC$. The Nichols algebra  is a Hopf algebra in $\cC$ fulfilling several equivalent universal properties. As such, its category of representations inside $\cC$ has the structure of a tensor category, with the coproduct in $\cC$ and an $M$ acting like a braided derivation. A standard textbook for Nichols algebras is \cite{HS20}, the reader is also referred to the lecture notes \cite{Len26}. We also give examples of Nichols algebras relevant to our article, namely the trivial cases of a symmetric or exterior algebra, the truncated polynomial rings, and most importantly the Borel part of a small quantum group $u_q(\g)$.

In Section \ref{subsec:NicholsAlgebra} we compute the Nichols algebra $\NicholsOf=\NicholsOf(M)$ of the simple object $M=\VirPhi_{1,2}\boxtimes \Pi_\ellNichols(\lambdaNichols)$ for $v\geq 3$ with fixed parameters $(\ellNichols,\lambdaNichols)$ such that a certain expression \eqref{formula_NicholsParameterCondition} is an odd integer. As it turns out, as a graded object this Nichols algebra is $\NicholsOf=1\oplus M\oplus \integral$ with the invertible object $\integral=\VirPhi_{1,1}\boxtimes \Pi_{2\ellNichols}(2\lambdaNichols)$. Intuitively, for $v>3$ this Nichols algebra resembles  an exterior algebra of a $2$-dimensional vector space, which becomes a simple module in $\cC^\k$, namely the standard representation of $\sl_2$. However, the category $\cC^\k$ has non-integer Frobenius Perron dimensions and there is no fiber functor to the category of vector spaces. In the degenerate case $v=3$, where $M\otimes M=\Pi$, the Nichols algebra has the same form, but the defining relations are different, instead of an exterior algebra it resembles $\C[x]/x^3$. In the degenerate case $v=2$ the object $M$ does not make sense in $\cC^\k$, but we consider instead the Nichols algebra of $M=\integral$, which is $\NicholsOf(M)=1\oplus \integral$ and resembles $\C[x]/x^2$.

In Section \ref{subsec:QuantumBorel} we define the \emph{Borel category} $\cB^\k$ as the nonsemisimple tensor category $(\cC^\k)_{\NicholsOf}$ of modules of the Nichols algebra $\NicholsOf=\NicholsOf(M)$ inside $\cC^\k$, with the tensor product of $\cC^\k$ and $\NicholsOf$ acting via its bialgebra structure.
The following properties of the  category $(\cC^\k)_\NicholsOf$ are shared by any Nichols algebra, and are closely related to the behavior of a basic algebra: The simple objects are of the form $X_\NicholsAlgebraCounit$ for all simple objects $X\in \cC^\k$ endowed with trivial $\NicholsOf$-action. There is a distinguished indecomposable extension of $\NicholsOf$-modules   
$$0\to M_\NicholsAlgebraCounit\to E\to \1_\NicholsAlgebraCounit\to 0$$
with nontrivial action induced by $M\otimes \1\to M$. In fact, this determines all $\Ext^1(Y_\NicholsAlgebraCounit,X_\NicholsAlgebraCounit)$. Note that in our case $M$ is simple. The projective objects in $\cB^\k$ are the induced objects $\NicholsOf(M)\otimes_{\cC^\k} X$. In our case, we determine the Loewy diagram of the induced module for any $X=\VirPhi_{1, s} \boxtimes \Pi_{\ell}(\lambda)$. 
\begin{center}
\begin{tikzpicture}[scale=1]
\node (top) at (0,3) [] {$\VirPhi_{1, s} \boxtimes \Pi_{\ell}(\lambda)$};
\node (left) at (-3,0) [] {$\VirPhi_{1, s - 1} \boxtimes \Pi_{\ell+\ellNichols}(\lambda+\lambdaNichols)$};
\node (right) at (3,0) [] {$\VirPhi_{1, s + 1} \boxtimes \Pi_{\ell+\ellNichols}(\lambda+\lambdaNichols)$};
\node (bottom) at (0,-3) [] {$\VirPhi_{1, s} \boxtimes \Pi_{\ell+2\ellNichols}(\lambda+2\lambdaNichols)$};
\draw[->, thick] (top) -- (left);
\draw[->, thick] (top) -- (right);
\draw[->, thick] (left) -- (bottom);
\draw[->, thick] (right) -- (bottom);
\end{tikzpicture}
\end{center}   
In Section \ref{subsec:QuantumGroup} we define the \emph{quantum group category} $\cU^\k$, a nonsemisimple nondegenerately braided tensor category, as the relative  Drinfeld center of $\cB^\k$ or equivalently the category of Yetter-Drinfeld modules 
\[
\cU^\k:=\cZ_{\cC^\k}(\cB^\k)=\YD{\NicholsOf(M)}(\cC^\k)
\]
Under finiteness conditions there exist  right and left adjoint functors $\cB^\k\to \cU^\k$, also called coinduction and induction.
In the context of categorical Yetter-Drinfeld modules \cite{LW21} Theorem 3.13 give a  nontrivial explicit construction for the right adjoints, also in infinite cases, which we want to call below $\coLaugwitz$, similarly there is an explicit left adjoint. Composed with the previous functor $\cC^\k \to\cB^\k$ endowing an object with trivial action, these left resp. right adjoints produce modules similar to Verma and coVerma modules for quantum groups.  

In Section \ref{subsec:PartialSemisimplification} we discuss qualitatively how the category $\cU^\k$ is related to the category of representations of the quantum supergroup $\Uq(\sl_{2|1})$ at $q=e^{2\pi\i(u/v)}$ by a process we call \emph{partial semisimplification}. For usual semisimplification of quantum groups see for example \cites{AP95,BK01}. In our case, only the $q$-deformed bosonic root is semisimplified, while the fermionic roots remains nonsemisimple. Technically, this is done by parabolically decomposing the quantum group, into the quantum group $\Uq(\sl_2)$ associated to the bosonic root $\alpha_1$ and a Nichols algebra $\NicholsOf(M)$ in its category of representations, where $M$ is spanned by the the fermionic roots $\alpha_2,\alpha_{12}$ and forms an irreducible $2$-dimensional representation of  $\Uq(\sl_2)$. If we semisimplify $\Uq(\sl_2)$ we essentially recover the category $\cC^\k$ above and the Nichols algebra $\NicholsOf(M)$ above. We want to mention recent results in \cite{CEO25} in classifying all abelian tensor ideals of the category of representations of a quantum group and associating them to nilpotent orbits.  It would be interesting to relate our construction to a subset of this list, related to parabolic situations, and have a similar result for quantum supergroups.

\bigskip

In Section \ref{subsec:affine-sl2-wt-mods} we explain in detail the category $\cD^\k=\cC_\k^{\mathrm{wt}}(\sl_2)$ of weight modules of the affine vertex algebra $L_\k(\sl_2)$ at admissible level $\k$, from the first authors recent work \cites{Cr24, CMY1, NORW24, ACK}. It has a certain tensor subcategory $\cD^\k_R$. 

In Section \ref{subsec:sl2abelian} we review the abelian structure of the category. In particular there are simple modules $\mathcal{D}^\pm_{r,s}$ obtained from highest weight modules of $\sl_2$ and simple modules $\mathcal{E}^\pm_{\lambda,\Delta_{r,s}}=\mathcal{E}_{\lambda,\Delta_{r,s}}$ for generic $\lambda$ obtained from dense $\sl_2$-modules, as well as their spectral flows $\sigma^\ell(\mathcal{D}^\pm_{r,s})$ and $\sigma^\ell(\mathcal{E}^\pm_{\lambda,\Delta_{r,s}})$ with identification $\sigma^{\ell}(\mathcal{D}^-_{u-r,v-s}) \cong \sigma^{\ell-1}(\mathcal{D}^+_{r,s-1})$. For special values $\lambda=\lambda_{r,s}$ the dense modules produce an indecomposable extension
 $$0 \longrightarrow \sigma^{\ell}(\cD_{r,s}^+) \longrightarrow \sigma^{\ell+1}(\cE_{u-r,v-s-1}^-) \longrightarrow \sigma^{\ell+1}(\cD^+_{r,s+1}) \longrightarrow 0$$
and similarly for $\cE_{u,s}^+$.
The projective covers are built from $\sigma^{\ell}(\cE_{u-r,v-s}^-)$ as well as $\sigma^{\ell+1}(\cE_{u-r,v-s-1}^-)$ and have the following shape


 
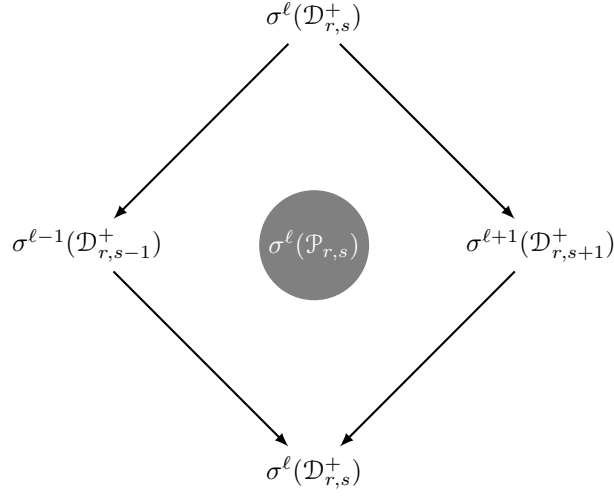
\begin{center}
\begin{tikzpicture}[scale=1]
\node (top) at (0,3) [] {$\sigma^{\ell}(\mathcal{D}^+_{r,s})$};
\node (left) at (-3,0) [] {$\sigma^{\ell-1}(\mathcal{D}^+_{r,s-1})$};
\node (right) at (3,0) [] {$\sigma^{\ell+1}(\mathcal{D}^+_{r,s+1})$};
\node (bottom) at (0,-3) [] {$\sigma^{\ell}(\mathcal{D}^+_{r,s})$};
\draw[->, thick] (top) -- (left);
\draw[->, thick] (top) -- (right);
\draw[->, thick] (left) -- (bottom);
\draw[->, thick] (right) -- (bottom);
\node (label) at (0,0) [circle, inner sep=2pt, color=white, fill=black!50!] {$\sigma^\ell(\mathcal{P}_{r,s})$};
\end{tikzpicture}
\end{center}
As a caveat, the reader should not compare this extension to the extension $E$ and the picture only accidentally agrees with the picture in the previous section.

In Section \ref{subsec:sl2monoidal} we review the computation of the tensor product in $\cD^\k$ via the Huang-Lepowsky-Zhang tensor product.

In Section \ref{subsec:sl2freefield} we review the free field realization, originally from \cite{A} 
\[   L_\k(\sl_2)\hookrightarrow \Vir_{c_{u,v}}\otimes \Pi(0)\]
The point is that we have chosen our Cartan category $\cC^\k$, the starting point for our quantum group, to coincide with the category of representations of this free field algebra.

 In Section \ref{subsec:sl2Amodules} the main work of this section is done: We consider the right-hand side  $A=\VirPhi_{1,1}\otimes \Pi_0(0)$ as commutative algebra over the left-hand side, that is  


 \begin{center}
\begin{tikzpicture}[scale=1]
\node at (-1.5,0)[] 
 {$A=$};
\node (top) at (0,1) [] {$\sigma(\mathcal{D}^+_{1,1})$};
\node (bottom) at (0,-1) [] {$\hphantom{=1}\mathcal{D}^+_{1,0}=\1$};
\draw[->, thick] (top) -- (bottom);
\end{tikzpicture}
\end{center}
 
 We compute and decompose the induced modules in $(\cD^\k)_A$. In particular, there is a distinguished extension of $A$-modules
 $$0\to \VirPhi_{1, 2} \otimes \Pi_{1}(-t/2)\to E\to  \VirPhi_{1, 1} \otimes \Pi_0(0)\to 0$$ 
 that matches the extension in Section~3 and we find the indecomposable projective $A$-modules, depict below, that match precisely the projective modules in Section~3:
  \begin{center}
\begin{tikzpicture}[scale=1]
\node (top) at (0,3) [] {$ \VirPhi_{r, s} \otimes \Pi_\ell(\lambda)$};
\node (left) at (-3,0) [] {$\VirPhi_{r, s - 1} \otimes \Pi_{\ell+1}(\lambda -t/2)$};
\node (right) at (3,0) [] {$\VirPhi_{r, s + 1} \otimes \Pi_{\ell+1}(\lambda -t/2)$};
\node (bottom) at (0,-3) [] {$\VirPhi_{r, s} \otimes \Pi_{\ell+2}(\lambda - t)$};
\draw[->, thick] (top) -- (left);
\draw[->, thick] (top) -- (right);
\draw[->, thick] (left) -- (bottom);
\draw[->, thick] (right) -- (bottom);
\end{tikzpicture}
\end{center}

We remark that from the perspective of the screening operator program, the screening operators are from $M=\VirPhi_{1, 2} \otimes \Pi_{1}(-t/2)$ \cite{A}, we have an embedding  $L_\k(\sl_2)\hookrightarrow \Vir_{c_{u,v}}\boxtimes \Pi(0)$ realized as the kernel of the screening operator.
The distinguished extension $E$ is a logarithmic deformation of $\1\oplus M$ and the decomposition of $A$ reflects the fact that $\sigma(\mathcal{D}^+_{1,1})$ is the socle of the restriction of $M$.

We also compute two particular induced $A$-modules: In Corollary \ref{cor:InductionLieDual} we recall that the induction of $\sigma(\mathcal{E}_{1,1}^+)$ is the projective cover of the the tensor unit in $\cD^\k_A$. On the quantum group side, this corresponds to the fact that inducing the Verma module of the opposite Borel produces the projective modules.  In Corollary \ref{cor:InductionA} we compute the induction of $A$ itself to be a direct sum $\1+K$ in $\cD^\k_A$, for a certain indecomposable module $K$. This result precisely coincides with our result in Example \ref{exm:ourNicholsAdjoint} for $\coLaugwitz(\1)$, which is the adjoint representation of the Nichols algebra.

\bigskip

The goal of the remaining sections is to prove the main result of this article

\begin{theorem}\label{thm:main}
We have an equivalence of braided tensor categories
\[\cU^\k\cong \cD^\k\]
\end{theorem}

The proof consists of two major steps and compares the free field realization of $L_\k(\sl_2)$, viewed as a commutative algebra $A\in\cD^\k$, to the extension of $\cC^\k$ by the Nichols algebra in the other direction

\begin{equation} \label{eq_proofStrategy}
\def\shiftL{1mm}
\def\shiftR{.5mm}
\begin{tikzcd}[row sep=10ex, column sep=15ex]
\cC_k^{\mathrm{wt}}(\sl_2)=
&
[-2.1cm]
\cD^\k
\arrow[r,dashed, no head, "\text{Step II}"]
\arrow[d,shift left=-\shiftL, swap, "A\otimes(-)"]
&
\cU^\k
\arrow[d,shift right=1mm, swap, "\forget_c"]
&
[-1.9cm]
=\cZ_{\cC^\k}(\cB^\k)\;=\;\YD{\NicholsOf(M)}(\cC^\k)
\\
&
(\cD^\k)_A
\arrow[r,dashed, no head, "\text{Step I}"]
\arrow[u,shift left=-\shiftR, swap,"\forget_A"]
\arrow[d,shift right=0, swap, "\split"]
&
\cB^\k
\arrow[u,shift right=\shiftR, swap, "\coLaugwitz"]
\arrow[d,shift right=0, swap, "\forget_\NicholsOf"]
&
=(\cC^\k)_\NicholsOf
\hspace{3.0cm}
\\
&
(\cD^\k)^\loc_A
\arrow[u,shift right=0, swap, "\mathrm{embed}"]
\arrow[r,equal]
&
\cC^\k
\arrow[u,shift right=0, swap, "\triv"]
&
\end{tikzcd}
\end{equation}
(in this diagram, $\forget_A,\forget_c,\forget_\NicholsOf$ are forgetful functors,  arrows next to each other are adjunctions, while the two sided arrows are one-sided inverse, the  corresponding adjunctions are discussed in Section~\ref{subsec:QuantumBorel})

\bigskip

In Section \ref{sec_equivalenceBorel} we carry out Step I: We prove that the tensor category $(\catD^\k)_A$ of modules over the commutative algebra $A$ in $\catD^\k$ is equivalent to the Borel category $\cB^\k$ that consists of modules over the Nichols algebra in $\cC^\k$, which is by definition identical to the subcategory of local modules $(\catD^\k)_A^\loc$. This reconstruction result relies crucially on the fact that in our case all \emph{simple} $A$-modules are already local. As a consequence the inclusion $\cB^\k\hookleftarrow \cC^\k$ admits a monoidal splitting functor $\catB^\k\twoheadrightarrow \catC^\k$ sending a module to the direct sum of its composition factors, see \cite{CLR23} Section 4. The interesting new aspect in $\cB^\k$ are indecomposable extensions of these simple modules, while $\cC^\k$ is semisimple.

From the vertex algebra sides, the tensor category $\cB^\k$ should be interpreted as twisted vertex algebra modules of the free field vertex algebra represented by $A$, which are honest vertex algebra modules over the vertex subalgebra $L_\k(\sl_2)$. Note that the splitting behavior is in strong contrast to usual orbifold theory for finite groups, where new twisted simple modules arise, and should be interpreted as our screenings being actions of a nilpotent algebra, namely the Nichols algebra. 

In Section \ref{subsec:BorelModulecat} and \ref{subsec:Mombelli} we review and slightly generalize our reconstruction result of \cite{LM25} to an infinite setting: It shows that if $\catC\subset \catB$ is a central subcategory of a tensor category, which admits  a monoidal splitting functor, then $\catB$ is the category of modules over some Hopf algebra $\Nichols$ in $\catC$. This is a relative version of Tannaka-Krein reconstruction and specializes to the Radford projection theorem. Applied to our situation, we hence know there exists a Hopf algebra $\Nichols$ that describes how arbitrary $A$-modules are can be obtained as indecomposable extensions of local $A$-modules, in a sense reversing the information loss of the splitting functor. The main technical issue in the infinite case is that we have to show that $\catB$ as an abelian module category over $\catC$ can be realized as modules over some algebra, which is essentially an existence-of-adjoint result and for finite categories a standard result in \cite{DSPS19}. In Section~\ref{subsec:BorelModulecat} we review this proof and give a list of explicit local finiteness properties under which the result still holds, and which clearly hold in situations like ours.

In Section \ref{subsec:NicholsArguments}
want to determine the Hopf algebra $\Nichols$. For this, we make a detour to the Andruskiewitsch-Schneider program for classifying pointed Hopf algebras \cite{AS10}, see the more recent survey \cite{AG19}, note that we switch to the dual Hopf algebra $\Nichols^*$. An essential assertion in their situation is the following: 
Any Hopf algebra $\Nichols^*$ comes with a coradical filtration, the first two terms $\Nichols^*_0$ and $\Nichols^*_1$ correspond, categorically speaking, to the simple comodules and to indecomposable extensions between simple comodules. The associated graded coalgebra $\gr(\Nichols^*)$ contains by its universal property the Nichols algebra of $M^*=\Nichols^*_1/\Nichols^*_0$. Hence in our situation $\Nichols^*_0=\1$ and the knowledge of the indecomposable extension $\1\to E^*\to M^*$ fixes the Nichols algebra $\NicholsOf(M^*)$. Two questions remain: Is  $\gr(\Nichols^*)$ possibly larger then $\NicholsOf(M^*)$, in that it contains additional generators in coradical degree $>1$, and is $\Nichols^*$ possibly a non-graded deformation of  $\gr(\Nichols^*)$. A point in the Andruskiewitsch-Schneider program is that these questions can be answered purely on the Nichols algebra side. If both possibilities can be ruled out completely, then we call the Nichols algebra \emph{robust}. A short calculation shows that our new Nichols algebra $\NicholsOf(M)$ in Section \ref{subsec:NicholsAlgebra} in $\cC^\k$ is robust, as well as its dual $\NicholsOf(M)^*=\NicholsOf(M^*)$, which is of the same form for dual parameters $(-\ellNichols,-\lambdaNichols)$.

In Section \ref{subsec:BorelEquivalence} we combine the previous results in this section with the concrete computations in Section \ref{subsec:sl2Amodules} to achieve our first main result, called above Step I:\\

\begin{introTheorem}[\ref{thm:BorelEquivalence}]
The category $(\cD^\k)_A$ in Section \ref{subsec:sl2Amodules} is equivalent as tensor category to the category of modules over $\NicholsOf(M)$ in $\cC^\k$ in Section \ref{subsec:NicholsAlgebra} for $$M=\VirPhi_{1,2}\boxtimes \Pi_{1}(-t/2),\; v\geq 3,\qquad 
M=\VirPhi_{1,1}\boxtimes \Pi_{2}(-t),\;v=2.$$ 
Under the correspondence, the splitting functor is the functor forgetting the Nichols algebra action, and the embedding of local modules is the functor endowing an object with trivial Nichols algebra action.
\end{introTheorem}
\medskip
\begin{proof}~
\begin{itemize}
\item By Section \ref{subsec:BorelEquivalence} all simple objects in $(\cD^\k)_A$ are in $(\cD^\k)_A^\loc$. Hence by \cite{CLR23} Section 4 we have a monoidal \emph{splitting functor} \[\Split:\;(\catD^\k)_A\to (\catD^\k)_A^\loc\]
which is left-inverse to the embedding. 
\item By Section \ref{subsec:BorelEquivalence} the category $(\cD^\k)_A$ fulfills the weaker finiteness assertions in Corollary \ref{cor:ModuleCatFiniteFinal}. Hence by Theorem \ref{thm:MombelliInfinite} there exists a Hopf algebra $\Nichols$ in $\cC^\k=(\catD^\k)_A^\loc$ realizing the tensor category in question \[(\cD^\k)_A\cong (\cC^\k)_\Nichols\]
\item By Section \ref{subsec:BorelEquivalence} $\Ext^1(1,X)$ is one-dimensional for $
X \cong \VirPhi_{1,2} \otimes \Pi_1(-t/2)$ if $v\geq 3$ resp. for $X \cong \VirPhi_{1,1} \otimes \Pi_2(-t)$ if $v=2$ and zero else. Hence the Nichols algebra associated to $\Nichols$ is the asserted Nichols algebra $\NicholsOf(M)$.
\item By Section \ref{subsec:NicholsArguments} this Nichols algebra is robust, hence $\Nichols\cong \NicholsOf(M)$.
\end{itemize}
\end{proof}

Note that in Section \ref{subsec:sl2freefield} we have computed much more of $(\cD^\k)_A$ then necessary and it is readily visible that the categories are indeed equivalent, as abelian categories as well as their fusion rule.

\bigskip

The remainder of the article is devoted to the proof of Step II. To this end, we prove an extended version of our reconstruction result in \cite{CLR23} Section 3: Let $\cD$ be a braided tensor category and $A$ a commutative algebra therein. By \cite{CLR23} Theorem 3.7 the induction functor lifts to a braided tensor functor we call Schauenburg functor
  \[\Schauenburg:\;
    \cD \longrightarrow \cZ_{\cD_A^\loc}(\cD_A)
    \]
Under the additional Assumptions \ref{assumption_Fullyfaithful}, which include rigidity and a nondegenerate braiding, it is proven in \cite{CLR23} Lemma 3.4 that this functor is exact and fully faithful. Our goal is to prove that this is in fact an equivalence of categories in our case. Therefore the goal of the remaining section is to prove and apply\\

\begin{introTheorem}[\ref{thm:Schauenburg}]
Suppose that the Schauenburg functor is exact and fully faithful, which is true for example under Assumption \ref{assumption_Fullyfaithful}. Suppose that $\cB$ is rigid and locally finite. Suppose that 
for all simple objects in $X$ in $\cB$ we have 
$$\Ind_A(X)\cong \forget_c(\coLaugwitz(X)$$
Then the Schauenburg functor is an equivalence of braided tensor categories. \\
\end{introTheorem}

As discussed more explicitly in Example \ref{exm:SchauenburgSplit}, if $\catD_A\cong (\catD_A^\loc)_\NicholsOf$ for some Hopf algebra $\NicholsOf$, the question is if 
$$\Ind_A(X)\cong \coLaugwitz(X)\cong {_{\text{ad}}}(\NicholsOf\otimes X)$$

For finite tensor categories and $\cC$ the category of local modules we have proven this result without the additional comparison assumption, but   using Frobenius Perron dimensions \cite[Corollary 3.8.]{CLR23}, and we would suspect that this is also true in (mildly) infinite cases. In our case, knowing the fusion rules, we can directly compute both sides at the level of objects and verify the additional assumption.

\bigskip

The idea for the proof of Theorem \ref{thm:Schauenburg} is a more systematic monadic version of more algebraic arguments in \cite{CLR23} Section 3.3. \\

In Section \ref{sec:adjunction} we recall basic facts about adjunctions and (co)monads, which are (co)monoids in the category of endofunctors of a category $\cC$. From the perspective of this article, (co)monads are important for the following reason:
\begin{itemize}
\item A monad $\monT$ has an associated category of representations $\cC_\monT$, each consisting of an object $X\in\cC$ together with a morphism $\monT(X)\to X$. Any algebra $A$ in the category gives rise to a monad $\monT(X)=A\otimes X$. Similarly for comonads and comodules.
\item A monad is attached to every adjunction. There are general criteria (e.g. Becks monadicity criterion) when the monad fully describes one category as category of modules of the monad in the other category. 
\item Adjunctions and respective (co)monads compatible with monoidal structure either resemble commutative algebras or Hopf algebras, depending on the details, see the table in Example \ref{exm_fourmonoidalmonads}. 
\item Now in view of applications, there are comonads of the second type representing $\catD$ over $\catD_A$ (so in the opposite direction of the commutative algebra~$A$) as well as the relative Drinfeld center, see Example \ref{exm_centralcomonad}.
\end{itemize}

In Section \ref{sec_troi} we recall and develop some basic tools (some of which we have not found in literature, but we would expect they are known to experts) in the general situation of two (monadic) adjunctions related by a functor $\funE$

\begin{equation*}
\xymatrix{
\catD 
\ar[rr]^{\funE}
\ar[rd]_{\funD}
&&
\catDD
\ar[ld]^{\funDD}
\\
&
\catB
&
\\
} 
\end{equation*}

In Section \ref{subsec:troi} we give the basic setup and in particular a comparison by the following natural transformation derived from the counit of the adjunction
\[\epsilon_{\funDD^\ra(X)}:\;\funE(\funD^\ra(X))=\funE(\funE^\ra(\funDD^\ra(X)))
\longrightarrow \funDD^\ra(X)\]

In Section \ref{subsec:monadRestriction} we analyze $\funE$ in the case where $\funD,\funDD$ are monadic adjunctions with associated monads $\monT,\monR$. We prove standard algebra assertions in this setup:
\begin{itemize}
    \item The functor $\funE:\catB_\monT\to \catB_\monR$ comes from a restriction-of-scalars with respect to a monad morphism $\fpullback:\monR\to \monT$.
    \item It has a right adjoint given explicitly by extension-of-scalars. In particular we can compute the effect of $\funE^\ra$ on induced modules $\funDD^\ra(X)$.
    \item The comparison morphism above becomes $\fpullback$ after applying $\funDD$.
    \item $\funE$ is an equivalence of categories iff $\fpullback_X$ is an isomorphism for all $X$. 
\end{itemize}

In Section \ref{subsec:Nakayama} we prove in our setup an interesting algebraic fact: Whenever a restriction-of-scalars functor is fully faithful and the image of the underlying first ring over the second ring is finitely generated, then $\fpullback$ is surjective. This appears in classical algebra topics like epimorphism of rings, and it is not true without finiteness assumptions, as the inclusion $\Z\hookrightarrow\mathbb{Q}$ shows as counterexample. In categorical language, this shows that in sufficiently good cases the fully faithful functor is epireflective resp. monocoreflective. 

\bigskip

In Section \ref{sec:QuantumGroupEquivalence} we prove after these preparations the category equivalence in Step~II.

In Section \ref{subsec:Schauenburg} we prove Theorem \ref{thm:Schauenburg} by collection the results above.

In Section \ref{subsec:QuantumGroupEquivalence} we collect all necessary knowledge about our concrete category $\cD^\k$ for applying this Theorem to our situation, such as rigidity, nondegenerateness, comparison functor evaluated on $A$. This concludes the proof of Theorem \ref{thm:main}.\\

\bigskip

\subsection{Variations}\label{sec:variations}

The construction has a few variations where we replace $\cC^\k, \cD^\k, \cB^\k$ and $\cU^\k$ by $\cC^\k_X, \cD^\k_X, \cB^\k_X$ and $\cU^\k_X$ with $X \in \{R, S, E, P, EP\}$ as follows. Firstly
$\cU^\k_X:=\cZ_{\cC^\k_X}(\cB^\k_X)$ throughout. 
\begin{enumerate}
    \item[(R)] $R$ stands for the \emph{right} category, that is for those modules where the left label is trivial.  We set $\cC^\k_R = \cC^R(u,v)\boxtimes \Pi(0)$. $\cD^\k_R$ is then the subcategory whose modules have as simple composition factors the  $\sigma^\ell(\mathcal{E}_{\lambda,\Delta_{1,s}})$ and $\sigma^\ell(\mathcal{D}^\pm_{1,s})$. This is a tensor subcategory of $\cD^\k$ \cite[Corollary 8.5]{Cr24}. $\cB^\k$ is an object in $\cC^\k_R$ and we define $\cB^\k_R$ as $\cB^\k$ viewed as an object in $\cC^\k_R$. In this case $\cU^\k_R$ is the partially semisimplified quantum group of $\mathfrak{sl}_{2|1}$.
    \item[(S)] $S$ stands for $N=2$ super conformal. The $N=2$ super Virasoro algebra $N_\k$ at central charge $\frac{3\k}{\k+2}$ is related to $L_\k(\sl_2)$ via the Kazama-Suzuki construction \cite{KS89} which for $L_\k(\sl_2)$ and $N_\k$ is discussed in great detail in \cites{CR3, NORW24}. Modernly this is treated via convolution operations and we explain this in detail in section \ref{sec:convolution}. The categories are then replaced by their images under the corresponding convolution operation. 

    The coset $\text{Com}(L_\ell(\sl_2), L_\ell(\sl_{2|1}))$ for $\ell$ related to $\k$ via the Feigin-Frenkel type relation $(\ell+1)(\k+1)=1$ can be described by a different, but similar, convolution operation.
    \item[(E)] $E$ stands for extension an alternative label is $U$ for uprolling. 
    In this case one considers the simple current extensions
    \[
    E_\k = \bigoplus_{\ell \in \mathbb Z} \sigma^{2\ell v}(L_\k(\sl_2))
    \]
    and 
    \[
    S_\k = \bigoplus_{\ell \in \mathbb Z} \Vir_{c_{u,v}}\otimes \Pi_{2\ell v}(0).
    \]
    Then $\cD^\k_E$ is the category of weight modules of $E_\k$ and $\cC^\k$ is the one of $S_\k$. Note that $S_\k$ is nothing but a lattice VOA times $\Vir_{c_{u,v}}$ and the categories $\cD^\k_E, \cC^\k_E$ can be described via the theory of simple current extensions, in particular there is an induction functor and the Nichols algebra $\cB^\k_E$ is correspondingly the induction of the Nichols algebra $\cB^\k$. The quantum group $\cU^\k_E$ could be called the uprolled version of $\cU^\k$ and this uprolling procedure and how it interplays with the relative center/Yetter-Drinfeld module construction is section 10 of \cite{CLR23}. The uprolling of quantum groups is \cite{CR4}.
    \item[(P)] $P$ stands for parafermionic coset. Let $\pi^h \subset L_\k(\sl_2)$ be the Heisenberg vertex subalgebra generated by the field $h(z)$ corresponding to the Cartan subalgebra of $\sl_2$. The parafermion coset is then $P_\k=\text{Com}(\pi^h, L_\k(\sl_2))$ and the coset $\text{Com}(\pi^h, \Vir_{c_{u,v}}\otimes \Pi(0)) \cong \Vir_{c_{u,v}}\otimes \pi$ is just the Virasoro minimal model $\Vir_{c_{u,v}}$ times a non-degenerate rank one Heisenberg VOA $\pi$. Accordingly $\cD^\k_P$ is the category of weight modules of $P_\k$ and $\cC^\k_P$ the one of $\Vir_{c_{u,v}}\otimes \pi$. The Nichols algebra $\cB^\k_P$ is the multiplicity space of $\pi^h$ in $\cB^\k$. This coset is described in all detail in \cite{ACR}. 
    \item[(EP)] Finally one can combine the two previous points, that is take the parafermionic coset of $E_\k$, that is $\text{Com}(\pi^h, E_\k)$. This corresponds to an uprolling procedure of the previous point. The corresponding VOAs are also discussed in detail in \cite{ACR}. They conjecturally (with few cases proven) give rise to $C_2$-cofinite VOAs and so $\cU^\k_{EP}$ should be a finite tensor category. 
\end{enumerate}
The same proof as the one of our main Theorem applies (respectively for uprolling it follows directly from section 10 of \cite{CLR23}) to give 
\[
\cU^\k_X \cong \cD^\k_X, \qquad \text{for} \ X \in \{R, S, E, P, EP\}
\]
as braided tensor categories. 

\bigskip

\subsection{Outlook for vertex algebras and quantum supergroups}

Let $\g$ be a simple Lie algebra and $\k$ a complex number, then there is an associated vertex algebra, $V^\k(\g)$, the affine vertex algebra of $\g$ at level $\k$. For $f \in \g$ a nilpotent element one can then construct another vertex algebra via quantum Hamiltonian reduction, $W^\k(\g, f)$, the $W$-algebra of $\g$ associated to $f$ at level $\k$ \cite{KacWak}. 
It only depends on the nilpotent orbit of $f$. 
By the Jacobson–Morozov theorem  $f$ belongs to an $\mathfrak{sl}_2$-triple $\{e, h, f\}$. Denote by $\g^\sharp$ the subalgebra of $\g$ that is annihilated by the action of this $\mathfrak{sl}_2$-triple, $\g^\sharp = \text{Ann}_{\mathfrak{sl}_2}(\g)$. Then there is a level $\k^\sharp$ depending on $\k, \g, f$, such that $V^{k^\sharp}(\g^\sharp)$ acts on $W^\k(\g, f)$. Hence every $W^\k(\g, f)$-module $M$ is bi-graded by $\g^\sharp$-weight and by conformal weight, i.e. let $\mathfrak{h}^\sharp$ be the Cartan subalgebra, then
\[
M = \bigoplus_{\substack{\lambda \in (\mathfrak{h}^{\sharp})^* \\ \Delta \in \mathbb C}} M_{\lambda, \Delta}
\]
with $M_{\lambda, \Delta}$ the generalized weight spaces of weight $(\lambda, \Delta)$. 
The category $W^\k(\g, f)$-$\text{wtmod}$ is the category of finitely generated $W^\k(\g, f)$, s.t. $\mathfrak{h}^\sharp$ acts semisimply and s.t. $\text{dim} M_{\lambda, \Delta} < \infty$ for all $\lambda, \Delta$. 
Its  subcategory $W^\k(\g, f)$-$\text{wtmod}^{\text{Ord}}$ of ordinary modules consists of those objects that satisfy that there exists a real number $N$, s.t. 
\[
M_\Delta = 0 \qquad \text{for} \ \Delta < N, \qquad \text{dim} M_\Delta < \infty, \qquad \text{with} \qquad M_\Delta = \bigoplus_{\lambda \in (\mathfrak{h}^{\sharp})^* } M_{\lambda, \Delta}.
\]
In the case of $f=0$ this category is called the Kazhdan-Lusztig category denoted by $V^\k(\g)\text{-}\mathrm{wtmod}^{\mathrm{KL}}$ and named after the celebrated result of Kazhdan and Lusztig
\begin{theorem} \cites{KL93a, KL93b, KL93c}
Let $\g$ be a simple Lie algebra with dual Coxeter number $h^\vee$. For $k \notin \mathbb Q_{>-h^\vee}$  as braided tensor categories
\[
V^\k(\g)\text{-}\mathrm{wtmod}^{\mathrm{KL}} \cong \Uq(\g)\text{-}\mathrm{mod}
\]
with $q = \mathrm{exp}\left({\frac{\pi i}{r^\vee(k+h^\vee)}}\right)$ and $r^\vee$ the lacing number of $\g$.
\end{theorem}
For $W$-algebras only very recently a related statement was established
\begin{theorem} \cites{CGN}
Let $\g$ be a simple and simply laced Lie algebra with dual Coxeter number $h^\vee$. For $k \notin \mathbb Q$ there is a fully faithful braided tensor functor 
\[
H_f: V^\k(\g)\text{-}\mathrm{wtmod}^{\mathrm{KL}} \rightarrow W^\k(\g, f)\text{-}\mathrm{wtmod}^{\mathrm{Ord}}
\]
such that $\mathrm{im} H_f$ is a thick subcategory of 
$W^\k(\g, f)\text{-}\mathrm{wtmod}^{\mathrm{Ord}}$.
\end{theorem}
Recall Feigin-Frenkel duality for principal $W$-algebras \cite{FF91}, $W^\k(\g, f_{\text{prin}}) \cong W^\ell({}^L\g, f_{\text{prin}})$ if the levels satisfy the duality relation
$r^\vee(k+h^\vee)(\ell + {}^Lh^\vee) =1$.

In order to state the correspondence to quantum groups we need to twist the Deligne product in the following sense:
Let $L$ be an even integral lattice, that is there is a symmetric bilinear form $B:  L \times L \rightarrow 2\mathbb Z$, 
with dual lattice $L'$ and let $\Gamma = L'/L$. 
Consider the lattice $N = L \oplus L$ with bilinear form $B_N: N \times N \rightarrow \mathbb Z, ((\lambda, \mu),(\lambda', \mu')) \mapsto B(\lambda, \mu') + B(\mu, \lambda')$. Then $N' = L' \oplus L'$.
Let $Q_N$ be the corresponding quadratic form. 
Assume that categories $\cC, \cD$ are both graded by the abelian group $\Gamma$, then 
\[
\mathrm{Vect}_{N'/N}^{Q_N} \boxtimes (\cD \boxtimes \cC)
\]
has a diagonal action of $N'/N$ and we set the $L$-twisted Deligne product
\[
\cD \boxtimes_L \cC := \left(\mathrm{Vect}_{N'/N}^{Q_N} \boxtimes (\cD \boxtimes \cC)\right)^{N'/N}.
\]
With this notion we obtain
\begin{theorem} \cites{CGN}
Let $\g$ be a simple and simply laced Lie algebra with dual Coxeter number $h^\vee$. For $k \notin \mathbb Q$ and $\ell$ satisfying $(k+h^\vee)(\ell+h^\vee)=1$ and $Q$ the root lattice of $\g$ there is a fully faithful braided tensor functor 
\[
H_f: V^\k(\g)\text{-}\mathrm{wtmod}^{\mathrm{KL}} \boxtimes_Q V^\ell(\g)\text{-}\mathrm{wtmod}^{\mathrm{KL}} \rightarrow W^\k(\g, f_{\text{prin}})\text{-}\mathrm{wtmod}^{\mathrm{Ord}}
\]
such that $\mathrm{im} H_f$ is a thick subcategory of 
$W^\k(\g, f_{\text{prin}})\text{-}\mathrm{wtmod}^{\mathrm{Ord}}$.
\end{theorem}
The Feigin-Frenkel duality generalizes to many non-principal $W$-algebras \cites{CL1, CL2, CKLSS} with two important differences: the dual algebra is usually a $W$-superalgebra and this is not an isomorphism of vertex superalgebras, but only an isomorphism of coset subalgebras. Moreover in the cases of so-called hook type $W$-superalgebras \cites{CL1, CL2} there is a convolution operation mapping one $W$-algebra to its dual $W$-superalgebra and there is a dual convolution operation mapping the $W$-superalgebra back to its dual $W$-algebra \cite{CLNS}. 
For this we consider a Lie (super)algebra $\g$ with a simple Lie subalgebra $\mathfrak a$. Then one considers the nilpotent element $f_{\mathfrak a}$ corresponding to the principal nilpotent element in the subalgebra $\mathfrak a$. The $W$-superalgebra $W^\k(\g, f_{\mathfrak a})$ then has an affine vertex subalgebra $V^{\k_{\mathfrak b}}(\mathfrak b)$ where $\mathfrak b$ is the Lie subalgebra of $\g$ that commutes with the $\mathfrak{sl}_2$-triple corresponding to $f_{\mathfrak a}$. The hook-type $W$-superalgebras are all of this type and the data is listed in Table 
\ref{tab:hook-type}.

\begin{table}[h]
\centering
\begin{tabular}{|c|c|c|c|}
\hline 
 $\mathfrak{g}$ & $\mathfrak{a}$ & $\mathfrak{b}$ & $k_{\mathfrak{b}}$ \\ \hline 
 $\mathfrak{sl}_{n+m}$            & $\mathfrak{sl}_n$      & $\mathfrak{gl}_m$        & $k+n-1$ \\ 
 $\mathfrak{so}_{2(n+m+1)}$        & $\mathfrak{so}_{2n+1}$ & $\mathfrak{so}_{2m+1}$   & $k+2n$ \\
 $\mathfrak{sp}_{2(n+m)}$          & $\mathfrak{sp}_{2n}$   & $\mathfrak{sp}_{2m}$     & $k+n-\tfrac{1}{2}$ \\
 $\mathfrak{so}_{2(n+m)+1}$        & $\mathfrak{so}_{2n+1}$ & $\mathfrak{so}_{2m}$     & $k+2n$ \\
 $\mathfrak{osp}_{1|2(n+m)}$       & $\mathfrak{sp}_{2n}$   & $\mathfrak{osp}_{1|2m}$  & $k+n-\tfrac{1}{2}$ \\
 $\mathfrak{sl}_{n+m|m}$           & $\mathfrak{sl}_{n+m}$          & $\mathfrak{gl}_m$      & $-(k+n+m)+1$ \\
 $\mathfrak{osp}_{2m+1|2(n+m)}$    & $\mathfrak{sp}_{2(n+m)}$       & $\mathfrak{so}_{2m+1}$ & $-2(k+n+m)+1$ \\
 $\mathfrak{osp}_{2(n+m)+1|2m}$    & $\mathfrak{so}_{2(n+m)+1}$     & $\mathfrak{sp}_{2m}$   & $-(\tfrac{1}{2}k+n+m)$ \\
 $\mathfrak{osp}_{2m|2(n+m)}$      & $\mathfrak{sp}_{2(n+m)}$       & $\mathfrak{so}_{2m}$   & $-2(k+n+m)+1$ \\
 $\mathfrak{osp}_{2(n+m)+2|2m}$    & $\mathfrak{so}_{2(n+m)+1}$     & $\mathfrak{osp}_{1|2m}$ & $-(\tfrac{1}{2}k+n+m)$ \\
\hline 
\end{tabular}
\caption{Hook-type $\mathcal{W}$-superalgebras}
\label{tab:hook-type}
\end{table}

The kernel VOAs are defined as follows \cites{CG, CGL}. They are of the form
\[
A^n[\g, k] = \bigoplus_{\lambda \in R_\g} V^\k_{\g}(\lambda) \otimes V^\ell_{ \tilde \g}(\lambda^*)
\]
with $V^\k_{\g}(\lambda)$ the Weyl module of highest-weight $\lambda$ of $\g$ at level $\k$, $\lambda^*$ the highest-weight of the dual representation and $R_\g$ the subset of dominant weights that lie in the lattice  $Q + \mathbb Z \omega_1 + \mathbb Z \omega_1^*$ with $\omega_1$ the first fundamental weight, i.e. the highest-weight of the standard representation of $\g$. $\tilde \g$ is a Lie superalgebra that admits a natural identification of $R_{\tilde g}$ with $R_\g$. The upper index indicates the relation of levels $k$ and $\ell$  which depend on $\g$ and $\tilde \g$ as listed in Table \ref{tab:kernel-levels}. 
\begin{table}[h]
\centering
\begin{tabular}{|c|c|c|c|c|} \hline 
$\mathfrak{g}$  & $\tilde{\mathfrak{g}}$  & Level relation \\ \hline 

$\mathfrak{gl}_m$ & $\mathfrak{gl}_m$
  & $\dfrac{1}{k+h^\vee_\g}+\dfrac{1}{\ell+h^\vee_{\tilde\g}}=n$
   \\[2ex]
$\mathfrak{so}_{2m+1}$ & $\mathfrak{osp}_{1|2m}$
  & $\dfrac{1}{k+h^\vee_\g}+\dfrac{1}{2(\ell+h^\vee_{\tilde\g})}=n$
   \\[2ex]
$\mathfrak{sp}_{2m}$ & $\mathfrak{sp}_{2m}$
  & $\dfrac{1}{k+h^\vee_\g}+\dfrac{1}{\ell+h^\vee_{\tilde\g}}=2n$
   \\[2ex]
$\mathfrak{so}_{2m}$ & $\mathfrak{so}_{2m}$
  & $\dfrac{1}{k+h^\vee_\g}+\dfrac{1}{\ell+h^\vee_{\tilde\g}}=n$
  \\[2ex]
$\mathfrak{osp}_{1|2m}$ & $\mathfrak{so}_{2m+1}$
  & $\dfrac{1}{2(k+h^\vee_\g)}+\dfrac{1}{\ell+h^\vee_{\tilde\g}}=n$
   \\
\hline 
\end{tabular}
\caption{Level relations for each kernel VOA $A^n[\mathfrak g,k]$.}
\label{tab:kernel-levels}
\end{table}
If $n \in \mathbb Z$ and $k \notin \mathbb Q$, then $A^n[\g, k]$ is  a simple vertex operator superalgebra in types $\mathfrak{gl}_m, \mathfrak{so}_{2m}, \mathfrak{sp}_{2n}$ \cite{Mor} and if $n$ is odd also in types $\mathfrak{so}_{2m+1}, \mathfrak{osp}_{1|2n}$ \cite{CMY-osp}.

\begin{theorem}{\textup{\cite{CLNS}}}
For $k \notin \mathbb Q$,  $\tilde k_{\mathfrak b_i} = - k_{\mathfrak b_i} -2h^\vee_{\mathfrak b_i}$, $i \in \{ 1, 2\}$,  $(\g_1, \mathfrak a_1, \mathfrak b_1, \g_2, \mathfrak a_2, \mathfrak b_2, r)$ as in Table \ref{tab:duality data} and $\ell$ satisfying the Feigin-Frenkel type relation
\[
r(k+ h^\vee_{\g_1})(\ell + h^\vee_{\g_2}) =1
\] there are the following isomorphisms of vertex superalgebras
\begin{align*}
     W^\ell(\g_2, f_{\mathfrak a_2}) \cong H^{\frac{\infty}{2}+\bullet}_{\mathrm{rel}}(\mathfrak b_1, W^\k(\mathfrak{g}_1, f_{\mathfrak a_1})\otimes\mathfrak A^1[\mathfrak b_1, \tilde k_{\mathfrak b_1}]) \\
      W^\k(\g_1, f_{\mathfrak a_1}) \cong H^{\frac{\infty}{2}+\bullet}_{\mathrm{rel}}(\mathfrak b_2, W^\ell(\mathfrak{g}_2, f_{\mathfrak a_2})\otimes\mathfrak A^{-1}[\mathfrak b_2, \tilde k_{\mathfrak b_2}])
\end{align*}
with the relative semi-infinite Lie algebra cohomology $ H^{\frac{\infty}{2}+\bullet}_{\mathrm{rel}}$.

\begin{table}[h]
\centering
\begin{tabular}{|c|c|c|c|c|c|c|}
\hline 
 $\mathfrak{g}_1$ & $\mathfrak{a}_1$ & $\mathfrak{b}_1$ & $\mathfrak{g}_2$ & $\mathfrak{a}_2$ & $\mathfrak{b}_2$ & r \\ \hline 
 $\mathfrak{sl}_{n+m}$            & $\mathfrak{sl}_n$      & $\mathfrak{gl}_m$        &  $\mathfrak{sl}_{n+m|m}$           & $\mathfrak{sl}_{n+m}$          & $\mathfrak{gl}_m$ & 1 \\ 
$\mathfrak{so}_{2(n+m+1)}$        & $\mathfrak{so}_{2n+1}$ & $\mathfrak{so}_{2m+1}$   & $\mathfrak{osp}_{2(n+m)+2|2m}$    & $\mathfrak{so}_{2(n+m)+1}$     & $\mathfrak{osp}_{1|2m}$  & 1 \\
$\mathfrak{sp}_{2(n+m)}$          & $\mathfrak{sp}_{2n}$   & $\mathfrak{sp}_{2m}$     & $\mathfrak{osp}_{2(n+m)+1|2m}$    & $\mathfrak{so}_{2(n+m)+1}$     & $\mathfrak{sp}_{2m}$  & 2 \\
 $\mathfrak{so}_{2(n+m)+1}$        & $\mathfrak{so}_{2n+1}$ & $\mathfrak{so}_{2m}$     & $\mathfrak{osp}_{2m|2(n+m)}$      & $\mathfrak{sp}_{2(n+m)}$       & $\mathfrak{so}_{2m}$  & 2\\
$\mathfrak{osp}_{1|2(n+m)}$       & $\mathfrak{sp}_{2n}$   & $\mathfrak{osp}_{1|2m}$  & $\mathfrak{osp}_{2m+1|2(n+m)}$    & $\mathfrak{sp}_{2(n+m)}$       & $\mathfrak{so}_{2m+1}$ & 4\\
\hline 
\end{tabular}
\caption{Duality data}
\label{tab:duality data}
\end{table}

\end{theorem}

We expect that the Kazhdan-Lusztig correspondence admits a generalization to quantum supergroups of basic classical Lie superalgebras and that the above Feigin-Frenkel type correspondences gives rise to correspondences between $V^\ell(\g_2)\text{-}\mathrm{wtmod}^{\mathrm{KL}}$ and $W^\k(\g_1, f_{\mathfrak a_1})\text{-}\mathrm{wtmod}^{\mathfrak{a}_1[t]\text{-loc.fin.}}$. Here the index $\mathfrak{a}_1[t]\text{-loc.fin.}$  
indicates that $\mathfrak{a}_1[t]$ acts locally finitely.

The purpose of our program is to develop a theory that allows to prove such correspondences. The biggest problem is that in general not much is known about the representation theory of affine vertex superalgebras and $W$-superalgebras in general. For example, the only affine vertex superalgebra that is well understood is $V^\k(\mathfrak{gl}_{1|1})$ \cite{CMYgl11} and indeed we were able to establish the equivalence of  $\smash{V^\k(\mathfrak{gl}_{1|1})\text{-}\mathrm{wtmod}^{\mathrm{KL}}}$ with $\smash{\Uq^H(\mathfrak{gl}_{1|1})\text{-}\mathrm{mod}}$ \cite{CLR23} where the upper index $H$ indicates that it is an unrolled version of the quantum supergroup. Here $q = \text{exp}\left( \frac{\pi i}{k}\right)$. 

We plan to use our techniques to extend this result to any basic classical Lie superalgebra of type I. 

In the present article we consider the simple affine vertex algebra of $\sl_2$ at any admissible level. Admissible levels are in particular rational levels and hence corresponding quantum (super)groups are at roots of unity. The idea is that a similar correspondence as at generic levels holds, but one has to replace the quantum (super)groups by what we call partial semisimplifications.

\newcommand{\SimonRem}[1]{}

\subsection{Another proof strategy: Abelian modeling}

We also give an alternative proof strategy that is more {robust} in the sense that several abstract arguments are replaced by concrete calculations with representing algebras if the abelian categories $\cD^\k$ and $\cD^\k_A$ are known explicitly enough. Hence it can also be applied in related cases where some technical preliminaries may not hold.

\bigskip

Suppose we have algebras $D,B,C$ in the category of finite-dimensional vector spaces and $\Nichols$ an algebra in the category $\cC^\k$, such that 
\begin{itemize}
\item $\Rep(D)\cong \cD^\k$ as abelian categories.
\item $\Rep(B)\cong (\cD^\k)_A$ as abelian categories.
\item $\Rep(C)\cong (\cD^\k)_A^\loc$ as abelian categories.
\item $B\hookrightarrow U$ such that $A\otimes (-)$ coincides with the restriction functor.
\item $B\cong \Rep(\Nichols)(\cC^\k)$ such that the module category action of $(\cD^\k)_A$ on $(\cD^\k)_A^\loc$ (via embedding) coincides with tensoring a $\Nichols$-module with an object in $\cC^\k$ and acting on the first factor. And such that the splitting functor coincides with the functor forgetting the $\Nichols$-action.   
\end{itemize}
Then the following abstract steps are replaced by the realization:
\begin{itemize}
    \item The arguments in Section \ref{subsec:BorelModulecat} are not necessary, because with $\Nichols$ we already have an explicit realizing algebra for the module category.
    \item The Nichols algebra arguments in Section \ref{subsec:NicholsArguments} can be replaced by checking explicitly that the algebra $\Nichols$
    coincides with the Nichols algebra $\NicholsOf(M)$.
    \item The entire monadicity arguments in Section \ref{sec:adjunction} reduces to classical algebra for the ring extension $U\supset B$. The algebra arguments in Section \ref{sec_troi} reduce to the classical Nakayama trick if $U$ and all its quotients are finitely generated modules over $B$, as we have argued in \cite{CLR23} Section 3.3.
\end{itemize}

\subsection{Conventions}

All our categories are abelian and $\C$-linear. We use the word tensor category synonymously with monoidal category, without implying any rigidity of finiteness. We use the word tensor functor synonymously with  monoidal functor for strict monoidal functors, and we explicitly say lax monoidal $F(X)\otimes F(Y)\to F(X\otimes Y)$ or oplax monoidal $F(X\otimes Y)\to F(X)\otimes F(Y)$, which in other sources is called monoidal or comonoidal.   

We reserve $\boxtimes$ for the Deligne  product of categories. We denote the tensor product in a tensor category by $\otimes$, including the HLZ tensor product over a vertex algebra $\mathcal{V}$. If there is possibility of confusion we put as index the respective category, for example $\otimes_\cC$ or for vertex algebras $\otimes_{\mathcal{V}}$. The symbol $\otimes_A$ is denotes the tensor product over an algebra inside some category. 

Standard textbooks on vertex algebra are \cites{Kac97,FBZ04}, standard textbook on category theory with different focus are  \cites{EGNO15,Riehl16,Walton24}.

\section{Cartan category}\label{sec:Cartan}

\subsection{Categories of type \texorpdfstring{$A_n$}{An}}\label{subsec:CartanAn}

We recall from from \cites{FK93,CM}: 
Let $n \in \mathbb Z_{>0}$. A braided fusion category $\mathcal C$ is said to be of type $A_n$ if it has $n$ inequivalent simple objects $X_0, \dots, X_{n-1}$ with 
\begin{itemize}
    \item  $X_0$ the tensor identity
    \item $X_1 \otimes X_i \cong X_{i-1} \oplus X_{i+1}$ for $1 \leq  i  \leq n-2$
    \item $X_1 \otimes X_{n-1} \cong X_{n-2}$
\end{itemize}
Then there exist  $p, q$  co-prime positive integers with  $p= n+1$ and 
\[
s = e^{\pi \i \frac{p+q}{2p}},
\]
such that the modular data of $\mathcal C$ is as follows, with 
$[i] = (s^{2i} - s^{-2i})/(s^2 - s^{-2})$:
\[
S_{i, j} = (-1)^{i+j}[(i+1)(j+1)], \qquad T_i = (-1)^i s^{i(i+2)}
\]
and this is a complete invariant of the braided tensor category $\mathcal C$. Let us denote the corresponding category $\cA(p, p+q)$. 

For later purpose we also quote the braiding and the associator from \cite{CM} page 4 in a particular example:

\begin{example}\label{exm_tripleProduct}
We have for $p>3$ the tensor product
\begin{align*}
X_1\otimes X_1 &= X_0 \oplus X_2 \\
\end{align*}
and the braiding is given by 
$$\begin{pmatrix}
    R(1,1,0) & 0 \\
    0 & R(1,1,2)
\end{pmatrix}
=
\begin{pmatrix}
    -s^{-3} & 0 \\
    0 & s \\
\end{pmatrix}
$$
We have the triple tensor product
\begin{align*}
X_1\otimes (X_1 \otimes X_1)
&=X_1\otimes (X_0 \oplus X_2)\\
&=X_1 \oplus (X_1 \oplus X_3)  \\
(X_1\otimes X_1) \otimes X_1
&=(X_0 \oplus X_2) \otimes X_1\\
&=X_1 \oplus (X_1 \oplus X_3)  
\end{align*}
and with respect to these bases the associator is given as follows, where the column corresponds to the summand in the first decomposition and the row to the summand in the second decomposition. In the convention of \cite{CM} we have $\theta_{1,1,0}=-[2]$ and $\theta_{1,1,2}=[3]$, which enters the denominator as a square, and 
 $\theta_{1,3,2}=[4]$:

\begin{align*}
&\begin{pmatrix}
    \left\{
    \begin{matrix}
        1 & 1 & 0 \\
        1 & 1 & 0 \\
    \end{matrix} \right\} 
    &
    \left\{
    \begin{matrix}
        1 & 1 & 2 \\
        1 & 1 & 0 \\
    \end{matrix}\right\}
    & 
    0
    \\
    \left\{
    \begin{matrix}
        1 & 1 & 0 \\
        1 & 1 & 2 \\
    \end{matrix} \right\}
    &
    \left\{
    \begin{matrix}
        1 & 1 & 2 \\
        1 & 1 & 2 \\
    \end{matrix} \right\} 
    & 
    0  
    \\
    0
    &
    0
    &
    \left\{
    \begin{matrix}
        1 & 1 & 2 \\
        1 & 3 & 2 \\
    \end{matrix}  \right\} 
    %
\end{pmatrix}
\\
&=
\begin{pmatrix}
\frac{-1}{[2]^2}
(-1)[2]!
&
\frac{-[3]}{[2]![3]^2}
[3]!
& 
0
\\
\frac{-1}{[2]![2]^2}
[3]!
&
\frac{-[3]}{[2]!^2[3]^2}
[3]!
& 
0
\\
0
&
0
&
\frac{-[2]!^2[3]}{[3]![2]![2]!\,[4]\,[3]}
(-1)[4]!
\end{pmatrix}
\\
&=
\begin{pmatrix}
\frac{1}{[2]}
&
-1
& 
0
\\
-\frac{[3]}{[2]^2}
&
-\frac{1}{[2]}
& 
0
\\
0
&
0
&
1
\end{pmatrix}
\end{align*}
We discuss the exceptional cases: We have $[2]=0$ resp. $[3]=0$ if $s^4+1=0$ resp. $s^8+s^4+1=0$, meaning $s^4$ is a primitive root of unity of order $2$ resp. $3$, which  corresponds to the cases $p=2$ resp. $p=3$. In the first case $X_1$ is not defined. In the second case $X_2$ is not defined and the fusion rule is $X_1\otimes X_1=X_0$, the corresponding associator is just the left upper element $1/[2]$. We will later require the statement that the upper $2\times 2$ matrix has no zeroes for $p>3$. 

\end{example}

\subsection{Representations of the Virasoro algebra \texorpdfstring{$\Vir_{u,v}$}{}}\label{subsec:CartanVir}

We consider the simple Virasoro vertex algebra at central charge 
\[
c_{u,v}=1-6\frac{(\k+1)^2}{\k+2}
=13-6\frac{v}{u}-6\frac{u}{v},\qquad\quad 
\k=-2+\frac{u}{v}
\]
for $u, v$ co-prime integers both greater than one.


We consider the category of weight modules and remark that this coincides in this case with the subcategory of highest weight modules and ordinary modules. This category is semisimple, in fact it is a modular tensor category, and it has simple modules $\VirPhi_{r, s}$ with $1 \leq  r \leq  u-1$ and $1 \leq s \leq v-1$ and the only identifications are $\VirPhi_{r, s} \cong \VirPhi_{u-r, v-s}$. 
Let $\mathcal C(u, v)$ be this category and $\mathcal C^L(u, v)$ be the subcategory whose simple objects are the $\VirPhi_{r, 1}$. Similarly let $\cC^R(u, v)$ be the subcategory whose simple objects are the $\VirPhi_{1, s}$. 
Set $X_i^L = \VirPhi_{i+1, 1}$ and $X_j^R = \VirPhi_{1, j +1}$ for $0 \leq  i \leq u-2$ and $ 0 \leq  j \leq v-2$
they satisfy
\begin{itemize}
\item set $u = n+1$
    \item  $X_0^L$ the tensor identity of $\mathcal C^L(u, v)$,
    \item $X_1^L \otimes X_i^L \cong X_{i-1}^L \oplus X_{i+1}$ for $1 \leq  i  \leq n-2$
    \item $X_1^L \otimes X_{n-1}^L \cong X_{n-2}^L$
\end{itemize}
as well as
they satisfy
\begin{itemize}
\item set $v = n+1$
    \item  $X_0^R$ the tensor identity of $\mathcal C^R(u, v)$,
    \item $X_1^R \otimes X_i^R \cong X_{i-1}^R \oplus X_{i+1}^R$ for $1 \leq  i  \leq n-2$
    \item $X_1^R \otimes X_{n-1}^R \cong X_{n-2}^R$
\end{itemize}
so $\mathcal C^L(u, v)$ is of type $A_{u-1}$ and $\mathcal C^R(u, v)$ is of type $A_{v-1}$. 
The conformal weight of the top levels of $\VirPhi_{r, 1}$ and $\VirPhi_{1, s}$ are
\[
h_{r, 1} = \frac{r^2-1}{4}\frac{v}{u} - \frac{r-1}{2}, \qquad 
h_{1, s} = \frac{s^2-1}{4}\frac{u}{v} - \frac{s-1}{2}
\]
Let us set $a = e^{\pi i \frac{v}{u}}$ and $b = e^{\pi i \frac{u}{v}}$. Let $T^L_i, S^L_{i, j}$ be the $T$-matrices of the $X_i^L$ and $T^R_i, S^R_{i, j}$ be the $T$-matrices of the $X_i^R$. The modular data tells us that
\[
T_i^L = e^{2\pi i h_{i+1, 1}} = (-1)^i a^{i(i+2)}, \qquad T_\ell^R = e^{2\pi i h_{1, \ell + 1}} = (-1)^\ell b^{\ell(\ell+2)}
\]
and the $S$-matrices are if normalized such that $S_{0, 0} =1$
\[
S^L_{i, j} =  (-1)^{i+j} \frac{a^{2(i+1)(j+1)}- a^{-2(i+1)(j+1)}}{a^2-a^{-2}}
\]
and
\[
S^R_{\ell, k} =  (-1)^{\ell + k} \frac{b^{2(\ell+1)(k+1)} - b^{-2(\ell+1)(k+1)}}{b^2-b^{-2}}
\]
so that we see that
\[
\cC^L(u, v) \cong \cA(u, \tilde v), \qquad \cC^R(u, v) \cong \cA(v, \tilde u)
\]
with $\tilde v$ the smallest positive integer in $v + 4u \ZZ$ and $\tilde u $ the smallest positive integer in $u + 4v \ZZ$. Note that the corresponding values $s$ in Section \ref{subsec:CartanAn} are 
\[
s^L=e^{\pi\i\frac{v}{2u}},\qquad 
s^R=e^{\pi\i\frac{u}{2v}}
\]

\begin{remark}
The relevant fusion rule in our case is analog to the previous section 
\[\VirPhi_{1,2}
\otimes \VirPhi_{1,2}
=\VirPhi_{1,1}
\oplus \VirPhi_{1,3}
\]
The braiding can also be directly computed from the non-integer difference of the conformal weights
\[h_{1,1}=0,\qquad 
h_{1,2}=-\frac{1}{2}+\frac{3}{4}(u/v),\qquad 
h_{1,3}=-1+2(u/v)\]
Hence the braiding is given on the two direct summands by the scalar factors 
\[ e^{\pi\i(1-\frac{3}{2}(u/v))}=-s^{-3},\qquad
e^{\pi\i(\frac{1}{2}(u/v))}=s
\]
\end{remark}

\subsection{Representations  of a half lattice vertex algebra}
\label{subsec:CartanLattice}

We also consider the Euclidean space $\mathfrak{h}=\R^2$ with basis $c,d$ and inner product 
$$\begin{pmatrix} (c,c) & (c,d) \\ (d,c) & (d,d)\end{pmatrix}=\begin{pmatrix}0 & 2 \\2 & 0\end{pmatrix}$$
and the following even integral lattices of rank $1$ and $2$
\[L(0)=c\Z,\qquad L=c\Z\oplus d\Z.\]
We denote the corresponding Heisenberg vertex algebra by $H$ and its simple modules by  $\pi_{xc+yd}$. We denote the rank $1$ lattice algebra by $\Pi(0)$ and the full rank lattice vertex algebra by $\Pi$. One can consider for example the three vertex algebras 
\[
H,\qquad
 \Pi(0),\qquad
 \Pi.
\]
Their category of representations are $\Vect_\Gamma$ for the
following abelian group $\Gamma$ depending on the case
\[\h,\quad
\big(c\R\oplus\tfrac{1}{2}d\Z\big)/c\Z=\R/\Z\oplus \Z,\quad
\big(\tfrac{1}{2}c\Z\oplus\tfrac{1}{2}d\Z\big)/
(c\Z\oplus d\Z)=\Z_2\times \Z_2
\]
For the middle case we introduce for later use the parametrization 
$$\Pi_\ell(\lambda)=\bigoplus_{n\in\Z}\pi_{\ell(\tfrac{\k}{4}c+\tfrac{1}{2}d)+\lambda c+nc},\qquad \ell\in\Z,\;\lambda\in \R/\Z$$
The category of representations is $\Vect_\Gamma^Q$ with the quadratic form \[Q(xc+yd)=\exp\left(\pi\i \begin{pmatrix}
   x\\y  
\end{pmatrix}^T \begin{pmatrix}0 & 2 \\2 & 0\end{pmatrix}
\begin{pmatrix}
   x\\y  
\end{pmatrix} 
\right)
\]
Associated to any quadratic form $(\Gamma,Q)$ is a cohomology  class of abelian $3$-cocycle $(\sigma,\omega)$, that describe a concrete braiding and associator \cites{MacL52,JS93}. For vertex algebras, this corresponds to a choice of normalizations for the abelian intertwining operators. In our case we can choose the following slightly nonstandard braiding \[\sigma(xc+yd,x'c+y'd)=\exp\left(\pi\i \begin{pmatrix}
   x\\y  
\end{pmatrix}^T \begin{pmatrix}0 & 4 \\ 0 & 0\end{pmatrix}
\begin{pmatrix}
   x'\\y'  
\end{pmatrix} 
\right),
\]
satisfying the defining property $Q(xc+yd)=\sigma(xc+yd,xc+yd)$, which has the nice property that it is in both arguments again independent of the representative in the lattice coset $v+\tfrac{1}{2}\Z$. In particular this is well-defined without choices of representative, and thus bimultiplicative, and accordingly  we may choose trivial associator $\omega=1$.
For later convenience we also express the braiding in the alternative parametrization 
\[
\Pi_\ell(\lambda)\otimes\Pi_{\ell'}(\lambda')
\xrightarrow{\;\exp(2\pi\i\,(\ell\k/2+2\lambda)\ell')\;}
\Pi_{\ell'}(\lambda')\otimes\Pi_\ell(\lambda)
\]

\subsection{Cartan category \texorpdfstring{$\cC^\k$}{Ck}}\label{subsec:Cartan}

Consider as \emph{Cartan category} the semisimple braided tensor category
$$\cC^\k=\cC(u,v)\boxtimes \Vect_{L^*/L}$$ 
where  
$\Vect_{L^*/L}$ is the category of representations of the half-lattice algebra $\Pi(0)$ in Section \ref{subsec:CartanLattice} with simple modules $\Pi_\ell(\lambda)$ for $\ell\in\Z$, $\lambda\in \R/\Z$. The simple objects in $\cC^\k$ are denoted 
$$\VirPhi_{r,s}\boxtimes \Pi_\ell(\lambda)$$

\subsection{Twisting the Deligne product}

Consider two braided tensor categories $\cC, \cD$ and assume that both are graded by the same abelian group $\Gamma$. 
Consider the category $\Vect_\Gamma^Q$ for some quadratic form $Q$. Let $B$ be the associated bilinear form. Consider the group $\Gamma \times \Gamma$ with bilinear form $\widetilde B: \Gamma \times \Gamma \rightarrow \mathbb C^*, ((g,h),(g', h')) \mapsto B(g, h') B(g', h)$. Let $\widetilde Q$ be the associated quadratic form and consider the Deligne product
\[
\cC \boxtimes \cD \boxtimes \Vect_{\Gamma \times \Gamma}^{\widetilde Q}. 
\]
Denote the graded subcategories corresponding to a group element $g$ by a subindex $g$. Then we set the $(\Gamma, Q)$-twisted Deligne product to be
\[
\cC \boxtimes_{(\Gamma, Q)} \cD := \bigoplus_{g, h \in \Gamma} \cC_g \boxtimes \cD_h \boxtimes \left(\Vect_{\Gamma \times \Gamma}^{\widetilde Q}\right)_{g, h}.
\]
The effect of this procedure is that the categories $\cC$ and $\cD$ don't centralize each other anymore, but only projectively centralize each other, that is the monodromy is
\[
M_{X \boxtimes 1, 1 \boxtimes Y} =  B(g, h)\, \mathrm{Id}_{X \boxtimes Y}, \qquad \text{for} \ X \in \cC_g, \ Y \in \cD_h. 
\]
As an example let $\g$ be simple and simply laced, set $q = e^{\frac{\pi i}{k+h^\vee}}$ and $q' = e^{\pi(k+h^\vee)}$. Then consider $\cC = \Uq(\g)$-mod and $\cD = \mathrm{U}_{q'}(\g)$-mod. 
Then both categories are graded by $P/Q$ with $P$ the weight lattice of $\g$ and $Q$ the root lattice. 
For $B$ we take $e^{-2\pi i\kappa}$ with $\kappa$ the Killing form. Then 
it is shown in \cite{CGN} that for irrational $k$ the principal $W$-algebra of $\g$ has a vertex tensor category of modules that is equivalent to the twisted Deligne product
\[
\Uq(\g)\textup{-mod} \ \boxtimes_{(P/Q, B)} \ \mathrm{U}_{q'}(\g)\textup{-mod}.
\]
We restrict to $\g = \mathfrak{sl}_n$.
If we now take $\k = - n + \frac{u}{v}$ to be a non-degenerate admissible level, then neither 
$\Uq(\g)\textup{-mod}$ nor $\mathrm{U}_{q'}(\g)\textup{-mod}$ are semisimple and we replace both by their semisimplifications, that is we quotient by all negligible morphisms and consider the subcategory of this quotient that is generated by the images of the simple non-negligible objects.

Under semisimplification, the highest-weight modules $M_{(u-n)\omega_i}$ of highest-weight
$(u-n)\omega_i$ with $\omega_i$ the $i$-th fundamental weight of $\mathfrak{sl}_n$ become invertible objects in  $\Uq(\g)\textup{-mod}^{s.s.}$. Denote the images of $M_{(u-n)\omega_i} \boxtimes M_{(v-n)\omega_i}$ in $\Uq(\g)\textup{-mod}^{s.s.} \ \boxtimes_{(P/Q, B)} \ \mathrm{U}_{q'}(\g)\textup{-mod}^{s.s.}$ by $J_i$ (and the identity by $J_0$). Then it is a technical computation to verify that the $J_i$ are transparent objects and so one can modularize. The modularization of $X$ is isomorphic to the one of $Y$ if and only if $X \cong J_i \otimes Y$ for some $i = 0, \dots, n-1$. To be very explicit we take the example of $\mathfrak{sl}_2$. We write $M^\k_{r, s}$ for the image of 
$M_{(r-1))\omega_1} \boxtimes M_{(s-1)\omega_1}$ in $\Uq(\mathfrak{sl}_2)\textup{-mod}^{s.s.} \ \boxtimes_{(P/Q, B)} \ \mathrm{U}_{q'}(\mathfrak{sl}_2)\textup{-mod}^{s.s.}$. Note that $P/Q \cong \mathbb Z/2\mathbb Z$. The transparent invertible objects are $J_0 = M^\k_{1, 1}$ and $J_1 = M^\k_{u-1, v-1}$. The tensor products  are
\[
J_0 \otimes M^\k_{r, s} \cong M^\k_{r, s}, \qquad 
J_1 \otimes M^\k_{r, s} \cong M^\k_{u-r, v-s}.
\]
Let $L^\k_{r, s}$ the image of $M^\k_{r, s}$ in the modularization, then 
$L^\k_{r, s} \cong L^\k_{r', s'}$ if and only if $(r', s') \in  \{ (r, s), (u-r, v-s)\}$, i.e. one gets the well-known identification of the modules of $\Vir_{u,v}$.
Since modularization  (de-equivariantization, since the transparent objects correspond to the group $\mathbb Z/2\mathbb Z$) is inverse to equivariantization and since twisting the Deligne product is obviously an invertible operation the uniqueness of tensor categories of type $A_n$ tells us that 
\[
\cC(u,v) \cong  \left(\Uq(\mathfrak{sl}_2)\textup{-mod}^{s.s.} \ \boxtimes_{(P/Q, B)} \ \mathrm{U}_{q'}(\mathfrak{sl}_2)\textup{-mod}^{s.s.}\right)^{\textup{mod}}.
\]
with the upper index mod indicating that the category is modularized.

\section{Quantum group category}\label{sec:QuantumGroup}

\subsection{Nichols algebras in braided tensor categories}\label{subsec:introNichols}

Let $\cC$ be a braided tensor category. The Nichols algebra $\NicholsOf(M)$ of an object $M$ is a Hopf algebra in~$\cC$.  A standard textbook for Nichols algebras is \cite{HS20}, the reader is also referred to the lecture notes \cite{Len26}. 

More precisely, the Nichols algebra is a quotient of the tensor algebra $\mathfrak{T}(M)$ with the coproduct induced by the coproduct on $M$ familiar from Lie algebras 
$$M\xrightarrow{\Delta} \1\otimes M + M\otimes \1$$
extended to $\mathfrak{T}(M)$, which naturally involves the braiding. The Nichols algebra can be directly defined as the tensor algebra modulo the kernel of the quantum symmetrizer $\bigoplus_{n\geq 0}\sha_n$, see \cite{HS20} Section 1.9 or \cite{Len26} Section 2.1. Equivalently, the Nichols algebra is defined by several universal properties, that are the main reason for its significance. For example, it is the smallest Hopf algebra quotient of the tensor algebra that contains $M$, and hence it is the smallest Hopf algebra in $\cC$ generated by $M$ with the chosen coproduct. Note that the theory is usually developed for certain classes of braided tensor categories from Hopf algebra theory. The equivalence of the defining properties does not depend on this \cite{CLR23} Section 5.2, but we consider it an important problem to clarify the root system theory in the general setting of a braided tensor category. The following almost trivial Nichols algebras give a good intuition for the Nichols algebras constructed in this article.

\begin{example}\label{exm:NicholsTruncated}
Let $M$ be a $1$-dimensional vector space with basis $x$ and the following diagonal braiding
$$x\otimes x\xrightarrow{q} x\otimes x$$
for $q\in\C^\times$. Then the Nichols algebra $\NicholsOf(M)$ is the tensor algebra, which is the polynomial ring in one variable $\C[x]$ if $q$ is not a root of unity, and it is the truncated polynomial ring $\C[x]/x^n$ if $q$ is a primitive $n$-th root of unity.\\

Vaguely speaking  the reason for this behavior is that if $x$ is a derivation, then $x^n$ is usually not a derivation and cannot be set to zero without setting $x$ to zero. However if the base field has characteristic $n$, then $x^n$ is again a derivation, and a similar effect occurs for a braiding involving an $n$-th root of unity. 
\end{example}
\begin{example}
Let $M$ be an $n$-dimensional vector space with basis $x_1,\ldots,x_n$ and the trivial  braiding  
$$x_i\otimes x_j\mapsto x_j\otimes x_i,$$
then the Nichols algebra is the free commutative algebra or symmetric algebra $\mathrm{Sym}(X)$. Similarly, for the constant braiding
$$x_i\otimes x_j\mapsto -x_j\otimes x_i,$$
the Nichols algebra is the free anticommutative algebra or exterior algebra $\Lambda(X)$.
More generally, for any symmetric braiding the Nichols algebra is braided commutative, and this constitutes the trivial case of Nichols algebras. 
\end{example}

\begin{example}[Quantum group]
    Let $\g$ be a semisimple complex finite-dimensional Lie algebra with Cartan subalgebra $\h$ and Killing form $(-,-)$, and let $\Gamma=\h^*$, so $\CC=\Vect_\Gamma$ is the category of weight spaces, which we endow with the braiding $q^{(\lambda,\mu)}$ for some $q\in \C^ \times$. Let $\alpha_1,\ldots,\alpha_n$ be a choice of simple roots, take the object $X=\C_{\alpha_1}\oplus\cdots\oplus \C_{\alpha_n}$, which has a braiding matrix $q_{ij}=q^{(\alpha_i,\alpha_j)}$. Then the Nichols algebra $\NicholsOf(X)$ is the quantum Borel part of the (small) quantum group $u_q(\g)^+$.
\end{example}

\subsection{Our Nichols algebra}\label{subsec:NicholsAlgebra}

Recall from Section \ref{subsec:Cartan} the semisimple braided tensor category $\cC^\k$ with simple objects $\VirPhi_{r,s}\boxtimes \Pi_\ell(\lambda)$
for  $1 \leq r \leq u-1, 1\leq s\leq v-1$ and $\ell\in\Z$ and $\lambda\in \R/\Z$.
Suppose that $\Pi_\ellNichols(\lambdaNichols)$ is chosen in such a way  that the following expression is an odd integer
\begin{align}\label{formula_NicholsParameterCondition}
\frac{1}{2}(u/v)+4xy
=\frac{1}{2}(u/v)+(\ellNichols \k/2+2\lambdaNichols)\ellNichols 
=1+((\ellNichols+1)\k/2+2\lambdaNichols)\ellNichols. 
\end{align}
Consider in $\cC$ the simple object and the simple invertible object
\begin{align*}
M&=\VirPhi_{1,2}\boxtimes\Pi_\ellNichols(\lambdaNichols)
\qquad v> 2
\\
\integral&=\VirPhi_{1,1}\boxtimes\Pi_{2\ellNichols}(2\lambdaNichols)
\end{align*}
and in the degenerate case $v=2$ we define instead $M=\Lambda$. This choice is later justified by proving the main theorem, a more systematic justification can be found in the final paragraph of Section \ref{subsec:PartialSemisimplification}. Recall from Example \ref{exm_tripleProduct} the tensor products in $\catC^\k$ for $v>3$
\begin{align*}
M\otimes M &= 
\;\,\underbrace{\VirPhi_{1,1}\boxtimes\Pi_{2\ellNichols}(2\lambdaNichols)}_{\integral}
\;\oplus\; 
\VirPhi_{1,3}\boxtimes\Pi_{2\ellNichols}(2\lambdaNichols)\\
M\otimes M\otimes M
&=2\VirPhi_{1,2}\boxtimes\Pi_{3\ellNichols}(3\lambdaNichols)
\;\oplus\; 
\VirPhi_{1,4}\boxtimes\Pi_{3\ellNichols}(3\lambdaNichols)
\end{align*}
while for $v=3$ we have $M\otimes M=\Lambda$ without the second summand.\\

We now want to compute for $M$ the Nichols algebra in the sense of section \ref{subsec:introNichols}. Note that in our case the Nichols algebra is very elementary, but it may be the first example in a category beyond vector spaces with additional structure. 

\begin{lemma}\label{lm_Nichols}
The Nichols algebra of the object $M$ is nonzero in degrees $0,1,2$ except for $v=2$ and is as follows:
\begin{itemize}
\item For $v=2$ the Nichols algebra of the invertible object $M=\integral$ is
$$\NicholsOf(M)=1\oplus \integral$$
with the defining relation being the simple invertible object $M^{\otimes 2}$ set to zero.\\

\item For $v=3$ the Nichols algebra of the invertible object $M$ is
$$\NicholsOf(M)=1\oplus M\oplus \integral$$
with the defining relation being the simple invertible object $M^{\otimes 3}$ set to zero.\\

\item For $v>3$ the Nichols algebra of the simple object $M$ is 
$$\NicholsOf(M)=1\oplus M\oplus \integral$$
with the defining relation being the following summand
set to zero 
$$\VirPhi_{1,3}\boxtimes\Pi_{2\ellNichols}(2\lambdaNichols)\subset M^{\otimes 2}$$
This relation can also be written as a modified braided commutator 
\[\left(\id-e^{2\pi\i(u/v)}c_{M,M}\right)(M\otimes M)=0\]
and the Nichols algebra could be called an an exterior algebra. 
\end{itemize}
Note that in each case $\integral$ is the integral in the Nichols algebra, i.e. the invertible object in highest degree.
\end{lemma}
These Nichols algebras are roughly comparable to a truncated polynomial ring $\C[x]/x^2$, to a truncated polynomial ring $\C[x]/x^3$ and to an exterior algebra of a simple $2$-dimensional object, respectively, which were discussed in the previous section. 
\begin{proof}
The necessary braidings and associators were computed in Example \ref{exm_tripleProduct}.
The braiding $c_{X,X}$ acts on the tensor square
\[M\otimes M= 
\VirPhi_{1,1}\boxtimes\Pi_{2\ellNichols}(2\lambdaNichols)
\; \oplus \;
\VirPhi_{1,3}\boxtimes\Pi_{2\ellNichols}(2\lambdaNichols)
\]
by the scalars 
\[
\begin{pmatrix}
    -s^{-3}\cdot e^{\pi\i(4xy)}
    &
    0
    \\
    0 
    &
    s\cdot e^{\pi\i(4xy)}
\end{pmatrix},\qquad
\text{with }s=s^R=e^{\pi\i\,\frac{1}{2}(u/v)}
\]
For the choice $\frac{1}{2}(u/v)+4xy$ an odd integer these scalars are
\[
\begin{pmatrix}
    e^{-2\pi\i(u/v)}    
    &
    0
    \\
    0 
    &
    -1
\end{pmatrix},\qquad
\]
We have made our choice precisely to force the second entry to be $-1$. Hence the second summand is in the kernel of the second quantum symmetrizer $\sha_2=\id+c_{M,M}$ and it is by definition zero in the Nichols algebra, similarly as in Example \ref{exm:NicholsTruncated}.
From the given braiding, it is also clear that  we can write the second summand for $v\neq 2$ as a braided commutator 
\[\left(\id-e^{2\pi\i(u/v)}c_{M,M}\right)(M\otimes M)\]
We now continue depending on the cases
\begin{itemize}
\item 
For $v=3$ the second summand is not present anyway and $M$ is an invertible object. This means the tensor cube is
\[(M\otimes M)\otimes M= 
\VirPhi_{1,2} \boxtimes \Pi_{3\ellNichols}(3\lambdaNichols)
\]
and the braiding $c_{M,M}$ is given by the scalar $e^{-2\pi\i(u/v)}$. This is for $v=3$ a primitive third root of unity, so the tensor cube is  clearly in the kernel of the quantum symmetrizer $\sha_3$.
\item 
For $v>3$ we consider the tensor cube
\[(X\otimes X)\otimes X= 
2\VirPhi_{1,2} \boxtimes\Pi_{3\ellNichols}(3\lambdaNichols)
\; \oplus \;
\VirPhi_{1,4}\boxtimes\Pi_{3\ellNichols}(3\lambdaNichols)
\]
From the tensor square summand $\VirPhi_{1,3}\boxtimes\Pi_{2\ellNichols}(2\lambdaNichols)$ being zero in the Nichols algebra, it follows by the Nichols algebra being an algebra (or a direct factorization of the quantum symmetrizer $\sha_3$) that in the tensor cube the second summand of type $\VirPhi_{1,2}\boxtimes\Pi_{3\ellNichols}(3\lambdaNichols)$ vanishes. Similarly, in the right bracketing the second summand of this type vanishes. If the respective associator has nonzero entries, then both summands in both brackets are zero. But this we checked explicitly in Example \ref{exm_tripleProduct} for $\cC^R(u,v)$ and for $\Vect_{L^*/L}$ the associators are nonzero anyway. Hence the tensor cube vanishes altogether in the Nichols algebra. Note that this is the categorical version of the argument that in the exterior algebra with $xy+yx=0$ all triple products are zero. \SimonRem{Argument without associator? Counterexample?} 
\item 
For $v=2$ we have defined alternatively $M=\Lambda$. This is an invertible object and the braiding is $e^{\pi\i\cdot 4(4xy)}$. Since by assumption $4(4xy)=-2(u/v)+4(2\Z+1)$ which is for $v=2$ an odd integer, the braiding is $-1$ and the tensor square is in the kernel of the quantum symmetrizer $\sha_2$. 
\end{itemize}
\end{proof}

\subsection{Borel category}\label{subsec:QuantumBorel}

\begin{definition}
Let $\cB^\k$ be the tensor category of representations of the Nichols algebra $\NicholsOf(M)$ in Lemma \ref{lm_Nichols} inside the category $\cC^\k$. Here we take the tensor product using the bialgebra structure of $\NicholsOf(M)$, that is, the tensor product in $\cB^\k$ is the tensor product in $\cC^\k$ with an action of $\NicholsOf(M)$ via the coproduct of $\NicholsOf(M)$.    
\end{definition}

We list some important standard adjunctions, this will be discussed more systematically in Section \ref{sec:adjunction}: As for any algebra,  there is a functor forgetting the action of $\NicholsOf(M)$. This is an exact and faithful monoidal functor, whose left adjoint and right adjoint are the induction and coinduction functor; it is a special case of change of rings for the ring homomorphism $1\hookrightarrow \NicholsOf(M)$. As for any bialgebra, the functor forgetting the action is strict monoidal. 
\[\begin{tikzcd}[row sep=6ex, column sep=8ex]
{
\renewcommand{\arraystretch}{0.7}
\cC
\arrow[r, rightarrow, shift right=-8pt,"\mathrm{coind}"{pos=0.5, anchor=center, inner sep=0pt}] 
\arrow[r, leftarrow, shift right=0pt,"\forget"{pos=0.5, anchor=center, inner sep=0pt}]
\arrow[r,rightarrow, shift right=8pt,"\mathrm{ind}"{pos=0.5, anchor=center, inner sep=0pt}] 
}
& 
\cB
&
\begin{array}{c}
    \scriptstyle \textnormal{left exact}\\
    \scriptstyle \textnormal{exact} \\
    \scriptstyle \textnormal{right exact}
\end{array}
&
\begin{array}{c}
    \scriptstyle \textnormal{lax monoidal} \\
    \scriptstyle \textnormal{strict monoidal}\\
    \scriptstyle \textnormal{oplax monoidal}
\end{array}
\end{tikzcd}\]
On the other hand, by the unit $\NicholsAlgebraCounit$ any object $V\in \cC^\k$ can be turned into a $\NicholsOf(M)$-module $V_\NicholsAlgebraCounit$ with trivial action. This is an exact and faithful tensor functor in the other direction, whose right adjoint and left adjoint are the invariant and coinvariants functor; it is a special case of change of rings for $\NicholsAlgebraCounit:\NicholsOf(M)\to 1$.
\[\begin{tikzcd}[row sep=6ex, column sep=8ex]
{
\renewcommand{\arraystretch}{0.7}
\cC
\arrow[r, leftarrow, shift right=-8pt,"\mathrm{inv}"{pos=0.5, anchor=center, inner sep=0pt}] 
\arrow[r, rightarrow, shift right=0pt,"\mathrm{triv}"{pos=0.5, anchor=center, inner sep=0pt}]
\arrow[r,leftarrow, shift right=8pt,"\mathrm{coin}"{pos=0.5, anchor=center, inner sep=0pt}] 
}
& 
\cB
&
\begin{array}{c}
    \scriptstyle \textnormal{left exact}\\
    \scriptstyle \textnormal{exact} \\
    \scriptstyle \textnormal{right exact}
\end{array}
&
\begin{array}{c}
    \scriptstyle \textnormal{lax monoidal} \\
    \scriptstyle \textnormal{strict monoidal}\\
    \scriptstyle \textnormal{oplax monoidal}
\end{array}
\end{tikzcd}\]
These are one-sided inverse to each other, since $\forget\circ\triv=\id_\cC$

\bigskip

The Nichols algebra has a property familiar from local or basic algebras in classical algebra, see e.g. \cite{Zimm14} Section 1.6 and Section 4.5, : As a graded algebra with $\NicholsOf(M)_0=1$, the ideal of all elements in degree~$>0$ is the unique maximal ideal. Hence the simple objects in $\cB^\k$ are simple objects in $\cC^\k$ with trivial action. On the other hand the induced modules are the projective covers, as in the classical case:

\begin{lemma}
Assume $\cC$ is semisimple and $\NicholsOf$ is a 
$\mathbb{N}$-graded algebra with $\NicholsOf_0=\1$ and for some bound $N$ we have $\NicholsOf_n=0$ for $n>N$. Then the induced module $\NicholsOf\otimes_{\cC} X$ of any simple module $X\in\cC$ is the projective cover of $X_\NicholsAlgebraCounit$.
\end{lemma}
\SimonRem{I think would work for $\cC$ nonsemisimple if $X$ is a projective cover}
\begin{proof}
    The induced module is projective, because for $N\twoheadrightarrow M$ and $\NicholsOf\otimes_{\cC} X\to M$ we can lift the $\cC$-morphism $X\to M$ to $X\to N$, since $X$ is projective in $\cC$ semisimple, and then there is a unique extension to a $\NicholsOf$-module morphism from the induced module (differently said: induction is right exact and hence maps projectives to projectives).

    The counit $\NicholsAlgebraCounit$ gives rise to a $\NicholsOf$-module epimorphism
    $$\NicholsAlgebraCounit\otimes \id:\;
    \NicholsOf\otimes_{\cC} X\twoheadrightarrow X_\NicholsAlgebraCounit$$

    We finally show that this is a projective cover, which means showing that the previous epimorphism is coessential: Suppose $L\subset \NicholsOf\otimes_{\cC} X$ is a submodule that together with the kernel of $\NicholsAlgebraCounit\otimes \id$ generates the entire induced module. Then the image must be a submodule of $X$, and since $X$ is assumed simple it is all of $X$. But then due to the bounded degree, we can backwards inductively show that $L=\NicholsOf\otimes_{\cC} X$. 
\end{proof}

\begin{corollary}
For $v\geq 3$ the projective cover of the simple object $V=\VirPhi_{r,s}\boxtimes \Pi_\ellNichols(\lambdaNichols)$ with trivial action is the induced module $\NicholsOf(M)\otimes_{\cC^\k} X$. Its Loewy diagram is 
 
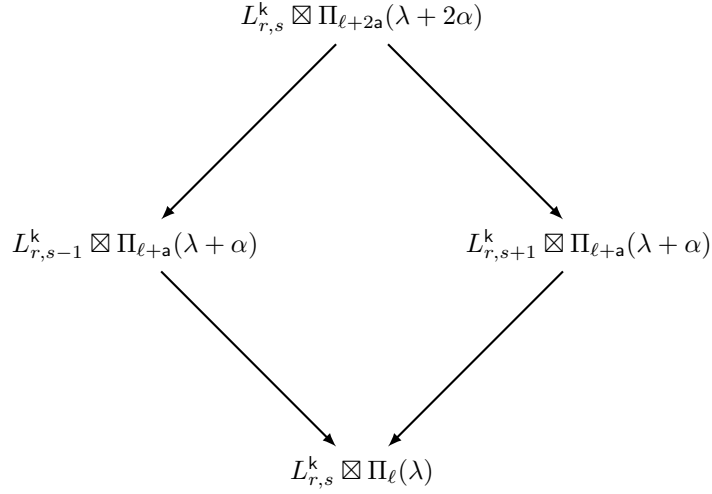
\begin{center}
\begin{tikzpicture}[scale=1]
\node (top) at (0,3) [] {$\VirPhi_{r, s} \boxtimes \Pi_{\ell+2\ellNichols}(\lambda+2\lambdaNichols)$};
\node (left) at (-3,0) [] {$\VirPhi_{r, s - 1} \boxtimes \Pi_{\ell+\ellNichols}(\lambda+\lambdaNichols)$};
\node (right) at (3,0) [] {$\VirPhi_{r, s + 1} \boxtimes \Pi_{\ell+\ellNichols}(\lambda+\lambdaNichols)$};
\node (bottom) at (0,-3) [] {$\VirPhi_{r, s} \boxtimes \Pi_{\ell}(\lambda)$};
\draw[->, thick] (top) -- (left);
\draw[->, thick] (top) -- (right);
\draw[->, thick] (left) -- (bottom);
\draw[->, thick] (right) -- (bottom);
\end{tikzpicture}
\end{center}   
with the usual convention that $\VirPhi_{r, 0} = \VirPhi_{r, v} = 0$. 
\end{corollary}

\SimonRem{Recall we compared basic/connected for algebra and algebra in a category in Screening Theorem Paper Section 5.2}

\subsection{Quantum group category}\label{subsec:QuantumGroup}

Recall that for any tensor category $\cB$ there is a braided tensor category called \emph{Drinfeld center}.
Recall from \cite{CLR23}~Section~3.1 that for a tensor category $\cB$ with a braided tensor category $\cC$, which is a central subcategory of $\cB$, that is, the embedding has an upgrade to a braided embedding in the center $\cZ(\cB)$, then there is a braided tensor category called \emph{relative Drinfeld center} $\cZ_\cC(\cB)$ \cites{Maj95,Lau20}. Recall from \cite{CLR23} Section 6.1. that for a Hopf algebra $\Nichols$ in a braided tensor category $\cC$ there is a braided tensor category of \emph{Yetter-Drinfeld modules} $\YD{\Nichols}(\cC)$ \cite{Besp95}, which agrees with the relative Drinfeld center of $\Rep(\Nichols)(\cC)$ \cite{Lau20}.

\begin{definition}
For our choice of $\cC^\k$ and Nichols algebra $\NicholsOf(M)$ in $\cC^\k$ with category or representations $\cB^\k$, we define our \emph{generalized quantum group}
\[\cU^\k:=\cZ_{\cC^\k}(\cB^\k)=\YD{\NicholsOf(M)}(\cC^\k).\]
It  is a nondegenerately braided tensor category.
\end{definition}
\begin{remark}
We can define variants of this quantum group by replacing in the Deligne product $\cC^\k$  by $\cC^R(u,v)\boxtimes \Pi(0)$ the first factor by the right category of Virasoro modules and/or the second factor $\Pi(0)$ by a fully unrolled version $\Vect_\h$ or the fully rolled version $\Pi$ in Section \ref{subsec:CartanLattice} or just by a rank one Heisenberg VOA. They all contain a version of the same $M$  and correspond to the variations discussed in Section \ref{sec:variations}
\end{remark}

For any relative Drinfeld center $\cZ_\cC(\cB)$ there is an exact faithful tensor functor $\forget_c:\cZ_\cC(\cB)\to \cB$ forgetting the half-braiding $c_{X,(-)}$ and under favorable assumptions there is a right adjoint lax monoidal functor $(\forget_c)^\ra:\cZ_\cC(\cB)\leftarrow \cB$. We remark that this adjunction is a lax comonad in Section \ref{sec:adjunction}. In particular in our case where $\cB$ is the category of representations of a Hopf algebra $\Nichols\in \cC$,  there is an explicit right adjoint 
$$(\forget_c)^\ra=\coLaugwitz:\;V\mapsto\Nichols\otimes_\cC V$$ 
with the regular $H$-coaction on the first factor and an explicit somewhat involved $H$-action on the tensor product involving double braidings on $V$, which turns the result into a Yetter-Drinfeld module, see  \cite{LW21} Theorem 3.13 and following. 
In particular, the image of the tensor unit $\1_\NicholsAlgebraCounit$ is $\Nichols$ with regular coaction and adjoint action, and it is a commutative algebra in $\cU^\k$. 
 The image of an arbitrary trivial $\Nichols$-module $X_\NicholsAlgebraCounit$ is in the quantum group case precisely the coVerma module  
 $$\coLaugwitz(X_\NicholsAlgebraCounit)
 =(u_q\otimes_{u_q^{\geq0}} X_\NicholsAlgebraCounit\big)^*,$$
 which is as a $\cC$-object $(\NicholsOf(M)^*\otimes X)^*$ with the coregular action of the lower Borel part $\NicholsOf(M)^*$ and the adjoint action of the upper Borel part $\NicholsOf(M)^*$ involving the highest-weight condition encoded by $X$. We compute in our case:

 \begin{example}\label{exm:ourNicholsAdjoint}
 The Nichols algebra $\Nichols=\NicholsOf(M)$ from Lemma \ref{lm_Nichols} for $v\geq 3$ has the following adjoint representation 
$${_{\mathrm{ad}}}\NicholsOf(M)=\1\oplus K$$
 where $K$ is an indecomposable extension 
 $$0 \to \Lambda \to K\to M \to 0,$$
 which also appears (not splitting) as a submodule in the regular representation.
 For the proof, note first in general that the adjoint action increases the degree, and the unit is a direct summand. The adjoint action of the primitive subspace $M$ on $M$ is a braided commutator, and the proof of  Lemma \ref{lm_Nichols} shows that this is nonzero unless $e^{2\pi\i(u/v)}=-1$, which is $v=2$.
 \end{example}
 
\subsection{Relation to \texorpdfstring{$\Uq(\sl_{2|1})$}{Uq(sl(2|1))}: Partial semisimplification}\label{subsec:PartialSemisimplification}


In the current section we have constructed a nondegenerately braided nonsemisimple tensor category $\cU^\k$ as Yetter-Drinfeld modules of a Nichols algebra inside the semisimple braided tensor category $\cC^\k$, which is related to the category of representations of the Virasoro algebra. Both categories do not admit a fiber functor to vector spaces. We want to explain in which way this $\cU^\k$ is nevertheless closely related by an interesting new general construction scheme to the quantum group $\Uq(\sl_{2|1})$ at $q=e^{2\pi\i(u/v)}$, which we call \emph{partial semisimplification}.  The essential idea is to write the quantum group as an extension of $\Uq(\sl_2)$ by a parabolic Nichols algebra and then passing to the semisimplification of $\Uq(\sl_2)$. The second copy of semisimplified $\mathrm{U}_{q'}(\sl_2)$  for $q'=e^{2\pi\i(v/u)}$ appearing in $\cU^\k$ will not play a role in this section.

\bigskip

As explained in \cite{HS20} Section 13.2 or summarized  \cite{Len26} Section 5.2, we have the following appearance of Nichols algebras in parabolic contexts, which is for example a key insight for reflection theory: Let $X=\bigoplus_{i\in I} X_i$ be a direct sum of simple objects in a braided tensor category $\cH$ (in our case graded vector spaces), where $I$ are called simple roots, and $\NicholsOf(X)$ the Nichols algebra. Let $J\subset I$ be a subset of simple roots, then we have a canonical embedding $\NicholsOf(X_J)\to \NicholsOf(X)$ and a projection, and thus by the Radford projection theorem we have a smash product with the space of coinvariants 
$$\NicholsOf(X)=\NicholsOf(X_{I\backslash J})\rtimes \NicholsOf(X_J)$$
where $X_J=\bigoplus_{i\in J} X_i$ is an object in $\cH$ and $\smash{X_{I\backslash J}=\bigoplus_{i\in I\backslash J}\mathrm{ad}_{\NicholsOf(X_J)}(X_i)}$ is an object in $\smash{\YD{\Nichols(X_J)}(\cH)}$. Moreover, we have an equivalence of braided tensor categories
\[\YD{\NicholsOf(X)}
\cong\YD{\NicholsOf(X_{I\backslash J})}\left(\YD{\NicholsOf(X_J)}(\cH)\right)
\]
The Nichols algebra $\NicholsOf(X_{I\backslash J})$ is an interesting example of a Nichols algebra in the category of representations of  the lower rank quantum group associated to $X_J$. Its root system has been determined  in  \cite{CL17} to be a restriction arrangement and \cite{AA20} show that essentially all Nichols algebras over quantum groups arise from such constructions. 

\bigskip

Depending on the context we can now consider $\cC$, the semisimplification and modularizations of a usual small quantum group $\YD{\NicholsOf(X_J)}(\cH)$ or respectively a corresponding category of tilting modules of a corresponding infinite quantum group, to which the object $X$ can be lifted. In particular, we can consider the image of the Nichols algebra $\smash{\NicholsOf(X_{I\backslash J})}$ under the semisimplification functor $F^{s.s.}$.

\begin{definition}
The \emph{partial semisimplification} of a braided tensor category of generalized quantum group type, that is of the form $\smash{\YD{\NicholsOf(X)}(\cH)}$, with respect to the parabolic data $J\subset I$ is defined as the braided tensor category
\[\YD{F^{s.s.}(\NicholsOf(X_{I\backslash J}))}(\cC)
\]
\end{definition}
Before discussing the example related to our article, let us mention a few general questions for this new definition
\begin{problem}\label{prob_NicholsSemisimplify}
When is the result of the semisimplification of a Nichols algebra actually again a Nichols algebra
$$
F^{s.s.}(\NicholsOf(X_{I\backslash J}))
=\NicholsOf(M),\qquad 
M=F^{s.s.}(X_{I\backslash J})
$$
Note that there is an obvious morphism in one direction, but it does neither have to be surjective nor injective:
\begin{itemize}
    \item Since $F^{s.s.}$ is not faithful nor conservative, the image $F^{s.s.}(M)$ may not be a generator of $F^{s.s.}(\mathfrak{T}(M))$. In extreme cases, $F^{s.s.}(M)$ might be zero (and with it the image of the tensor algebra), but the image of the quotient $F^{s.s.}(\mathfrak{B}(M))$ may be nonzero and generated by higher degree elements that in the image are primitive. This happens in particular in our  case $v=2$ as we discuss after Example \ref{ex:ssParabolic}.
    \item Since $F^{s.s.}$ is not left exact, there may be additional relations in $\NicholsOf(F^{s.s.}(M))$, because elements may become primitive up to terms killed by $F^{s.s.}$ resp. the kernel of the quantum symmetrizer may become larger. We have not constructed an example for this, but intuitively it may happen in root strings become so long that they touch negligible objects (without $M$ doing so already). 
\end{itemize}
\end{problem}

\begin{problem}\label{prob:Ostrik}
Does our construction for parabolics of quantum groups in quantum groups produce the quotient category under the thick abelian tensor ideal associated in \cite{Ost97,CEO25} to the Richardson cell of the chosen parabolic?
Can one establish a corresponding classification for quantum super groups?
\end{problem}

A main class of examples of interest is as follows:

\begin{example}\label{ex:ssParabolic}
Let $\g$ be a contragredient Lie superalgebra and $\Uq(\g)$ be the corresponding quantum group. Let $J$ be the set of even simple roots. Assume that~$\g$ is a Lie superalgebra of type $1$, which means that $\NicholsOf(X_{I\backslash J})$ is an exterior algebra. As a remark, this also implies that the answer to Problem \ref{prob_NicholsSemisimplify} should in this case automatically be positive, so the image under semisimplification is $\NicholsOf(M)$ for $M=F^{s.s.}(X_{I\backslash J})$. Partial semisimplification in this case means, in a sense, that the $q$-deformed bosonic root vectors are semisimplified, while the fermionic root vectors remain non-semisimple. 
\end{example}
We now explicitly discuss qualitatively the example directly related to the category $\cU^\k$ in this article: Let $\Uq(\sl_{2|1})$ be the associated quantum super group, whose category of representations can be described as $\smash{\YD{\NicholsOf(X_1\oplus X_2)}(\cH)}$ for $\cH=\Vect_\Gamma$ the category of representations of the Cartan subalgebra and $X_1\oplus X_2$ the $2$-dimensional vector space with diagonal braiding

\[
(q_{ij})=\begin{pmatrix} q^2 & q^{-1} \\ q^{-1} & -1 \end{pmatrix},\qquad q=e^{2\pi\i(u/v)}
\]

\begin{remark}To be more precisely, in our construction the overall Cartan $\Gamma$ for our version of $\Uq(\sl_{2|1})\otimes \mathrm{U}_{q'}(\sl_2)$ has rank $4$, according to the  free field realization of $\Vir$ times the category or representations of $\Pi(0)$: There are discrete labels $r,s$ labeling $L_{r,s}$ and the weights of $\Uq(\sl_{2})\otimes \mathrm{U}_{q'}(\sl_2)$, and a discrete and a continuous label $\ell,\lambda$.   
\end{remark}

This Nichols algebra has root vectors $E_1,E_2$ and $E_{12}:=[E_1,E_2]_{q_{12}}$ , where $E_2,E_{12}$ are fermionic. We pick the parabolic $\sl_2$ with $J=\{1\}$. The quantum Serre relation is $[E_1,E_{12}]_{q_{11}q_{12}}=0$. Hence $X_{I\backslash J}=\mathrm{span}(E_2,E_{12})$ is the standard representation of the Hopf subalgebra $u_q(\sl_2)$ generated by $E_1,F_1,K_1$. The tensor square of $X_{I\backslash J}$ has relations 
\begin{align*}
E_2^2&=0 \\
E_{12}^2&=0 \\
E_2E_{12} &= E_2E_1E_2 = -q_{12}^{-1} E_1E_{12}
\end{align*}
The braiding over $u_q(\sl_2)\otimes \C[K_2]$ in this tensor square is 
$q_{22}=q_{12,12}=-1$ for the $3$-dimensional submodule generated by $E_2\otimes E_2$, and it is $q_{2,12}=-q^{-1}$ for the $1$-dimensional submodule generated by $E_2$. Now recall from \cites{AP95} and \cite{BK01} Theorem 3.3.20 that the semisimplification of the category of representations of $\Uq(\sl_2)$ roughly coincides with our category
$\cC^\k_R$ related to the Virasoro algebra. Then the Nichols algebra $F^{s.s.}(\NicholsOf(X_{I\backslash J} ))$ in this section coincides with our Nichols algebra $\NicholsOf(M)$ for $M=F^{s.s.}(X_{I\backslash J})$ in $\cC^\k_R\subset \cC^\k$ in Lemma \ref{lm_Nichols} and the partial semisimplification in this section coincides with the category $\cU^\k_R\subset \cU^\k$ constructed in the previous section. In a sense, the $q$-deformed bosonic root $\alpha_1$ is semisimplified, while the fermionic root $\alpha_2$ remains non-semisimple, and they lead to the characteristic diamond structure of the projective covers.

The Verma modules associated to this Nichols algebra i.e. image of the functor $\coLaugwitz$ are in this language the Kac modules of $\Uq(\sl_{2|1})$ induced from an irreducible representations of $\Uq(\sl_2)$, semisimplified. They will correspond to the restrictions of the Virasoro modules $L_{r,s}\boxtimes\Pi_\ell(\lambda)$ and hence to $\widehat{\sl}_2$-modules $\sigma^{\ell+1}\mathcal{E}_{2\lambda-\k,\Delta_{r,s}}$ according to formula \eqref{eqn:FF-affine-sl2-correspondence}. For special values of $\lambda$ these are indecomposable extensions of the $\widehat{\sl}_2$-modules $\sigma^{\ell+1}\mathcal{D}_{r,s}^\pm$ according to formula \eqref{eqn:exact-seq-E+E-}, which correspond to the simple modules of $\Uq(\sl_{2|1})\otimes \mathrm{U}_{q'}(\sl_2)$, partially semisimplified.

\bigskip

We finally consider the degenerate case $v=2$: Here, the module $X_{I\backslash J}$ is a projective simple module over $\Uq(\sl_2)$ and is sent to zero by semisimplification. The tensor square $(X_{I\backslash J})^{\otimes 2}$ is the projective cover of a $1$-dimensional object, which is also sent to zero. However, in the Nichols algebra  the only surviving term in degree $2$ is the $1$-dimensional quotient of the tensor square,  say $\integral$, which is not sent to zero, but becomes a new primitive generator of the image of the Nichols algebra 
$$F^{s.s.}(\NicholsOf(X_{I\backslash J}))=1\oplus 0 \oplus F^{s.s.}(\integral)=\NicholsOf(M),\qquad M=F^{s.s.}(\integral)$$
This matches the degenerate case $v=2$ in our main result and the choices in Lemma \ref{lm_Nichols}.

\begin{problem}
Different Weyl chambers gives different partial semisimplifications of the same category. What are their relation? In terms of vertex algebra, there is the Wakimoto realization of $\mathfrak{sl}_2$ inside a $\beta\gamma$-system, which is associated to $u_q(\sl_2)$ at $q^4=1$ and one (resp. two)  fermionic screenings, so this realization would lead to a similar result involving $U_q(\sl(2|1))$, but now in the other Weyl chamber with two fermionic simple roots and a braiding matrix 

\[
(q_{ij})=\begin{pmatrix} -1 & -q \\ -q & -1 \end{pmatrix},\qquad q=e^{2\pi\i(u/v)}
\]

\end{problem}

\begin{remark}\label{rem:Mukhin}
The category of representations of $U_q(\sl_{2|1})$ appears at least conjecturally for another vertex algebra, namely for the $W$-algebra $W^\ell(\sl_{2|1})$, and it is related to $L_\k(\mathfrak{sl}_2)$ via Kazama-Suzuki duality which we will explain in all detail in section \ref{sec:convolution}.
Let us for now discuss generic level and ignore additional free fields. 

One explicit mechanism for such an equivalence would be an invertible bimodule. While the regular bimodule for a vertex algebra is complicated to construct, we could consider the regular bimodule $U_q(\sl_{2|1})$ over itself and transport the left and right action to  $W^\ell(\sl_{2|1})$ and $L_\k(\mathfrak{sl}_2)$, respectively. It would be very interesting if this could be related to the bimodule constructed in \cite{FJM25} and more explicitly \cite{FJM26}~Example~2 in the context of deformed $W$-algebras, in the limit with zero deformation.
\end{remark}

\section{Weight modules for affine \texorpdfstring{$\sl_2$}{sl(2)} at admissible levels}\label{subsec:affine-sl2-wt-mods}


Let $\k$ be an admissible level for $\sl_2$, that is $\k = -2 + \frac{u}{v}$
where $u, v$ are coprime positive integers both greater than one. Set $t = \k+2$. The main interest of our article the following category
\begin{definition}
Let $\cD^\k$ be the category of  weight modules of the simple affine vertex algebra $L_\k(\sl_2)$ of $\sl_2$ at the admissible level $\k$.
\end{definition}

The category $\cD^\k$ is called $\cC_k^{\mathrm{wt}}(\sl_2)$ in \cite{ACK}, from which we quote extensively in this section. 
We recall what is known about this category:

\subsection{The abelian category}\label{subsec:sl2abelian}

We start with the abelian structure.

\begin{theorem}\textup{\cite{ACK}}
    For  $\k$ admissible, the category $\cD^\k$ of finitely-generated weight $L_\k(\sl_2)$-modules is a locally finite abelian category with enough projective objects.
\end{theorem}
We first recall the indecomposable lower-bounded modules in $\cD^\k$. To describe them, define
\begin{equation}
\lambda_{r,s} = r-1-ts,\qquad\Delta_{r,s}=\frac{(r-ts)^2-1}{4t}
\end{equation}
for integers $1\leq r\leq u-1$ and $0\leq s\leq v-1$. These have the symmetries
\begin{equation}\label{eqn:lambda-Delta-symmetries}
    \lambda_{u-r,v-s} = -\lambda_{r,s}-2,\qquad\Delta_{u-r,v-s}=\Delta_{r,s}
\end{equation}
for $1\leq r\leq u-1$ and $1\leq s\leq v-1$.
\begin{enumerate}
\item 
For $r\in\ZZ_{\geq 1}$, let $\cL_{r,0}$ denote the simple $L_\k(\sl_2)$-module whose top level is the simple $r$-dimensional $\sl_2$-module $L_r$.
 It is a module for $L_\k(\sl_2)$ if and only if $1\leq r\leq u-1$. 
 \item For $\lambda\in\CC$, let $\cD_\lambda^+$ denote the simple $L_\k(\sl_2)$-module whose top level is the  simple highest-weight $\sl_2$-module $D_\lambda^+$ of highest-weight $\lambda$. 
It is a module for $L_\k(\sl_2)$ if and only if $\lambda=\lambda_{r,s}$ for some integers $1\leq r\leq u-1$ and $0\leq s\leq v-1$. In this case, we set $\cD_{r,s}^+:=\cD_{\lambda_{r,s}}^+$, so that in particular $\cD_{r,0}^+=\cL_{r,0}$. The module $\cD_{r,s}^+$ is the irreducible highest-weight $\slhat_2$-module whose highest-weight vector has $h_0$-weight $\lambda_{r,s}$ and conformal weight $\Delta_{r,s}$. 
\item 
 Let $\cD_{r,s}^-$ be the contragredient $L_\k(\sl_2)$-module of $\cD_{r,s}^+$. This is the irreducible lowest-weight $\slhat_2$-module whose lowest-weight vector has $h_0$-weight $-\lambda_{r,s}$ and conformal weight $\Delta_{r,s}$. Note that $\cD_{r,0}^- = \cL_{r,0}$.
 \item For $\lambda+2\ZZ \in \CC/2\ZZ$ and $\Delta\in\CC$ one can construct two $\sl_2$-modules $\cE^\pm_{\lambda,\Delta}$
 whose weight-support is $\lambda+2\ZZ$ with one-dimensional weight spaces and Casimir eigenvalue $\Delta$. 
 Let $\cE^\pm_{\lambda,\Delta}$ be the almost simple quotient of the Verma module induced from $\smash{E^\pm_{\lambda,2t\Delta}}$.
 It is a module for $L_\k(\sl_2)$ if and only if $\Delta=\Delta_{r,s}$ for some integers $1\leq r\leq u-1$ and $1\leq s\leq v-1$. These modules are simple and isomorphic if $\lambda +2\ZZ \neq \pm\lambda_{r,s} + 2\ZZ$. In this case we will just write 
 $\cE_{\lambda,\Delta_{r, s}}$ for $\cE^\pm_{\lambda,\Delta_{r, s}}$.

 If $\lambda\in\pm\lambda_{r,s}+2\ZZ$, then we set $\cE_{r,s}^+=\cE^+_{\lambda_{r,s},\Delta_{r,s}}$ and $\cE^-_{r,s}=\cE^-_{-\lambda_{r,s},\Delta_{r,s}}$. Note that $\cE^\pm_{\mp\lambda_{r,s},\Delta_{r,s}} =\cE^\pm_{u-r,v-s}$ by \eqref{eqn:lambda-Delta-symmetries}.
 The $\cE_{r,s}^\pm$ are indecomposable, and  there are non-split short exact sequences
\begin{align}\label{eqn:exact-seq-E+E-}
 &0 \longrightarrow \cD_{r,s}^+ \longrightarrow \cE_{r,s}^+ \longrightarrow \cD^-_{u-r,v-s} \longrightarrow 0, \\   &0 \longrightarrow \cD_{r,s}^- \longrightarrow \cE_{r,s}^- \longrightarrow \cD^+_{u-r,v-s} \longrightarrow 0.
\end{align}
for $1\leq r\leq u-1$ and $1\leq s\leq v-1$.
\end{enumerate}
The remaining simple objects of $\cD^\k$
are spectral flow twist of the lower-bounded ones, that is modules for $\ell \in \ZZ$ of the form 
\begin{equation*}
\sigma^\ell(\cL_{r,0}),\qquad\sigma^\ell(\cD_{r,s}^+),\qquad\sigma^\ell(\cD_{r,s}^-),\qquad\sigma^\ell(\cE_{\lambda,\Delta_{r,s}})
\end{equation*} 
There are some identifications
\begin{align}\label{eqn:HW-as-SF-of-LW}
\cD_{r,s}^+\cong\begin{cases}
    \sigma(\cD^-_{u-r,v-s-1}) & \text{if}\,\,\,1\leq s\leq v-2\\
    \sigma^2(\cD^-_{r,v-1}) & \text{if}\,\,\, s=v-1\\
\end{cases},\qquad \cL_{r,0}\cong\sigma(\cD_{u-r,v-1}^-)
\end{align}
for $1\leq r\leq u-1$. As a result, any simple object of $\cD^\k$ is isomorphic to exactly one of the following:
\begin{equation*}
    \sigma^\ell(\cD_{r,s}^+),\qquad\sigma^\ell(\cE_{\lambda,\Delta_{r,s}})
\end{equation*}
for integers $1\leq r\leq u-1$ and $1\leq s\leq v-1$, $\lambda\notin\pm\lambda_{r,s}+2\ZZ$, and $\ell\in\ZZ$.

It is also shown in \cite{ACK} that every simple object of $\cD^\k$ has a projective cover. The simple modules $\sigma^\ell(\cE_{\lambda,\Delta_{r,s}})$ are already projective in $\cD^\k$, while
logarithmic projective covers of the highest-weight modules were constructed \cites{A,ACK}. Let $\cP_{r,s}$ denote the projective cover of $\cD_{r,s}^+$ for $1\leq r\leq u-1$ and $1\leq s\leq v-1$, so that 
$\sigma^\ell(\cP_{r, s})$ is the projective cover (and also injective hull) of $\sigma^\ell(\cD^+_{r, s})$. Then there are non-split short exact sequences
\begin{align}
         0 \longrightarrow \sigma^{\ell+1}(\cE^-_{u-r, v-s-1}) \longrightarrow 
        &\, \sigma^{\ell}(\cP_{r, s}) \longrightarrow  \sigma^{\ell}(\cE^-_{u-r, v-s}) \longrightarrow 0, \ \  \ 1\leq s\leq v-2 \\
         0 \longrightarrow \sigma^{\ell+2}(\cE^-_{r, v-1}) \longrightarrow 
        &\,\sigma^{\ell}(\cP_{r, v-1}) \longrightarrow  \sigma^{\ell}(\cE^-_{u-r, 1}) \longrightarrow 0 
\end{align}
for $1\leq r\leq u-1$, $1\leq s\leq v-1$, and $\ell\in\ZZ$.

\subsection{The monoidal structure}\label{subsec:sl2monoidal}

We continue with the monoidal structure

\begin{theorem}
    For admissible levels $\k$, we have the following results
    \begin{enumerate}
        \item \cite{Cr24}: $\cD^\k$ admits the braided tensor category structure with ribbon twist of \cite{HLZ06} for categories of strongly $\CC$-graded generalized modules for a vertex operator algebra. 
        \item \cites{CMY1, NORW24}: $\cD^\k$ is a ribbon category. 
        \item \cite{Cr3}: The Verlinde conjecture of \cites{CR1, CR2} for $\cD^\k$ is true. 
    \end{enumerate}
\end{theorem}
In particular all fusion rules are known, mainly from \cites{Cr3, NORW24}.
The tensor product respects spectral flow, that is for any objects $M_1$, $M_2$ of $\cD^\k$ and $\ell_1,\ell_2\in\ZZ$, 
   \[
   \sigma^{\ell_1}(M_1)\boxtimes\sigma^{\ell_2}(M_2)\cong\sigma^{\ell_1+\ell_2}(M_1\boxtimes M_2).
   \]
If the conformal weight of a module $M$ lies in the coset $\Delta + \ZZ$ and the weight support is in $\lambda  + 2\ZZ$, then the conformal weight of $\sigma^\ell(M)$ lies in the coset $\Delta + \frac{\ell \lambda}{2}+ \frac{\ell^2 k }{4}$.

Theorem 4.7 of \cite{Cr24} are the following fusion rules
\begin{align}\nonumber
\mathcal{L}_{r}^{\k} \boxtimes \sigma^{\ell}(\mathcal{L}_{r'}^{\k}) &\cong \bigoplus_{r''=1}^{u-1} N_{r,r'}^{u,r''} \sigma^{\ell}(\mathcal{L}_{r''}^{\k}) \\ \nonumber
\mathcal{L}_{r}^{\k} \boxtimes \sigma^{\ell}(\mathcal{D}_{r',s}^{\pm}) &\cong \bigoplus_{r''=1}^{u-1} N_{r,r'}^{u,r''} \sigma^{\ell}(\mathcal{D}_{r'',s}^{\pm}) \\ \nonumber
\mathcal{L}_{r}^{\k} \boxtimes \sigma^{\ell}(\mathcal{E}_{\lambda;\Delta_{r',s}}) &\cong \bigoplus_{r''=1}^{u-1} N_{r,r'}^{u,r''} \sigma^{\ell}(\mathcal{E}_{r-1+\lambda;\Delta_{r'',s}}) \\
\mathcal{L}_{r}^{\k} \boxtimes \sigma^{\ell}(\mathcal{P}_{r',s}^{\k}) &\cong \bigoplus_{r''=1}^{u-1} N_{r,r'}^{u,r''} \sigma^{\ell}(\mathcal{P}_{r'',s}^{\k}) \nonumber
\end{align}
for all $1 \le r, r' \le u - 1$, $1 \le s \le v - 1$ and $\lambda \in \mathbb{C}$ with
\begin{equation}\nonumber
N_{t,t'}^{u,t''} =
\begin{cases}
1 & \text{if } |t-t'|+1 \le t'' \le \min\{t+t'-1,\, 2u-t-t'-1\} \text{ and } t+t'+t'' \text{ odd}, \\
0 & \text{else}.
\end{cases}
\end{equation}
The conformal weight of the spectrally flown modules is given in (2.11) of \cite{CR1}. This together with the fusion rules above allows us to compute that the modules $\sigma^{\ell}(\mathcal{L}_{r}^{\k}), \sigma^{\ell}(\mathcal{D}_{r,s}^{\pm}),  \sigma^{\ell}(\mathcal{E}_{\lambda;\Delta_{r,s}}), \sigma^{\ell}(\mathcal{P}_{r,s}^{\k})$ centralize all $\mathcal{L}_{r'}^{\k}$ for odd $r'$ if and only if $r \in \{ 1, u-1 \}$, in particular the variation $\cD^\k_R$ of $\cD^\k$ can be characterized as the centralizer of all $\mathcal{L}_{r'}^{\k}$ for odd $r'$.

\subsection{The free field realization}\label{subsec:sl2freefield}

This realization is due to Adamovic \cite{A}.
We recall the presentation of \cite{CMY1}.  

\bigskip

Let $\Vir_{c_\k}$ be the simple Virasoro vertex operator algebra of central charge $c_\k =1 - \frac{6(k+1)^2}{k+2}$ 
as in Section \ref{subsec:CartanVir}. It is strongly rational \cite{W}, its simple modules are the $\VirPhi_{r,s}$ of lowest conformal weight $h_{r,s}=\frac{1}{4uv}((su-rv)^2-(u-v)^2)$ for $1\leq r \leq u-1$ and $1\leq s \leq v-1$ and identifications $\VirPhi_{r,s} = \VirPhi_{u-r,v-s}$. Recall that the fusion rules are
\begin{equation}\label{eqn:Vir-fus-rules}
    \VirPhi_{r,s} \otimes \VirPhi_{r',s'} \cong  \bigoplus_{\stackrel{r'' = \vert r-r'\vert+1}{r+r'+r''\,\text{odd}}}^{\min(r+r'-1,2u-r-r'-1)}\bigoplus_{\stackrel{s'' = \vert s-s'\vert+1}{s+s'+s''\,\text{odd}}}^{\min(s+s'-1,2v-s-s'-1)} \VirPhi_{r'',s''}.
\end{equation}

\bigskip
As in Section \ref{subsec:CartanLattice}, 
let $L = \ZZ c + \ZZ d$ be the lattice with bilinear form
$$\begin{pmatrix} (c,c) & (c,d) \\ (d,c) & (d,d)\end{pmatrix}=\begin{pmatrix}0 & 2 \\2 & 0\end{pmatrix}$$
and let $V_L=\pi^{c,d}\otimes \CC[L]$ be the corresponding lattice vertex algebra, where $\pi^{c,d}$ is the rank $2$ Heisenberg vertex operator algebra generated by $c$ and $d$ and $\CC[L]$ is the group algebra of the lattice $L$. Define the vertex subalgebra $\Pi(0) = \pi^{c,d} \otimes \CC[\ZZ c] \subseteq V_L$.
\begin{theorem}[\cite{A} Theorem 5.5]\label{thm: conformal embedding}
    If $\k$ is a non-integral admissible level of $\sl_2$, then there is an injective vertex algebra homomorphism $L_\k(\sl_2) \hookrightarrow \Vir_{c_\k} \otimes \Pi(0)$ sending $h_{-1}\vac$ to $2\mu_{-1}\vac$, where $\mu=\frac{\k}{4}c+\frac{1}{2}d$. This embedding is conformal if the conformal vector of $\Pi(0)$ is taken to be $\frac{1}{2}(c_{-1}d_{-1}-d_{-2}+\frac{\k}{2}c_{-2})\vac$.
\end{theorem}

Recall that $\cC^\k$ is defined as the category of generalized vertex algebra modules over $\Vir_{c_\k} \otimes \Pi(0)$. Then the embedding in the previous theorem gives a exact faithful lax monoidal restriction functor 
\[
\Res: \cC^\k \rightarrow  \cD^\k
\]
We parametrize as anticipated in Section \ref{subsec:CartanLattice} the simple modules for $\Pi(0)$ as $\Pi_\ell(\lambda)$
for $\ell \in \ZZ$ and $\lambda \in \CC$, where the label $\lambda$ indicates the $\mu_0$-eigenvalues and by construction $\Pi_\ell(\lambda) \cong \Pi_\ell(\lambda+1)$.
Thus $\cC^\k=$ has simple modules $\VirPhi_{r,s} \otimes \Pi_\ell(\lambda)$
for $1\leq r \leq u-1$, $1\leq s \leq v-1$, $\ell \in \ZZ$, and $\lambda \in \mathbb C/\mathbb Z$. Then the restriction functor satisfies (see \cite{A} Section 7 as well as \cite{CMY1}):
\begin{equation}\label{eqn:FF-affine-sl2-correspondence}
    \Res(\VirPhi_{r,s} \otimes \Pi_{\ell -1}(\lambda)) \cong  \begin{cases}
    \sigma^\ell(\cE^-_{u-r, v-s}) & \text{if}\,\, \lambda = \nu_{r, s} \ \text{(atypical)} \\
    \sigma^\ell(\cE^-_{r, s}) & \text{if}\,\,\lambda = \nu_{u-r, v-s} \ \text{(atypical)}\\
    \sigma^\ell(\cE_{2\lambda-\k, \Delta_{r, s}}) & \text{otherwise \ (typical)} \\
\end{cases} , 
\end{equation}
for 
\[
\nu_{r, s} = \frac{1}{2}(r-1-t(s-1))).
\]
The modules $\sigma^\ell(\cD^+_{r, s})$ for $1\leq r\leq u-1$ and $1\leq v\leq s-1$ appear as both submodules and quotients of simple $\Vir_{c_\k} \otimes \Pi(0)$-modules. Namely, 
\begin{equation*}
       \Res(\VirPhi_{r,s} \otimes \Pi_{\ell -1}(\nu_{r, s})) \cong \sigma^\ell(\cE^-_{u-r, v-s}) \twoheadrightarrow \sigma^\ell(\cD^+_{r, s}),
\end{equation*} 
and using \eqref{eqn:HW-as-SF-of-LW},
\begin{equation*}
\sigma^\ell(\cD^+_{r, s}) \hookrightarrow \begin{cases}
    \sigma^{\ell+1}(\cE^-_{u-r, v-s -1}) \cong \Res(\VirPhi_{r,s+1} \otimes \Pi_{\ell}(\nu_{r, s+1})) & \text{if}\,\,\,1\leq s \leq v-2 \\
    \sigma^{\ell+2}(\cE^-_{r, v-1}) \cong \Res(\VirPhi_{u-r,1} \otimes \Pi_{\ell+1}(\nu_{u-r, 1})) & \text{if}\,\,\,s = v-1 \\
\end{cases} .
\end{equation*}
The case
$\ell=r=s=1$  of \eqref{eqn:FF-affine-sl2-correspondence} is
\begin{align*}
    \Res(\Vir_{c_\k}\otimes\Pi(0)) = \sigma(\cE_{u-1,v-1}^-),
\end{align*}
and thus by \eqref{eqn:exact-seq-E+E-} and \eqref{eqn:HW-as-SF-of-LW}, there is a non-split short exact sequence
\begin{equation}\label{eqn:A-as-Lk(sl2)-mod}
    0\longrightarrow L_\k(\sl_2) \longrightarrow \Res(\Vir_{c_\k}\otimes\Pi(0))\longrightarrow \sigma(\cD_{1,1}^+)\longrightarrow 0.
\end{equation}
The vertex algebra $\Vir_{c_\k}\otimes\Pi(0)$ itself can be identified with a commutative algebra in $\cD^\k$ \cite{HKL}  and we denote this algebra by $A$. Its category of local modules is braided tensor equivalent to $\cC^\k$ \cite{CKM}.
We now turn to its category of modules.

\subsection{The category of \texorpdfstring{$A$-modules}{A-modules}}
\label{subsec:sl2Amodules}

We want to study the category of $A$-modules in $\cD^\k$, denoted by $\cD^\k_A$. 
This section is based again on \cite{CMY1} which in turn used the main results of \cite{CMSY24}. The main results are as follows
\begin{enumerate}
    \item Every irreducible object in $\cD^\k_A$ is already local, that is in $\cC^\k$ \cite{CMY1} Theorem 4.12.
    \item Thus the assumptions of Theorem 3.14 of \cite{CMSY24} are satisfied and hence $\cD^\k_A$ is rigid. The dual of an object $X$ will be denoted by $X^*$.
    \item A conformal vertex algebra has a duality structure given by the contragredient dual \cite{ALSW21}. We denote the contragredient dual of an object $M$ by $M'$. Theorems 4.15 and 4.16 of \cite{CMY1} say that for every simple object $M$ in $\cD^\k$ one has $\Ind(M') \cong \Ind(M)^*$ with $\Ind: \cD^\k \rightarrow \cD^\k_A$ the usual induction functor (the left adjoint of $\Res$).
\item Theorem 4.14 of \cite{CMY1} is the description of the projective objects in $\cD^\k_A$ for $v\geq 3$: For any $1\leq r\leq u-1$, $1\leq s\leq v-1$, $\lambda+\ZZ \in \CC/\ZZ$, and $\ell \in \ZZ$,
    there is a projective and injective module $R_{r,s, \lambda, \ell}$ in $\cD^\k_A$ with the following Loewy diagram, 
    
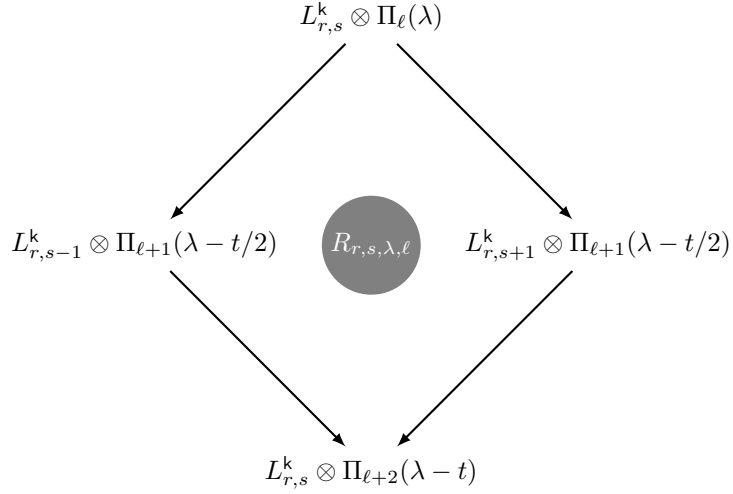
\begin{center}
\begin{tikzpicture}[scale=1]
\node (top) at (0,3) [] {$ \VirPhi_{r, s} \otimes \Pi_\ell(\lambda)$};
\node (left) at (-3,0) [] {$\VirPhi_{r, s - 1} \otimes \Pi_{\ell+1}(\lambda -t/2)$};
\node (right) at (3,0) [] {$\VirPhi_{r, s + 1} \otimes \Pi_{\ell+1}(\lambda -t/2)$};
\node (bottom) at (0,-3) [] {$\VirPhi_{r, s} \otimes \Pi_{\ell+2}(\lambda - t)$};
\draw[->, thick] (top) -- (left);
\draw[->, thick] (top) -- (right);
\draw[->, thick] (left) -- (bottom);
\draw[->, thick] (right) -- (bottom);
\node (label) at (0,0) [circle, inner sep=2pt, color=white, fill=black!50!] {$R_{r, s, \lambda, \ell}$};
\end{tikzpicture}
\end{center}
with the convention that $\VirPhi_{r, 0} = \VirPhi_{r, v} = 0$ and $t = k+2 = \frac{u}{v}$.

The case $v=2$ is similar but much simpler, see section 4.2 of  \cite{CMY1}, i.e. the Loewy diagram of the projectives  is 
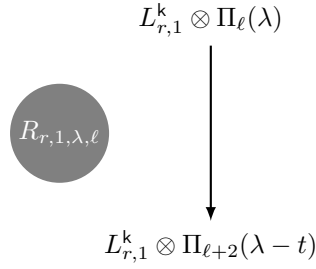
\begin{center}
\begin{tikzpicture}[scale=1]
\node (top) at (0,3) [] {$ \VirPhi_{r, 1} \otimes \Pi_\ell(\lambda)$};
\node (bottom) at (0, 0) [] {$\VirPhi_{r, 1} \otimes \Pi_{\ell+2}(\lambda - t)$};
\draw[->, thick] (top) -- (bottom);
\node (label) at (-2,1.5) [circle, inner sep=2pt, color=white, fill=black!50!] {$R_{r, 1, \lambda, \ell}$};
\end{tikzpicture}
\end{center}
\item Corollary 3.10 together with Theorem 3.4 of  \cite{CMY1} verify a certain mild left exactness condition, the third condition of Theorem 3.20 of \cite{CMSY24}. The other conditions of that Theorem are that the induction functor commutes with duality and that $\cD^\k_A$ is rigid. Thus the conclusion of that Theorem applies and this conclusion is that the Schauenburg functor
\[
\cS: \cD^\k \rightarrow \cZ(\cD^\k_A)
\]
is fully faithful. For us this is most important.
\item In \cite{CMY1} the Schauenburg functor being fully faithful was key as it allowed to apply Theorem 3.21 of \cite{CMSY24} saying that $\cD^\k$ is rigid. 
\item Let $N = R_{1, 1, 0, 0}$ be the projective cover of $A$. It satisfies 
\[
N \otimes (\VirPhi_{r, s} \otimes \Pi_\ell(\lambda))  \cong R_{r, s, \lambda, \ell}
\]
that is the tensor product of a simple object $X$ with $N$ is the projective cover of $P_X$ of the simple object $X$.
\item From the Loewy diagram of $N$ we immediately see that 
\[
\Ext^1(A, X) \cong \begin{cases} \mathbb C & X \cong \VirPhi_{1,2} \otimes \Pi_1(-t/2) \ \text{and} \ v \geq 3 \\
\mathbb C & X \cong \VirPhi_{1,1} \otimes \Pi_2(-t) \ \text{and} \ v = 2 \\ 
0 & \ \text{otherwise} 
\end{cases}
\]
and similarly and since the projectives are also injectives
\[
\Ext^1(X, A) \cong \begin{cases} \mathbb C & X \cong \VirPhi_{1,2} \otimes \Pi_{-1}(t/2) \ \text{and} \ v \geq 3 \\
\mathbb C & X \cong \VirPhi_{1,1} \otimes \Pi_{-2}(t) \ \text{and} \ v = 2 \\ 
0 & \ \text{otherwise} 
\end{cases}
\]
\item From Theorem 4.16 of \cite{CMY1} one observes that $R_{1, 1, 0, 0}=N$ and $\Ind(\sigma(\cE_{1,1}^+))$ have the same simple composition factors and by Frobenius reciprocity one sees that the socle of $N$ is the only simple object in $\cD^\k_A$ 
that admits a non-zero homomorphism from $\sigma(\cE_{1,1}^+)$. Hence 
\begin{corollary}\label{cor:InductionLieDual}
 The induction of $\cE_{1,1}^+$ is the projective cover of the tensor unit in $\cD^\k_A$
 $$\Ind_A(\sigma(\cE_{1,1}^+))\cong R_{1, 1, 0, 0}$$
\end{corollary}

\item The kernel of the projection $N \rightarrow A$ is an indecomposable object $K$ satisfying the non-split exact sequence
\[
0 \rightarrow \VirPhi_{1, 1} \otimes \Pi_2(-t) \rightarrow  K \rightarrow \VirPhi_{1, 2} \otimes \Pi_1(- t/2) \rightarrow 0
\]
with the convention as above, that $\VirPhi_{1, 2}=0$ if $v =2$.
By Lemma 4.11 of \cite{CMY1} this implies:
\begin{corollary}\label{cor:InductionA}
 The induction of the object $A$ itself is
 $$\Ind_A(A)\cong A\oplus K$$
\end{corollary}
\end{enumerate}

\SimonRem{Thomas says:
$A^*=\sigma^{-1} \mathcal{E}_{1,1}^-$}

\section{Convolution equivalence}\label{sec:convolution}

This section is essentially a summary of some results of \cite{CGNS} Section 6, which works out convolutions of vertex algebras using the Lie algebra $\mathfrak{gl}_1$. The generalization to simple Lie algebras is also possible \cite{CLNS}.
Let $V$ be a vertex algebra that contains a Heisenberg vertex algebra $\pi^J$. Let $J$ be the corresponding Heisenberg field which we assume to be non-degenerate. 
We also assume that $\pi$ acts semisimply on $V$, so that $V$ is graded by Heisenberg weight. We normalize $J$ in such a way that this is an integer grading, 
\[
V = \bigoplus_{n \in \mathbb Z} V_n, \qquad V_n = \{ v \in V | J_0v =nv \}.
\]
Each component must be of the form $C_n  \otimes \pi^J_n$ for some module for the coset $C_0 = \text{Com}(\pi^J, V)$, so that
\[
V = \bigoplus_{n \in \mathbb Z} C_n  \otimes \pi^J_n.
\]
Let $\ell$ be the level of $J$, that is 
\[
J(z)J(w) = \frac{\ell}{(z-w)^2}.
\]
Let $\pi^K$ be another Heisenberg vertex algebra with Heisenberg field $K$. Assume that $K$ is normalized such that its level is $-\ell$. 
Consider the semi-infinite cohomology $H^{\mathrm{rel},p}(V)$ of $\mathfrak{gl}_1((z))$ with coefficients in $V$, taken relative to $\mathfrak{gl}_1$. 
It satisfies 
\begin{lemma} \cites{FGZ, CFL} \label{lem:vanish}
For any $p \in \mathbb{Z}$ and $\lambda, \mu \in \mathfrak h^*$, we have
\[
H^{\mathrm{rel},p}(\pi^{J}_\lambda \otimes \pi^{K}_\mu)
   = \delta_{p,0} \, \delta_{\lambda+\mu,0} \, 
     \mathbb{C}.
\]
\end{lemma}
Let $m$ be an integer and $V_{\sqrt{m}\mathbb Z}$ the lattice VOA of the lattice $\sqrt{m}\mathbb Z$. It decomposes as
\[
V_{\sqrt{m}\mathbb Z} = \bigoplus_{n \in \mathbb Z} \pi^L_n 
\]
with $\pi^L$ a Heisenberg vertex algebra whose Heisenberg field $L$ has level $\frac{1}{m}$.
Consider $V \otimes V_{\sqrt{m}\mathbb Z}$. It has a Heisenberg subalgebra of rank two generated by $J, L$. Let $A = J - L$ and $B =\frac{1}{m}J + \ell L$, so that $A$ and $B$ are orthogonal to each other. Then 
\[
V \otimes V_{\sqrt{m}\mathbb Z} = \bigoplus_{n_1, n_2 \in \mathbb Z} C_{n_1} \otimes \pi^A_{n_1 - n_2} \otimes \pi^B_{\frac{n_1}{m} + \ell n_2}. 
\]
The field $A$ has level $\ell + \frac{1}{m}$ and so if we add another Heisenberg $\pi^K$ of level $-\ell - \frac{1}{m}$ and take cohomology we obtain
\[
H^{\mathrm{rel},0}(V \otimes V_{\sqrt{m}\mathbb Z} \otimes \pi^{K}) =
\bigoplus_{n_1, n_2 \in \mathbb Z} C_{n_1} \otimes H^{\mathrm{rel},0}(\pi^A_{n_1 - n_2} \otimes \pi^{K}) \otimes \pi^B_{\frac{n_1}{m} + \ell n_2} = 
\bigoplus_{n\in \mathbb Z} C_{n} \otimes \pi^B_{\frac{n}{m} + \ell n}
\]
If we replace $B$ by $D = \frac{B}{\frac{1}{m}+\ell}= \frac{mB}{1+m\ell}$, then this becomes
\[
H^{\mathrm{rel},0}(V \otimes V_{\sqrt{m}\mathbb Z} \otimes \pi^{K}) =
\bigoplus_{n\in \mathbb Z} C_{n} \otimes \pi^D_{n}
\]
where the level of $D$ is $\ell' = \frac{\ell}{1+m\ell}$, i.e.
\[
 \frac{1}{\ell'} = \frac{1}{\ell} + m.
\]
In other words, the cohomology has the only effect of changing the level and hence conformal weight of Fock modules. 
This operation is called ($\mathfrak{gl}_1$-)convolution operation, that is
\[
*_m(V):= H^{\mathrm{rel},0}(V \otimes V_{\sqrt{m}\mathbb Z} \otimes \pi^{K}).
\]
It defines a $\mathbb Z$-groupoid on the category of vertex algebras that have non-degenerate Heisenberg subalgebras, since $*_m \circ *_n \cong  *_{n+m}$ and since $*_0 \cong \text{Id}$.
We call two vertex algebras $\mathfrak{gl}_1$-convolution equivalent if they are in the same convolution orbit. 
The simplest example of a convolution operation are just lattice VOAs, $*_m (V_{\sqrt{n}\mathbb Z}) \cong V_{\sqrt{n+m}\mathbb Z}$. However the case $n+m =0$ is degenerate. This doesn't appear if one adds an additional Heisenberg. In that case one has 
\[
*_m (V_{\sqrt{n}\mathbb Z} \otimes \pi) \cong  \begin{cases} V_{\sqrt{n+m}\mathbb Z} \otimes \pi & n+m \neq 0 \\
\Pi(0) & n+m =0
\end{cases}
\]
In the case of $V = V^\k(\mathfrak{sl}_2)$ one has 
\[
*_1(V^\k(\mathfrak{sl}_2)) \cong W^\ell(\mathfrak{sl}_{2|1}), \qquad
*_1(V^\k(\mathfrak{sl}_2)) \cong \text{Com}(V^m(\mathfrak{sl}_2), V^m(\mathfrak{sl}_{2|1})
\]
where the levels satisfy the duality relations $(k+2)(\ell+1)=1$ and $(k+1)(m+1)=1$.
The first correspondence firmates under the name Kazama-Suzuki duality and so one sometimes calls these operations also generalizations of the Kazama-Suzuki duality.
In \cite{CGNS} it was shown that these operation provide block-wise equivalences of representation categories. It is also a lax tensor functor and it usually does not preserve the braiding. However if one restricts to modules that are integer graded by Heisenberg weight then it also preserves the braiding. 

This relation between representation categories can also be seen via coset constructions. The point is that also \cite{CGNS}
\[
*_m(V) \cong \text{Com}(\pi^A, V \otimes V_{\sqrt{m}\mathbb Z})
\]
This means $V \otimes V_{\sqrt{m}\mathbb Z}$ is a conformal extension of $*_m(V) \otimes \pi^A$ and conversely $*_m(V) \otimes V_{\sqrt{-m}\mathbb Z}$ is a conformal extension of $V \otimes \pi^{\widetilde A}$ for another Heisenberg VOA $\pi^{\widetilde A}$.
Precise formulas are obtained by introducing the $\pi^\k_r$-twisted convolution for $r \in \mathbb Z$:
\begin{align*} \nonumber 
     *_{m, r}(V) :&=  H^{\mathrm{rel},0}(V \otimes V_{\sqrt{m}\mathbb Z} \otimes \pi^{K}_r)\\
     &= \bigoplus_{n_1, n_2 \in \mathbb Z} C_{n_1} \otimes H^{\mathrm{rel},0}(\pi^A_{n_1 - n_2} \otimes \pi^{K}_r) \otimes \pi^B_{\frac{n_1}{m} + \ell n_2}\\
     &= \bigoplus_{n\in \mathbb Z} C_{n} \otimes \pi^B_{\frac{n}{m} + \ell (n+r)} \nonumber 
 = \bigoplus_{n\in \mathbb Z} C_{n} \otimes \pi^D_{n + \frac{\ell m  r}{1+m\ell}}  
\end{align*}
From this one gets the decomposition
\begin{align*}
    V \otimes V_{\sqrt{m}\mathbb Z} &= \bigoplus_{n_1, n_2 \in \mathbb Z} C_{n_1} \otimes \pi^A_{n_1-n_2} \otimes \pi^B_{\frac{n_1}{m} +\ell n_2} \\
    &= \bigoplus_{n_1, n_2 \in \mathbb Z} C_{n_1} \otimes \pi^A_{n_1-n_2} \otimes \pi^B_{\frac{n_1}{m} \ell n_1 -\ell (n_1 -n_2} \\ 
    &= \bigoplus_{n, r \in \mathbb Z} C_{n} \otimes \pi^B_{\frac{n}{m} \ell n + \ell r}  \otimes \pi^A_{-r} \\
    &= \bigoplus_{n, r \in \mathbb Z} *_{m, r}(V)  \otimes  \pi^A_{-r}.
\end{align*}
Now, if $*_m(V)$ admits a vertex tensor category $\cD^m$ containing the $*_{m, r}(V)$ for $r \in \mathbb Z$, then by the theory of vertex superalgebra extensions $V \otimes V_{\sqrt{m}\mathbb Z}$ can be identified with a commutative superalgebra in the direct limit completion of $\cD^m$ \cites{HKL, CKL, CMYdirect}. As a consequence the category of local modules for this commutative superalgebra is isomorphic to a corresponding vertex tensor category of modules for $V \otimes V_{\sqrt{m}\mathbb Z}$ \cite{CKM24}. 
The category of modules for $_{\sqrt{m}\mathbb Z}$ is just the category of $\mathbb Z/m\mathbb Z$-graded vector spaces with quadratic form coming from the natural bilinear form on $\frac{1}{\sqrt{m}}\mathbb Z$. Thus the category of modules of $V \otimes V_{\sqrt{m}\mathbb Z}$ is the Deligne product of a vertex tensor category $\cD$ of $V$-modules with this category of graded vector spaces. 
The $*_{m, r}(V)$ are simple currents \cite{CKLR} and hence $\cD^m$ is graded by their monodromy. Monodromy provides a block decomposition. $\cD$ has a similar block decomposition and the block-wise equivalence is precisely between those blocks, see \cite{CGNS} Theorems 5.1 and 6.3. The block of trivial monodromy of $\cD$ contains all objects $M$, that are integer graded by $J$-Heisenberg weight, that is objects $M$ of the form
\[
M = \bigoplus_{n \in \mathbb Z} M_n \otimes \pi^J_n.
\]
The same argument as for $V$ gives 
\begin{align*} \nonumber 
     *_{m, r}(M) :&=  H^{\mathrm{rel},0}(M \otimes V_{\sqrt{m}\mathbb Z} \otimes \pi^{K}_r)\\
     &= \bigoplus_{n\in \mathbb Z} M_{n} \otimes \pi^B_{\frac{n}{m} + \ell (n+r)} = \bigoplus_{n\in \mathbb Z} M_{n} \otimes \pi^D_{n + \frac{\ell m  r}{1+m\ell}}  
\end{align*}
and 
\begin{align*}
    M \otimes V_{\sqrt{m}\mathbb Z} &= \bigoplus_{n, r \in \mathbb Z} *_{m, r}(M)  \otimes  \pi^A_{-r}.
\end{align*}
Let $\cD_{\mathbb Z}$ and $\cD^m_{\mathbb Z}$ be the subcategories of $\cD$ and $\cD^m$ that are integer graded by the Heisenberg fields $J$ respectively $D$. As any intertwiner for $V, *_m(V)$ is necessarily a Heisenberg intertwiner and those preserve the lattice $\mathbb Z$ the categories  $\cD_{\mathbb Z}$ and $\cD^m_{\mathbb Z}$ are braided tensor subcategories of 
$\cD$ and $\cD^m$. Note that of $M$ is an object in $\cD_{\mathbb Z}$ then $*_m(M)$ is in $\cD^m_{\mathbb Z}$. 
The functor
\[
\cD^m_{\mathbb Z} \rightarrow \cD^m_{\mathbb Z} \boxtimes \Vect_{\mathbb Z/m\mathbb Z}, \qquad  X \mapsto \Ind(X \otimes \pi^D) 
\]
has the property that $\Ind(*_m(M)\otimes \pi^D)\cong M \otimes V_{\sqrt{m}\mathbb Z}$. This follows from the theory of simple current extensions, that has been studied in many works, see in particular \cites{CMYsc1, CMYsc2}. From this one sees that the block-wise equivalence restricted to $\cD_{\mathbb Z}$ and $\cD^m_{\mathbb Z}$\footnote{The grading by coset $\mathbb Z+\lambda$ corresponding to Heisenberg weights is another (and obvious) block decomposition.} is an equivalence by combining $\Ind( \bullet \otimes \pi^D)$ with the embedding of $\cD$ in $\cD \boxtimes \Vect_{\mathbb Z/m\mathbb Z}$.
Assume now that there is another commutative algebra $F$ ($F$ for free field) in $\cD_{\mathbb Z}$. Then via the equivalence there is a corresponding commutative algebra $*_m(F)$ in  $\cD^m_{\mathbb Z}$. Let $\cD_F$ and $\cD^m_{*_m(F)}$ be the corresponding categories of algebra modules. The categories of local modules will just be denoted by $\cC$ and $\cC^m$. Then the subcategories that are integer graded must be tensor equivalent $\cD_{F, \mathbb Z} \cong \cD^m_{*_m(F), \mathbb Z}$ and braided tensor equivalent $\cC_{\mathbb Z} \cong \cC^m_{\mathbb Z}$. In particular any Nichols algebra $N$ in $\cC_{\mathbb Z}$ can equally well be regarded as a Nichols algebra $*_m(N) \cong N$ in $\cC^m_{\mathbb Z}$.
 In particular if we can prove that $\cD \cong \YD{N}(\cC)$, then automatically $\cD^m \cong \YD{N}(\cC^m)$. 
 We thus say, that two categories $\cD, \cE$ are $\mathfrak{gl}_1$-convolution equivalent if 
 \[
 \cD \cong \YD{N}(\cC) \qquad \text{and}
 \qquad 
 \cE \cong \YD{N}(\cC^m)
 \]
 for some $m \in \mathbb Z$.

In our case we have the following picture for the universal vertex algebras, maybe best phrased for $k \in \mathbb C\setminus \mathbb Q$ and $m, \ell$ related to $k$ via
\[
\frac{1}{m+2}+ \frac{1}{k+2} = 1, \qquad (\ell+1)(k+2) =1.  
\]
Let $C^\k = \text{Com}(\pi, V^\k(\mathfrak{sl}_2))$. Then the relations are visualized by the following diagram where the red dashed arrows indicate conformal embeddings. 
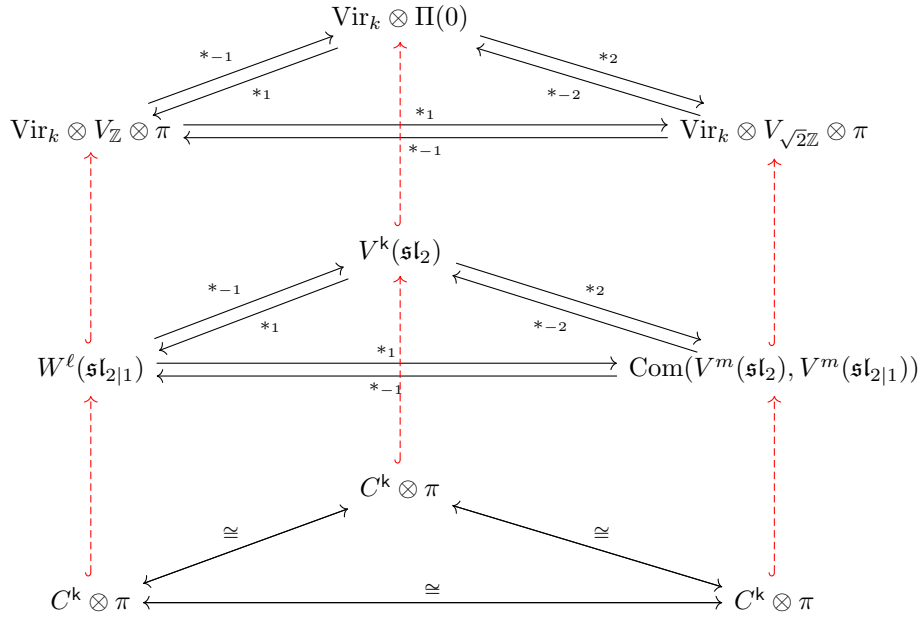
\[
\begin{tikzcd}[row sep=6ex, column sep=12ex]
& \Vir_{k} \otimes \Pi(0) \arrow[dl, shift left, "*_1"]  \arrow[dr, shift left, "*_2"]  \\
\Vir_{k} \otimes V_{\mathbb Z} \otimes \pi \arrow[ur, shift left, "*_{-1}"] \arrow[rr, shift left, "*_1"] &&
\Vir_{k} \otimes V_{\sqrt{2}\mathbb Z} \otimes \pi \arrow[ul, shift left, "*_{-2}"]\arrow[ll, shift left, "*_{-1}"] \\
& V^{\k}(\mathfrak{sl}_2)  \arrow[uu, hook, red, dashed] \arrow[dl, shift left, "*_1"]  \arrow[dr, shift left, "*_2"] \\ 
W^{\ell}(\mathfrak{sl}_{2|1}) \arrow[uu, hook, red, dashed] \arrow[ur, shift left, "*_{-1}"] \arrow[rr, shift left, "*_1"] &&
\text{Com}(V^m(\mathfrak{sl}_2), V^m(\mathfrak{sl}_{2|1})) \arrow[uu, hook, red, dashed] \arrow[ul, shift left, "*_{-2}"]\arrow[ll, shift left, "*_{-1}"] \\
& C^{\k} \otimes \pi  \arrow[uu, hook, red, dashed] \arrow[dl]  \arrow[dr, "\cong"] \\ 
C^{\k} \otimes \pi \arrow[uu, hook, red, dashed] \arrow[ur, "\cong"] \arrow[rr, "\cong"] &&
C^{\k} \otimes \pi \arrow[uu, hook, red, dashed] \arrow[ul]\arrow[ll] \\
\end{tikzcd}
\]
If $k = -2 +\frac{u}{v}$ is an admissible level for $\mathfrak{sl}_2$, then the same diagram still works, but replacing $V^m(\mathfrak{sl}_{2|1}), W^{\ell}(\mathfrak{sl}_{2|1}), V^{\k}(\mathfrak{sl}_2), C^\k$ by their simple quotients. 

\section{The equivalence \texorpdfstring{$(\catD^\k)_A\cong \cB^\k$}{DA with B}}\label{sec_equivalenceBorel}

\newcommand{\catM}{\mathcal{M}}

\subsection{Algebra realizations of module categories}\label{subsec:BorelModulecat}

We now assume $\catM$ is a left module category over a tensor category $\cC$. We assume always that the action is exact in both arguments. In this section we discuss two technical representability conditions in the infinite cases, which we require later:

\begin{definition}
A module category $\catM$ over $\catC$ is said to admit inner $\catC$-Homs, if for any $M,M'\in \catM$ there exists a  object
$$\underline{\Hom}(\catM,\catM')\in \catC$$ 
with the following representing property
$$\Hom_\catM(V\otimes M,M')
=\Hom_\catC(V,\underline{\Hom}(M,M'))$$
\end{definition}

The problem here is essentially the existence of a right adjoint functor for a right exact functor. For finite $\catM,\catC$ finite, such inner $\catC$-Homs hence exist, see  \cite{DSPS19} Proposition~2.15.
 We want to make it very explicit in the case we require.
\begin{lemma}\label{lemma_finite}
Assume $\catC$ is semisimple and Hom-spaces in $\catM$ are finite-dimensional. Then the $\catC$-Hom exists and is explicitly given by 
$$\underline{\Hom}(M,M')
=\bigoplus_{V\in\catC\;\text{simple}}
\Hom_\catM(V\otimes M,M')V
$$
iff for given $M,M'\in\catM$ only  finitely many simples $V$ contribute to the direct sum. 
\end{lemma}
We remark that in other cases the $\catC$-Hom should still exist in the Ind-completion.

\begin{example}\label{exm_semisimplefinite}
Assume that there is a faithful exact $\catC$-module functor $F:\catM\to \catC$. Then the conditions in the previous lemma are met if they are met in $\catC$, that is, for given simples $W,W'$ there are only finitely many simples $V$ with $\Hom_\catC(V\otimes W,W'')\neq 0$.
\end{example}

The second representability condition is

\begin{definition}
A module category $\catM$ over $\catC$ is called \emph{representable} if there exists an algebra $R\in \catC$ such that as module categories we have an equivalence $\catM\cong \catC_R$ to the category of right modules over $R$ with the obvious left action of $\catC$.
\end{definition}

We clarify how these two conditions are connected: The following assertion is surely not surprising or new:

\begin{lemma}
Representability $\catM\cong \catC_R$  implies for closed $\catC$ the existence of inner $\catC$-Homs.
\end{lemma}
\begin{proof}
    A closed tensor category $\catC$ has by definition an inner $\underline{\Hom}(W,W')\in \catC$  
$$\Hom_\catC(V\otimes W,W')
=\Hom_\catC(V,\underline{\Hom}(W,W'))$$
Assume that $W,W'\in \catC$ carry each a right action of $R$, then we can define 
$$\underline{\Hom}_R(W,W')\subset \underline{\Hom}(W,W')$$ 
as an equalizer as follows: We first collect standard arguments: Any element in $\Hom_\catC(X,Y)$ gives by the defining property rise to a homomorphism $1\to \underline{\Hom}(X,Y)$. Hence the actions give rise to morphisms from $1$ to $\underline{\Hom}(W\otimes R ,W)$ resp. $\underline{\Hom}(W'\otimes R ,W')$. Moreover the identity on $W'\otimes R$ gives rise to a morphism from $1$ to $\underline{\Hom}(W'\otimes R,W'\otimes R)\cong \underline{\Hom}(W',\underline{\Hom}(R,W'\otimes R))$ and hence to a morphism 
$\underline{\Hom}(W,W')\to \underline{\Hom}(W\otimes R,W'\otimes R)$. Then using composition of inner $\cC$-Homs gives two morphism

\begin{center}
\begin{tikzcd}[column sep=-1.5em]
  & 
  \underline{\Hom}(W\otimes R ,W) \otimes \underline{\Hom}(W,W') 
  \ar[dr]
  & 
  \\
\underline{\Hom}(W,W') \ar[ur] \ar[dr] 
& 
& 
\underline{\Hom}(W\otimes R ,W') \\
  & 
  \underline{\Hom}(W\otimes R,W'\otimes R) \otimes \underline{\Hom}(W'\otimes R ,W')
  \ar[ur] 
  &
\end{tikzcd}
\end{center}
\bigskip

We define $\underline{\Hom}_R(W,W')$ as the equalizer of these two morphisms. It is clear that the defining property of $\underline{\Hom}(W,W')$ descends to the property defining $\underline{\Hom}_R(W,W')$ as the inner $\catC$-Hom in the module category of $R$-modules
$$\Hom_\catM(V\otimes W,W')
=\Hom_\catC(V,\underline{\Hom}_R(W,W'))$$
\end{proof}

Conversely, \cite{DSPS19} Section 2 establishes the representability from the existence of inner $\catC$-Homs, in the finite case, and we simply spell out explicitly the conditions that are not automatic in the infinite case:

We choose now an object $P\in\catM$ and consider the endomorphism algebra $R:=\underline{\Hom}(P,P)$ in $\catC$. We get  natural functors
$$G:\catC_R
\xleftarrow{\underline{\Hom}_\catC(P,-)}\catM$$
$$F:\catC_R\xrightarrow{(-)\otimes_R P}
\catM$$
The following proof requires no finiteness:

\begin{theorem}[\cite{DSPS19} Proposition 2.24]\label{thm:DSPS}
Suppose that $\catM$ is a $\catC$ module category with action exact in both arguments, which admits inner $\catC$-homs. Let $P\in\cM$ any chosen object and $R,G,F$ as in the preceding paragraph.
\begin{enumerate}[(a)]
\item $F,G$ are adjoint.
\item Suppose $G$ can be enhanced to a $\catC$-module functor. Then $F,G$ are an equivalence of categories if $P$ is a $\catC$-projective $\catC$-generator, that is:
\begin{enumerate}
    \item[(b1)] $\underline{\Hom}(P,-)$ is also right exact.
\item[(b2)] $\underline{\Hom}(P,-)$ is faithful. Equivalently by \cite{DSPS19} Lemma 2.22: for every $M\in\catM$ there exists a $V\in \catC$ and a surjection $V\otimes P\to M$.

\end{enumerate}
\end{enumerate}
\end{theorem}

If $\catC$ is rigid then the enhancement to a module functor is guaranteed by \cite{DSPS19}  Corollary 2.13. Regarding the conditions (b1) and (b2), we follow \cite{DSPS19} Lemma 2.25: The proof of (b1) does not depend on the assumption of finiteness.
\begin{lemma}\label{lm:DSPSprojective}
If $\catC$ is rigid and and $\cM$ is a module category with right exact action and inner $\cC$-Homs, then the following holds: If $P\in \catM$ is projective, then it is also $\catC$-projective. 
\end{lemma}
\begin{proof}
    This is the main content of the proof of  \cite{DSPS19} Lemma 2.25 starting in line 4 and this part does not require any finiteness, only existence of inner $\cC$-Homs.
\end{proof}
The final assertion of proof of \cite{DSPS19} Lemma 2.25 cannot be directly copied, because it uses the existence of a global projective generator, which requires finiteness. However, a $\catC$-generator can be in a sense much smaller, and hence (b2) can be achieved in the following situation:

\begin{lemma}\label{lm_finiteB}
Suppose $\cM$ is a module category with exact action and inner $\cC$-Homs and suppose  $\cM$  is Artinian as abelian category (for example locally finite).
Suppose that there is a finite set of simple objects $S_1,\ldots,S_n$ in $\catM$ such that any simple module $S$ arises as a quotient of $V\otimes S_i$ for some $i$ and some $V\in \catC$. Suppose the projective covers $P_i$ of the $S_i$ exist in $\cM$. Then the direct sum of the projective covers $P_i$ is a projective $\catC$-generator $P$.   

\end{lemma}
\begin{proof}
This is similar to \cite{DSPS19} Lemma 2.22, but we spell out the details for convenience:
Being a $\cC$-generator $P$ means by definition that $\underline{\Hom}(P,-)$ is faithful. So we have to prove that for any any nonzero morphism $f:X\to Y$ in $\cM$ the corresponding morphism $f':\underline{\Hom}(P,X)\to \underline{\Hom}(P,Y)$ is nonzero, and this can be proven by providing $V\in \catC$ and $t:V\to \underline{\Hom}(P,X)$ such that $f\circ t:V\to \underline{\Hom}(P,Y)$ is nonzero. By the defining adjunction property of inner $\cC$-Homs this means to provide for any nonzero morphism $f:X\to Y$ a morphism $t:V\otimes P\to X$ such that $f\circ t:V\to \underline{\Hom}(P,Y)$ is nonzero. We will now provide such a morphism under the given assumption:

\bigskip

Let $\iota:S\hookrightarrow Y$ be a simple object in the image of $f$ and $\tilde{\iota}:\tilde{S}\hookrightarrow Y$ be its preimage under $f$. Note that image and preimages (as an example of a pullback) exist for abelian categories, and a simple $S$ exists in the image by assuming Artinian. A projective cover $\pi:P_S\twoheadrightarrow S$ by definition admits a lift $\tilde{\pi}: P_S\to \tilde{S}$. More generally, in our situation 
assume we have $V\otimes S_i\twoheadrightarrow S$ for $V\in \cC$ as asserted, then $\pi:V\otimes P_i\twoheadrightarrow S$ admits a lift $\tilde{\pi}:V\otimes P_i\to \tilde{S}$, because $V\otimes P_i$ is still projective by assumed exactness of the action in the second argument. The concatenation of $\tilde{\pi}$ with the embedding $\tilde{S}\hookrightarrow X$ gives a morphism $V\otimes P_i\to X$ whose concatenation with $f$ is a nonzero morphism $V\otimes P_i \to Y$ factoring over $S\hookrightarrow Y$. \qedhere

\end{proof}

\begin{example}\label{ex_finiteB}
In particular this is the case if $\catM=\catB$ is a tensor category and $\catC$ is a subcategory containing all simple modules, then taking $M_1$ the tensor unit suffices and its projective cover ${P}_\1$ is a projective generator.    
\end{example}

To summarize:
\begin{itemize}
    \item The first key finiteness condition is the existence of inner $\cC$-Homs. This is true e.g. under the finiteness conditions in Lemma \ref{lemma_finite} for $\cC$ semisimple.
    \item The second key finiteness condition is the existence of a projective $\cC$-generator. This is true under the finiteness condition in Lemma \ref{lm_finiteB}, and in particular it holds in our situation of $\cC\subset \cM$ a monoidal subcategory containing all simple modules.   
    \item In general  Lemma \ref{lm:DSPSprojective} proves that projective implies $\cC$-projective, under the assumption that $\cC$ is rigid.
    \item Finally the existence of a $\cC$-projective $\cC$-generator asserts by Theorem \ref{thm:DSPS} that the algebra $R$ realizes the module category $\cM$.
\end{itemize}

Combining these assertions in our case yields:

\begin{corollary}\label{cor:ModuleCatFiniteFinal}
    Assume $\catC$ is a rigid tensor category and $\catM$ a tensor category with an exact $\catC$ module functor $\catM\to \catC$. Assume $\catC$ is semisimple with finite-dimensional Hom spaces, and assume for given simples $W,W'$ there are only finitely many simples $V$ with $\Hom_\catC(V\otimes W,W'')\neq 0$. Assume that $\catM$ is locally finite and all simple modules are contained in $\catC$ and all projective covers in $\cM$ exist. Then there exists an algebra $R$ in $\catC$ with $\catM\cong\catC_R$.
\end{corollary}

\subsection{Relative Tannaka Krein Reconstruction}\label{subsec:Mombelli}

\newcommand{\Hdual}{C}

In \cites{LM25} Theorems 6.7 and~6.11 we prove the following theorem, where $C=\Nichols^*$ is the dual to the main Hopf algebra in this article.  

\begin{theorem}\label{thm:MombelliInfinite}
 Let $\catC$ be a finite and rigid tensor category and $\catB$ a right $\catC$- module category (action being right exact in both variables). Let $F:\catB\to \catC$ be an exact faithful functor. Then the relative coend developed in \cite{BM21} produces a coalgebra $\Hdual\in \catC$ such that as module categories 
 \begin{align}\label{formula_MombelliModuleCat}
     \catB\cong \catC^\Hdual
 \end{align}
 Suppose in addition that $\catC$ is braided and $F$ is a tensor functor. Then $\Hdual$ has the structure of a bialgebra and the equivalence \eqref{formula_MombelliModuleCat} is an equivalence of tensor categories.  Suppose in addition that $\cB$ is a rigid monoidal $\cC$-module category and $F$ preserves  this structure, then $\Hdual$ is a Hopf algebra.
\end{theorem}
In addition, 
    Theorem 6.17  concludes from rigidity of $\catC,\catB$ that $\Hdual$ is a Hopf algebra.
This result can be seen as a relative version of Tannaka Krein reconstruction \cite{Schau91} for $\cC$ the category of vector spaces, and a categorical version of the Radford biproduct for $\cC$ the category of representations of a Hopf algebra in vector spaces \cites{Maj95,Rad85}.

We will now generalize this result beyond tensor categories, dropping finiteness and spelling out what is actually needed. Then our main result in this section is this infinite version of \cite{LM25}:

\begin{theorem}[Relative Tannaka-Krein reconstruction]\label{thm:infiniteML}
\SimonRem{GV is enough}
 Let $\catC$ be a rigid tensor category  and $\catB$ a right $\catC$-module category, the action being right exact. Let $F:\catB\to \catC$ be an exact faithful functor. Suppose that the finiteness assertions in Corollary \ref{cor:ModuleCatFiniteFinal} hold, so there exists a realizing coalgebra $\Hdual\in \catC$, such that as module categories 
 \begin{align}\label{formula_newMombelliModuleCat}
     \catB\cong \catC^\Hdual
 \end{align}
 and $F$ is the forgetful functor. Then the relative coend produces such a coalgebra. Suppose in addition that $\catC$ is braided, $\cB$ is a monoidal $\cC$-module category and  $F$ is a tensor functor. Then $\Hdual$ has the structure of a bialgebra and the equivalence \eqref{formula_newMombelliModuleCat} is an equivalence of tensor categories. Suppose in addition that $\cB$ is a rigid monoidal $\cC$-module category and $F$ preserves this structure, then $\Hdual$ is a Hopf algebra. 
\end{theorem}
\begin{proof}
The proof of \cite{LM25} essentially still holds, but we have to carefully review to see where we use finiteness: The proof relies on  the existence of inner $\catC$-Homs and representability $\catM={_R}\catC$, which we have now as a separate assumption; note that the cited article uses right module categories and left $R$-modules. The crucial point is the existence of certain relative coends to reconstruct the coalgebra. For this we now proceed as   \cite{LM25} Section 4: The standard algebra consequence of representability is that the  assignment of a left $R$-module ${_R}\hat{F}$ to the corresponding functor   
$$\Phi:\catM^{op}\to \mathrm{Rex}_\catC(\catM,\catC)$$
$$\hat{F}_R \mapsto F=(\hat{F}^*\otimes_R -)
=\underline{\Hom}_R(-,\hat{F})^*$$
is an equivalence, 
because any right exact functor $F$ can be evaluated on the regular representation ${_R}R$ to define $\hat{F}$ representing $F$.

Consider now the functor 
$$\catB^{op}\boxtimes \catB \to \catB$$
$$M\otimes N\mapsto N.{^*}F(M)$$
\SimonRem{typo in LM25}
which also comes with a natural  $\catC$-prebalancing \cite{LM25} (3.9).
We claim that $\hat{F}=\Phi^{-1}(F)$ is a (and hence the) relative coend of this functor as defined in \cite{LM25} Definition 3.2
$$\oint^{M\in\catB} M.{^*}F(M) = {_R}\hat{F} $$
Indeed, if we spell out this functor involving $F$ in terms of $\hat{F}$
\SimonRem{We dont assume pivotality, but left-right-dual ok. Recall: apply rightdual  coevaluation for leftdual object and then leftdual evaluation constructs the iso}
$$N.{^*}F(M)=N\otimes \underline{\Hom}_R(M,\hat{F})$$
then evaluation gives a dinatural transformation 
$$M.{^*}F(M)=M\otimes \underline{\Hom}_R(M,\hat{F}) \to \hat{F}$$
and the necessary pentagon \cite{LM25} (3.5) is clearly fulfilled, and universality is also clearly true.

\bigskip

As as second step, for two functors $F,E:\catB\to \catC$ we can apply $E$ to the result above and find the relative coend 
$$\oint^{M\in\catB} E(M).{^*}F(M) = 
E(\hat{F})=  \underline{\Hom}_R(\hat{F},\hat{E})^*$$
The remainder of \cite{LM25} establishes that the relative coend 
$$C:=\oint^{M\in\catB} F(M).{^*}F(M) = 
 \underline{\Hom}_R(\hat{F},\hat{F})^*$$
has the explicitly the following additional structures:
\begin{itemize}
\item a coalgebra structure by \cite{LM25} Proposition 6.2.
\item a bialgebra structure by \cite{LM25} Theorem 6.11 if $\cC$ is braided, $\cB$ is a monoidal $\cC$-module category and $F$ is strictly monoidal
\item a Hopf algebra by \cite{LM25} Corollary 6.19 if in addition to the assumptions $\cB$ is a rigid monoidal $\cC$-module category  and $F$ preserves this structure.
\end{itemize}
These assertions uses the universal property and the coherence data of the coend, if it exists, and do not depend on any finiteness assumptions. 

\end{proof}

\SimonRem{Think again. I it like we are constructing a Hopf monad and then give a realizing Hopf algebra? Would it not be much easier to directly have a right adjoint (then lax) to the faithful exact tensor functor? Can we not use SAFT for this? }

\begin{example}[compare \cite{LM25} Lemma 6.3]
Suppose that $\catC$ is rigid and $\catB$ is a module category which is known to be of the form 
$$ \catB\cong \catC^\Hdual$$
for some coalgebra $\Hdual$ and $F$ is the forgetful functor. 
Taking $R={^*}\Hdual$ as the realizing algebra, the forgetful functor $F:\catC_R\to \catC$ is realized by the regular representation $\hat{F}={_R}R$. Then the relative coend recovers as intended

$$\oint^{M\in\catB} F(M).{^*}F(M) =\underline{\Hom}_R(R,R)^*=R^*=\Hdual$$

Suppose in addition that $\catC$ is braided, and assume $\Hdual$ is a bialgebra in this braided tensor category. Then $\catC^\Hdual$ is a tensor category and the forgetful functor is a tensor functor. Suppose in addition $\Hdual$ is a Hopf algebra, then $\catC^\Hdual$ is a rigid tensor category and the forgetful functor is compatible with the dualities. 

As a remark, we could moreover endow this relative coend with the structure of an object in $\catB$, then in our example we would recover the adjoint representation, as expected from Hopf algebra theory.
\end{example}

\subsection{Nichols algebra arguments}\label{subsec:NicholsArguments}

Let $\catC$ be a braided tensor category. Let $M=M_1\oplus\cdots\oplus M_n$ be a semisimple object in $\catC$ with simple summands $M_i$. Let $\NicholsOf(M)$ be the corresponding Nichols algebra. 

\begin{definition}\label{def:robust}
We call the Nichols algebra $\NicholsOf(M)$ \emph{robust}, if any Hopf algebra $\Nichols$ with the following properties is already isomorphic to the Nichols algebra:
\begin{itemize}
    \item Every simple $\Nichols$-comodule $X$ is trivial $X=X_\NicholsAlgebraCounit$.
    \item $\Ext^1(1,X_\NicholsAlgebraCounit)$ in the category of comodules has for any simple object $X$ the same finite $\KK$-dimension as the multiplicity of $X$ in $M$.  
\end{itemize}
\end{definition}

As the definition wants to imply, this property can be deduced from an intimate knowledge of the Nichols algebra in question by checking the deformability of the relations. 
For $\catC$ the category of $\Gamma$-graded vector spaces this is in the core of the Andruskiewitsch-Schneider program for the classification of pointed Hopf algebras, see 
\cites{AS10} and the more recent survey \cite{AG19}. It is a main result of Angiono et. al. in \cite{AG11} and \cite{AKM15} that in this case all finite-dimensional Nichols algebras are robust. This is proven by deriving general criteria for relations to be robust and then going trough the classification of finite-dimensional Nichols algebras by Heckenberger \cite{Heck09} and their defining relations \cite{An13}.

Our Nichols algebra in Lemma \ref{lm_Nichols} is in a different category, and many tools are not available yet; on the other hand it is a very small example. Note that as a set of defining relations we understand a subobject of the tensor algebra generating the kernel of the projection to the Nichols algebra.

\begin{lemma}\label{lm_robust}
Let $\NicholsOf(M)$ be a Nichols algebra in a semisimple rigid braided tensor category. 
\begin{itemize}
    \item Suppose that a set of defining relations of the Nichols algebra as direct sum of simple objects in $\cC$ does not contain the unit object in $\cC$. Then any filtered Hopf algebra $\Nichols$ with $\gr(\Nichols)=\NicholsOf(M)$ is already the Nichols algebra (no deformation).
    \item Suppose that the previous statement holds also for the gradewise dual $\NicholsOf(M)^*\cong\NicholsOf(M^*)$. Then $\Nichols(M)$ is generated in degree $1$.
\end{itemize}  
Both statements together mean that $\Nichols(M)$ is robust.
\end{lemma}
\begin{proof}
We essentially follow \cite{AKM15} for the first statement and \cite{AS00} for the second statement:\\

By \cite{MW09} Lemma 4.2.2 and \cite{AKM15} Theorem 5.3, which are both in a general categorical context, the algebra had no nontrivial deformations if $\Hom_\catC(R,M)=\{0\}$ for all defining relations $R$. The semisimplicity is a particular case in which the splitting condition for $\NicholsAlgebraCounit$ in \cite{AKM15} Lemma 4.1 is fulfilled, see \cite{AKM15} Remark ~4.2.\\

By \cite{AS00} Lemma 5.5, which is not formulated in general categorical context, but continues to hold obviously, a graded Hopf $\Nichols$ is generated in degree $1$ iff the dual Hopf algebra $\Nichols^*$ has all its primitive elements in degree $1$. But this implies that in the coradical filtration (for general categories see \cite{EGNO15} Section 1.13) that $\Nichols^*$ is a nontrivial deformation. 
\end{proof}

\SimonRem{Could use more elaboration, but surely true}

\begin{corollary}\label{cor_robust}    
The Nichols algebra $\NicholsOf(M)$ in $\cC^\k$ in Lemma \ref{lm_Nichols} is robust. This remains true over any uprolling of $\cC^\k$. 
\end{corollary}
\begin{proof}
 By the discussion in Lemma \ref{lm_Nichols} the defining relation for $v>3$ is  $R=\VirPhi_{1,3}\boxtimes\Pi_{2\ellNichols}(2\lambdaNichols)$, which is a simple object not isomorphic to $M$, hence the previous Lemma \ref{lm_robust} applies. For $v=3$ with $M$ simple and invertible the defining relation is $R=M^{\otimes 3}=\VirPhi_{1,2}\boxtimes\Pi_{3\ellNichols}(3\lambdaNichols)$, which is again not isomorphic to $M$ unless $\Pi_{2\ellNichols}(2\lambdaNichols)=1$. In the case $v=2$ the relation is $\Pi_{2\ellNichols}(2\lambdaNichols)^{\otimes 2}$, which is not isomorphic to $M$ unless $\Pi_{3\ellNichols}(3\lambdaNichols)=1$ 

We now check if there can exist uprollings where a relation is $\1$ in the two degenerate cases $v=2,3$. We do so by computing the quadratic form, using $\frac12(u/v)+4xy=N$ odd: 
\[Q(j(xc+yd))
=e^{j^2\pi\i \,4xy }
=e^{j^2\pi\i\,(N-\frac12(u/v))}\]
By considering the denominators, this can never be $1$ for $v=3,j=2$ or $v=2,j=3$. Hence there is no such uprolling.\\

Both statements are true for $\NicholsOf(M^*)$ as well, which is just the same Nichols algebra for the parameters $(-\ellNichols,-\lambdaNichols)$.
\end{proof}

\subsection{Main Statement}\label{subsec:BorelEquivalence}

\begin{theorem}\label{thm:BorelEquivalence}
The category $(\cD^\k)_A$ in Section \ref{subsec:sl2Amodules} is equivalent as tensor category to the category of modules over $\NicholsOf(M)$ in $\cC^\k$ in Section \ref{subsec:NicholsAlgebra} for $$M=\VirPhi_{1,2}\boxtimes \Pi_{1}(-t/2),\; v\geq 3,\qquad 
M=\VirPhi_{1,1}\boxtimes \Pi_{2}(-t),\;v=2.$$ 
Under the correspondence, the splitting functor is the functor forgetting the Nichols algebra action, and the embedding of local modules is the functor endowing an object with trivial Nichols algebra action.
\end{theorem}
\begin{proof}
The proof combines results above 
\begin{itemize}
\item By Section \ref{subsec:sl2Amodules} all simple objects in $(\cD^\k)_A$ are in $(\cD^\k)_A^\loc$. Hence by \cite{CLR23} Section 4 we have a monoidal \emph{splitting functor} \[(\catD^\k)_A\to (\catD^\k)_A^\loc\]
which is left-inverse to the embedding. This functor is by construction faithful and exact. 
\item By Section \ref{subsec:sl2Amodules} the category $(\cD^\k)_A$ fulfills the weaker finiteness assertions in Corollary \ref{cor:ModuleCatFiniteFinal} and is rigid. Hence by Theorem \ref{thm:infiniteML} there exists a Hopf algebra $\Nichols$ with an equivalence of tensor categories
$$\catD_A\cong \catC_\Nichols,\qquad
\catC:=\catD_A^\loc.$$
Note that the tensor product on the left side is $\otimes_A$ and on the right side the underlying tensor product $\otimes$ in $\catC$ and action of $\Nichols$ via its coalgebra structure.
\item By Section \ref{subsec:sl2Amodules} $\Ext^1(1,X)$ is one-dimensional for $
X \cong \VirPhi_{1,2} \otimes \Pi_1(-t/2)$ if $v\geq 3$ resp. for $X \cong \VirPhi_{1,1} \otimes \Pi_2(-t)$ if $v=2$ and zero else. In this case we say that the Nichols algebra associated to $\Nichols$ is the asserted Nichols algebra $\NicholsOf(M)$.
\item By Section \ref{subsec:NicholsArguments} this Nichols algebra $\Nichols(M)$ is robust, hence the properties in the previous bullet are by Definition \ref{def:robust} sufficient an isomorphism of Hopf algebras $\Nichols\cong \NicholsOf(M)$.
\end{itemize}
\end{proof}

\section{Adjunctions, Monadicity and representing algebras}\label{sec:adjunction}

The basic material in this section can largely be found in \cite{Riehl16} and \cite{BV07}. 

\subsection{Adjunctions}

In this section we consider  a functor between linear categories $F:\catD \to \catC$ admitting a right adjoint functor $\Fra=F^{\mathrm{ra}}$, or in formulas $F\dashv G$:
\[\begin{tikzcd}[row sep=5ex, column sep=6ex]
\catD\arrow[r,swap, shift right=2pt,"F"] & \arrow[l,swap, shift right=2pt,"\Fra"] \catC
\end{tikzcd}\]
\begin{align*}
\Hom_\catD(M,\Fra(N)) 
&\cong\Hom_\catC(F(M),N) 
\end{align*}

Note that in this article we always draw the right adjoint functor above and the left adjoint below. We collect some well-known properties on functors: 
Recall that in particular right exact functors preserve surjective morphisms and left exact functors preserve injective morphisms.
In abelian categories an exact functor $F$ is faithful iff $X\neq 0$ implies $F(X)\neq 0$. Exact faithful functors in abelian categories also reflect exact sequences. For example, in tensor categories in the sense of \cite{EGNO15} the tensor product is faithful, exact and reflects exact sequences. \SimonRem{Last two SY24 Rem 2.2}\\

We collect some well-known properties on adjunctions:
An adjoint functor is unique, if it exists \cite{Riehl16} Proposition 4.4.1.
An adjunction implies that the left adjoint $F$ is right exact and the right adjoint  $\Fra$ is left exact see \cite{Riehl16} Corollary 4.5.11.
Conversely, if $\catD$ is finite and $F$ is right exact, then the right adjoint $\Fra$ exists, similarly if  $\catD$ is finite and $\Fra$ is left exact, then the left adjoint $F$ exists. Theorems ensuring the existence of adjoints are derived from the general adjoint functor theorem and are discussed in \cite{Riehl16} Section 4.5, the specific claimed statement about finite categories can be found in \cite{MacL52} Chapter V Section 8 Corollary.\SimonRem{Riehl 4.6.17: Locally presentable means: cocontinuous iff admits right adjoint and continuous and accessible off admits left adjoint} 

Recall that for an adjunction $F\dashv G$, there are natural transformations called \emph{unit} and \emph{counit}  
\begin{align*}
\eta:\; &\id \to \Fra\circ F \\
\epsilon:\;  &F\circ \Fra \to \id 
\end{align*}
fulfilling zigzag identities that are obtained as images  of
$F(\id_{M})$ for $N=F(M)$ resp. of $\Fra(\id_{N})$ for $M=\Fra(N)$ under the adjoint equivalence of Homs. Conversely, the full adjoint equivalence can be be recovered from given $\eta,\epsilon$ fulfilling the zigzag identities by
\begin{align*}
\Hom_\catD(M,\Fra(N)) 
&\cong\Hom_\catC(F(M),N) \\
f&\mapsto \epsilon\circ F(f) \\
\Fra(g)\circ \eta & \mapsfrom  g 
\end{align*}
The following properties can be directly read off the definition

\begin{lemma}[\cite{Riehl16} Lemma 4.5.13]\label{lm_AdjUnitInjective}
The left adjoint $F$ is faithful (injective on Hom-sets) resp. full (surjective on Hom-sets), iff the unit of the adjunction $\eta:\id\to \Fra\circ F$ is injective resp. surjective.  Similarly, the right adjoint $\Fra$ is faithful resp. full, iff the counit of the adjunction $\epsilon:F\circ\Fra\to \id$ is surjective resp. injective. 
\end{lemma}

Now assume $F$ is a tensor functor with tensor structure $F^{\otimes}:F\otimes F\to F$. A standard argument, appearing maybe first in \cite{Dav10} Section 5, shows that the right adjoint $\Fra$ of an oplax (in particular of a strict) monoidal functor is lax monoidal functor via

$$\begin{tikzcd}[row sep=5ex, column sep=6ex]
\Fra(X)\otimes \Fra(Y) 
\arrow[r,dashed]
\arrow[d,"\eta"]
&
\Fra(X\otimes Y)
\\
\Fra F(\Fra(X)\otimes \Fra(Y))
\arrow[r,"F^\otimes"]
&
\Fra(F\Fra(X)\otimes F\Fra(Y))
\arrow[u,"\epsilon\otimes \epsilon"]
\end{tikzcd}$$

\subsection{Monads}

\SimonRem{If comonadicity does not hold, we can always see it directly}

A \emph{monad} is a monoid $\monT$ in the category of endofunctors $\mathrm{End}(\catD)$. That is, there are natural transformations $$\mu^\monT_X:\;\monT(\monT(X)))\to X,\qquad
\eta^\monT_X:\;X\to \monT(X)$$
fulfilling unitality and associativity. A \emph{module} over a monad $\monT$ is an object $X\in \catD$ together with a morphism $\rho:\monT(X)\to X$ compatible with multiplication and unit of the monad. This yields a category $\catD_\monT$ of modules $(X,\rho)$ over a monad $\monT$. Forgetting the module is an exact faithful functor $\catD_\monT\to \catD$. On the other hand, for every object $X\in\catD$ the image $\monT(X)$ has a structure of a $\monT$-module by $\mu^\monT_{X}$. We call this functor $\funT$ and the image an \emph{induced module}. This gives an adjunction 
\[\begin{tikzcd}[row sep=6ex, column sep=6ex]
\renewcommand{\arraystretch}{0.7}
\catD
\arrow[r, leftarrow, shift right=-2pt,"\forget_\monT"]
\arrow[r, rightarrow, swap,shift right=2pt,"\funT"]
& 
\catD_\monT
&
\begin{array}{c}
    \scriptstyle \textnormal{faithful exact}\\
    \scriptstyle \textnormal{right exact} \\
\end{array}
\end{tikzcd}\]
with $\monT=\forget_\monT\circ \funT$ and with unit and counit
\begin{align*}
    \eta_X:\;
    (\forget_\monT\circ \funT)(X)=\monT(X)
    &\stackrel{\eta_X^\monT}{\longleftarrow} X \\
    \epsilon_X:\;
    (\funT\circ\forget_\monT)(X,\rho)=\funT(X)
    &\stackrel{\rho}{\longrightarrow}
    (X,\rho)
\end{align*}

\SimonRem{
Explicit structure
\begin{align*}
\Hom_{\catD}(X, (Y,\rho))
&\;\cong\;
\Hom_{\catD_\monT}(\monT(X), Y)
\\
f&\mapsto \rho\circ\monT(f)
\\
g\circ \eta_X^\monT&\mapsfrom g
\end{align*}
}

Dually, a \emph{comonad} $\comonC$ is a comonoid in the category of endofunctors, and there is a category of $\comonC$-comodules $\catC^\comonC$ and an adjunction with the forgetful functor, now as the left adjoint

\[\begin{tikzcd}[row sep=6ex, column sep=6ex]
\renewcommand{\arraystretch}{0.7}
\catC^\comonC
\arrow[r, rightarrow, swap, shift right=2pt,"\forget_\comonC"]
\arrow[r, leftarrow,shift right=-2pt,"\funC"]
& 
\catC
&
\begin{array}{c}
    \scriptstyle \textnormal{left exact}\\
    \scriptstyle \textnormal{faithful exact} \\
\end{array}
\end{tikzcd}\]

\begin{example}
An algebra $T$ in a tensor category $\catD$ gives rise to a monad $T\otimes(-)$, and we recover the category of modules of $T$. Similarly,  a coalgebra $C$ in  a tensor category $\catD$, gives rise to a comonad $C\otimes(-)$, and we recover the category of modules of $T$. The induction functor sends any object to the corresponding free module resp. cofree comodule.
\end{example}
\begin{example}\label{exm_centralcomonad}
The forgetful functor from the Drinfeld center $\cZ(\catD)$ to $\catD$ is faithful and exact. However, the associated comonad does typically not come from right tensoring with a (bi)algebra object in $\catD$. 

For example if $\cat=\Rep(H)$ are representations of a Hopf algebra, then  $\cZ(\cat)$ are objects in $\cat$ with an additional coaction of $H$ defining the half braiding, but the compatibility condition of action and coaction (Yetter-Drinfeld condition) cannot be expressed by turning $H$ into a suitable object in $\cat$. Accordingly, the comonad cannot be expressed by $\monT(X)=H\otimes X$ on the level of objects. Rather, there is a slightly more complicated action of $H$ on the vector space $H\otimes X$ involving the double braiding around $X$, see Section \ref{subsec:QuantumGroup} and \cite{LW21} Theorem 3.13. 

For finite rigid semisimple $\cat$, this comonad can indeed be expressed as $ \monT(X)=\bigoplus_Y Y\otimes X\otimes Y^*$, summing over all simple objects $Y$, and in more general cases there is a similar formula using coends, see \cite{Shim19} Section 2.7.
\end{example}

 An arbitrary adjunction $F\dashv \Fra$ with unit $\eta$ and counit $\epsilon$
 produces a monad structure on the functor $\monT=\Fra\circ F$ with the following multiplication and unit:
\begin{align*}
(\Fra\circ F)\circ(\Fra\circ F) 
=\Fra\circ (F\circ \Fra)\circ F
&\stackrel{\eta}{\longrightarrow} \Fra\circ F 
&%
\id
&\stackrel{\epsilon}{\longrightarrow}
\Fra\circ F 
\end{align*}
There is a \emph{comparison functor} $\kappa^\monT$ between the given adjunction and the adjunction of the monad $\monT=\Fra\circ F$
 \[\begin{tikzcd}[row sep=5ex, column sep=6ex]
\catD
\arrow[r,swap, shift right=2pt,"F"]
\arrow[d,equal]& 
\arrow[l,swap, shift right=2pt,"\Fra"] 
\catC
\arrow[d,"\kappa^\monT"]
\\
\catD
\arrow[r, leftarrow, shift right=-2pt,"\forget_\monT"]
\arrow[r, rightarrow, swap,shift right=2pt,"\funT"]
& 
\catD_\monT
\end{tikzcd}\]
Namely, for any $X$ the image $\Fra(X)$ comes with  the $\monT$-module structure 
$$(\Fra\circ F)(\Fra(X))\stackrel{F(\epsilon)}{\longrightarrow} \Fra(X).$$
Similarly, an arbitrary adjunction $F\dashv G$ produces a comonad on the functor $\comonC=F\circ G$ and a comparison functor

 \[\begin{tikzcd}[row sep=5ex, column sep=6ex]
\catD
\arrow[r,swap, shift right=2pt,"F"]
\arrow[d,"\kappa^\comonC"]& 
\arrow[l,swap, shift right=2pt,"\Fra"] 
\catC
\arrow[d,equal]
\\
\catC^\comonC
\arrow[r, rightarrow, swap, shift right=2pt,"\forget_\comonC"]
\arrow[r, leftarrow,shift right=-2pt,"\comonC"]
& 
\catC
\end{tikzcd}\]

The adjunction $F\dashv\Fra$ is \emph{(co)monadic} if $\kappa$ is an equivalence of categories. Becks monadicity theorem \cite{MacL52} Section VI.7 asserts \SimonRem{(my source: BrugieresNatale Sec1.7: Monadicity and Comonoidal structure)}

\begin{lemma}
    Given any adjunction of additive functors
    \[\begin{tikzcd}[row sep=5ex, column sep=6ex]
    \catD\arrow[r,swap, shift right=2pt,"F"] & \arrow[l,swap, shift right=2pt,"\Fra"] \catC
    \end{tikzcd}\]
    Then if the right adjoint $G$ is faithful and exact , then monadicity holds:
    $$\kappa^\monT: \catC \cong \catD_\monT$$
    Similarly, if the left adjoint $F$ is 
    faithful and exact then comonadicity holds
    $$\kappa^\comonC: \catD \cong \catC^\comonC$$
\end{lemma}
Note that this condition is clearly necessary, because the forgetful functor of a monad and comonad is exact and faithful.

\newcommand{\cotimes}{%
  \begin{tikzpicture}[baseline=(X.base)]
    \node[draw=black, fill=black, circle, inner sep=0.5ex] (X) {};
    \node[text=white, scale=0.8] at (X) {$\times$};
  \end{tikzpicture}%
}

 \subsection{Monads in monoidal setting}\label{subsec:MonoidalMonad}

We now assume $\catD$ resp. $\catC$ is a tensor category. Then a monad $\monT$ resp. comonad $\comonC$ can be a monoid resp. comonoid in the category of strict monoidal endofunctors. Weaker, we could ask for lax monoidal endofunctors, meaning $\monT(X)\otimes \monT(Y)\to \monT(X\otimes Y)$, or for oplax monoidal (or: comonoidal) endofunctors. So we have four combinations. 
\begin{example}\label{exm_fourmonoidalmonads}
We give fundamental examples for all four combinations:
\begin{enumerate}[(a)]
\item For $T$ a commutative algebra in $\cZ(\catD)$ we have a lax (or: monoidal) monad
\[\begin{tikzcd}[row sep=6ex, column sep=6ex]
\renewcommand{\arraystretch}{0.7}
(\catD,\otimes)
\arrow[r, leftarrow, shift right=-2pt,"\forget_T"]
\arrow[r, rightarrow, swap,shift right=2pt,"T\otimes(-)"]
& 
(\catD_T,\otimes_T)
&
\begin{array}{c}
    \scriptstyle \textnormal{faithful exact}\\
    \scriptstyle \textnormal{right exact} \\
\end{array}
&
\begin{array}{c}
    \scriptstyle \textnormal{lax monoidal}\\
    \scriptstyle \textnormal{strict monoidal} \\
\end{array}
\end{tikzcd}\]

\item For $C$ a cocommutative coalgebra in $\cZ(\cC)$ we have an oplax (or: comonoidal) comonad 

\[\begin{tikzcd}[row sep=6ex, column sep=6ex]
\renewcommand{\arraystretch}{0.7}
(\catC^C,\otimes^C)
\arrow[r, rightarrow, swap, shift right=2pt,"\forget_C"]
\arrow[r, leftarrow,shift right=-2pt,"C\otimes(-)"]
& 
(\catC,\otimes)
&
\begin{array}{c}
    \scriptstyle \textnormal{left exact}\\
    \scriptstyle \textnormal{faithful exact} \\
\end{array}
&
\begin{array}{c}
    \scriptstyle \textnormal{strict monoidal}\\
    \scriptstyle \textnormal{oplax monoidal} \\
\end{array}
\end{tikzcd}\]
with $\otimes^C$ the cotensor product, which is an equalizer rather then a coequalizer.

\item For $T$ a  bialgebra in $\cZ(\catD)$ we have an oplax monad (or \emph{bimonad})

\[\begin{tikzcd}[row sep=6ex, column sep=6ex]
\renewcommand{\arraystretch}{0.7}
(\catD,\otimes)
\arrow[r, leftarrow, shift right=-2pt,"\forget_T"]
\arrow[r, rightarrow, swap,shift right=2pt,"T\otimes(-)"]
& 
(\catD_T,\otimes)
&
\begin{array}{c}
    \scriptstyle \textnormal{faithful exact}\\
    \scriptstyle \textnormal{right exact} \\
\end{array}
&
\begin{array}{c}
    \scriptstyle \textnormal{strict monoidal}\\
    \scriptstyle \textnormal{oplax monoidal} \\
\end{array}
\end{tikzcd}\]
where the coproduct of $T$ gives an oplax tensor structure on $T\otimes(-)$ \newline
\[T\otimes (X\otimes Y)\to (T\otimes X)\otimes(T\otimes Y)\]

\item For $C$ a bialgebra in $\cZ(\catC)$ we also have a lax comonad  (or \emph{bicomonad})
\[\begin{tikzcd}[row sep=6ex, column sep=6ex]
\renewcommand{\arraystretch}{0.7}
(\catC^C,\otimes)
\arrow[r, rightarrow, swap, shift right=2pt,"\forget_C"]
\arrow[r, leftarrow,shift right=-2pt,"C\otimes(-)"]
& 
(\catC,\otimes)
&
\begin{array}{c}
    \scriptstyle \textnormal{left exact}\\
    \scriptstyle \textnormal{faithful exact} \\
\end{array}
&
\begin{array}{c}
    \scriptstyle \textnormal{lax monoidal}\\
    \scriptstyle \textnormal{strict monoidal} \\
\end{array}
\end{tikzcd}\]
\end{enumerate}
\end{example}

For any adjunction of lax resp. oplax monoidal functors we have a lax resp. oplax monad and comonad.

\begin{remark}
Note that the right adjoint of a lax monoidal functor is oplax, so in an adjunction of lax monoidal functors, the right adjoint is even strictly monoidal. Similarly, in an adjunction of oplax monoidal functors, the left adjoint is even strictly monoidal. So in general we have precisely the combinations in the previous examples.
\end{remark}

\subsection{Corestriction}

Consider again an adjunction, which we assume to be monadic, so we can write it in the form. We keep in mind the example of an algebra $T$ and the induction functor.

\[\begin{tikzcd}[row sep=6ex, column sep=6ex]
\renewcommand{\arraystretch}{0.7}
\catD
\arrow[r, leftarrow, shift right=-2pt,"\forget_\monT"]
\arrow[r, rightarrow, swap,shift right=2pt,"\funT"]
& 
\catD_\monT
&
\begin{array}{c}
    \scriptstyle \textnormal{faithful exact}\\
    \scriptstyle \textnormal{right exact} \\
\end{array}
\end{tikzcd}\]

As an exact functor, the forget functor (or: restriction) has in many situations  a further right adjoint called \emph{coinduction} $\mathrm{Coind}_\monT:\catD\leftarrow \catD_A$.  In the example of an algebra, this can be constructed using inner Homs. 

In {unusually} favorable situations  
the induction functor $T$ itself may be exact. In the example of an algebra, this  is for example the case in the presence of duality. In such cases, we could call the left adjoint to the induction functor \emph{corestriction} $\mathrm{Cores}_\monT:\catD\rightarrow \catD_A$. 

\SimonRem{Then may itself be faithful and exact and thus serve itself as monadic functor.}

\begin{example}
We compare these notions to the more general standard notions for an algebra map $\fpullback:R\to T$ between algebras $R,T$ in vector spaces. Such a map induces a restriction functor $\fpullback^*$ between the categories of representations by precomposing a $T$-action with $\fpullback$ to an $R$-action. This functor is clearly exact, it has a left adjoint $\fpullback_!$ given by induction of an $R$-module $N$ to $T_{\fpullback}\otimes_R N$ and a right adjoint $\fpullback_*$ given by coinduction to $\Hom_R({_{\fpullback}}T,N)$.

If $R,T$ are commutative algebras and the categories of representation become monoidal via $\otimes_T,\otimes_R$, then restriction is  lax monoidal via and induction is strict monoidal. 

Under favorable circumstances, the induction functor itself can be exact and admit yet another left adjoint $\fpullback^!$, but in literature this appears more frequently in the derived setting. We summarize all these notions in the following picture, where again a functor set above another is right adjoint and a functor set below another is left adjoint.

\[\begin{tikzcd}[row sep=6ex, column sep=6ex]
{
\renewcommand{\arraystretch}{0.7}
\begin{array}{c}
    \scriptstyle \fpullback_*\\[.5mm]
    \scriptstyle \fpullback^* \\[.5mm]
    \scriptstyle \fpullback_!\\[.5mm]
    \scriptstyle \fpullback^!
\end{array}
}
&
\catD_R
\arrow[r, rightarrow, shift right=-15pt,"\mathrm{Coind}"{pos=0.5, anchor=center, inner sep=0pt}] 
\arrow[r, leftarrow, shift right=-5pt,"\mathrm{Res}"{pos=0.5, anchor=center, inner sep=0pt}]
\arrow[r,rightarrow, shift right=5pt,"\mathrm{Ind}"{pos=0.5, anchor=center, inner sep=0pt}] 
\arrow[r, leftarrow, shift right=15pt,"\mathrm{Cores}"{pos=0.5, anchor=center, inner sep=0pt}]
& 
\catD_T
&
\begin{array}{c}
    \scriptstyle \textnormal{left exact}\\ [-.2em]
    \scriptstyle \textnormal{exact} \\[-.2em]
    \scriptstyle \textnormal{exact}\\[-.2em]
    \scriptstyle \textnormal{right exact}
\end{array}
&
\begin{array}{c}
    \scriptstyle \textnormal{}\\[-.2em]
    \scriptstyle \textnormal{lax monoidal} \\[-.2em]
    \scriptstyle \textnormal{strict monoidal}\\[-.2em]
    \scriptstyle \textnormal{oplax monoidal}
\end{array}
\end{tikzcd}\]

\end{example}

We make the corestriction explicit in the presence of duality.  These calculations are surely not new, we spell it out for completeness and convenience. They should apply in more generality in closed categories. 
\begin{lemma} 
Let $\catD$ be a rigid tensor category and $A$ an   algebra. Then the dual object $A^*$ in $\catD$ has the structure of a right $A$-module and the following functor is left adjoint to the induction functor $A\otimes(-)$.
$$\mathrm{Cores}:\; \catD\leftarrow \catD_A$$
$$A^*\otimes_A M \mapsfrom M$$
\end{lemma}
(note that restriction can be rewritten as  $A\otimes_A M$)
\begin{proof}

The dual object $A^*$ in $\catD$ has a right $A$-module structure given by 
$$\rho_{A^*}:\;A^*\otimes A
\xrightarrow{\id\otimes \id \otimes \mathrm{coeval}_A}
A^*\otimes A \otimes A\otimes A^*
\xrightarrow{\id\otimes \mu_A \otimes \id}
A^*\otimes A\otimes A^*
\xrightarrow{\mathrm{eval}_A\otimes \id}
A^*
$$
It follows from the zigzag identities and associativity in $A$ that this indeed a module structure. Similarly, we can check easily that evaluation as a morphism in $\catD$ factors through the tensor product over $A$
$$A^*\otimes_A A\xrightarrow{\mathrm{eval}_A} \1 $$
and that coevaluation as a morphism in $\catD$  
lands in the cotensor product
$$1 \xrightarrow{\mathrm{coeval}_A} A\Box_A A^* $$
by which we mean here that the following two maps coincide
$$A \xrightarrow{\id\otimes \mathrm{coeval}_A} A\otimes (A\otimes A^*) 
\xrightarrow{\mu_A\otimes \id} A\otimes A^*$$
$$A \xrightarrow{\mathrm{coeval}_A\otimes \id}  (A\otimes A^*) \otimes A
\xrightarrow{\id\otimes \rho_{A^*}} A\otimes A^*$$

\bigskip

We now claim that $A^*\otimes_A(-)$ is left adjoint to $A\otimes(-)$. The unit and count are given in terms of evaluation and coevaluation
\begin{align*}
A^*\otimes_A (A \otimes X)
&\xrightarrow{\mathrm{eval}_A\otimes \id} X \\
M
&\xrightarrow{\mathrm{coeval}_A} A\otimes (A^*\otimes_A M) 
\end{align*}
The fact that the first morphism factors over the tensor product over $A$ and that the second morphism is a morphism of $A$-modules follow from the properties of evaluation and coevaluation established above.  
The compatibility of unit and counit of the adjunctions are again the zig-zag identities of evaluation and coevaluation. 

\end{proof}

\subsection{Pas de deux}\label{sec_deux}

Consider again a monadic adjunction. We continue to keep in mind the example of an algebra $T$ and the induction functor.

\[\begin{tikzcd}[row sep=6ex, column sep=6ex]
\renewcommand{\arraystretch}{0.7}
\catD
\arrow[r, leftarrow, shift right=-2pt,"\forget_\monT"]
\arrow[r, rightarrow, swap,shift right=2pt,"\funT"]
& 
\catD_\monT
&
\begin{array}{c}
    \scriptstyle \textnormal{faithful exact}\\
    \scriptstyle \textnormal{right exact} \\
\end{array}
\end{tikzcd}\]

Suppose as in the previous section that the induction functor is exact and faithful. Then it can serve itself as a monadic functor, more precisely: The lax monoidal induction functor $T$ and the strict monoidal forgetful functor $\fpullback_\monT$ then also define a lax comonad $\comonC=T\circ\forget_\monT$. Then $\catD$ can be recovered from $\catD_\monT$ as the category of $\comonC$-comodules 
$$\kappa_\comonC:\;\catD\stackrel{\sim}{\longrightarrow} (\catD_\monT)^\comonC$$

The comparison functor is given by sending an object $X\in \catD$ to the induced $\monT$-module $T(X)$, which comes in addition with a coaction  
$T(X)\to T(T(X))$ by inserting a unit. An explicit inverse of $\kappa_\comonC$ is given by sending a $\comonC$-comodule to the coinvariant part, see \cite{BLV11} Theorem 6.10.

This reconstruction expresses a general principle called \emph{descent}, the coaction being the descent data. In the monadic setting this has been established in \cite{BLV11} Section 6.4 and 6.5.

\SimonRem{Hopf modules are somehow the dual thing, right.}

\begin{remark}
Differently viewed, the induction functor and the corestriction functor define an oplax monad. 
\end{remark}

\begin{example}
Suppose that $T\in\catD$ is a commutative algebra such that the induction functor $T\otimes (-)$ is faithful and exact. This is true for example if $\catD$ is rigid.  \SimonRem{(duality implies exact) ($1\to A$ injective (unitality) sendet $Hom(X,Y)\to Hom(X,A\otimes Y)=Hom(X,ResInd Y)\cong(IndX,IndY)$ proves faithful}
Then the lax comonad is the following endofunctor of $\catD_T$:
$$ \comonC:\; {_T}X\mapsto {_T}T\otimes X$$
The comparison functor is given by sending an object $X$ to the induced module ${_T}T\otimes X$, which comes in addition with a $\comonC$-coaction
$$({_T}T\otimes X)
\to {_T}T\otimes (T\otimes X)$$
by inserting a unit. 
\end{example}

\section{Ménage à trois}\label{sec_troi}

\subsection{Basics}\label{subsec:troi}

In this section we are interested in the situation of a commuting triangle of categories and functors,
\begin{equation}\label{formula_TriangleSituation}
\xymatrix{
\catD 
\ar[rr]^{\funE}
\ar[rd]_{\funD}
&&
\catDD
\ar[ld]^{\funDD}
\\
&
\catB
&
\\
} 
\end{equation}

Supposing that the respective adjoints exist, then by uniqueness of the adjoint we have natural isomorphisms

$$\funE^\ra(\funDD^\ra(X))\cong \funD^\ra(X)$$
$$\funE^\la(\funDD^\la(X))\cong \funD^\la(X)$$

In the other direction, unit and counit of $E$ give a comparison

\begin{align}\label{formula_EcomparisonRa}
\funE(\funD^\ra(X))
\cong
&\funE(\funE^\ra(\funDD^\ra(X)))
\stackrel{\epsilon_{\funDD^\ra(X)}}{\relbar\joinrel\relbar\joinrel\relbar\joinrel\longrightarrow}
\funDD^\ra(X)\\ \label{formula_EcomparisonLa}
\funE(\funD^\la(X))
\cong
&\funE(\funE^\la(\funDD^\la(X)))
\stackrel{\eta_{\funDD^\la(X)}}{\longleftarrow\joinrel\relbar\joinrel\relbar\joinrel\relbar}
\funDD^\la(X)
\end{align}

\subsection{Restriction and extensions of scalars for monads}\label{subsec:monadRestriction}

We assume that the two functors to $\catB$ are monadic, then we are in the following general situation

\begin{equation}\label{formula_monadTriangleSituation}
\xymatrix@C=4em@R=4em{
\cat_{\monT} 
\ar@<3pt>[rr]^{\funE}
\ar[rd]^{\forget_\monT}
&&
\cat_{\monR}
\ar[ld]^{\forget_\monR}
\ar[ll]^{\funE^\la}
\\
&
\catB
\ar@<3pt>[lu]^{\funT}
\ar@<3pt>[ru]^{\funR}
&
\\
} 
\end{equation}

In this situation the functor $\funE$ and a left adjoint $\funE^\la$ can be expressed for monads exactly as in the classical algebraic version, which we first recall. This material is probably not new, but we did not find it in literature. Note that \cite{SZ24} discuss a more general program of Eilenberg-Watts theory for monads.

\begin{example}\label{exm_rings}
 Take algebras $T,R$ and an algebra morphism $\fpullback:T\leftarrow R$. These define monads over the category of vector spaces with $\monT(V)=T\otimes V$ and $\monR(V)=R\otimes V$. 
 
 \bigskip
 
 The restriction functor  (or pull back functor) is  
    \begin{align*}
        \funE=\fpullback^*:\Rep(T)&\longrightarrow \Rep(R)\\
        X&\mapsto {_\fpullback}X 
    \end{align*}
    The left adjoint is given by extension of scalars 
    \begin{align*}
    \funE^\la=\fpullback^*:\Rep(T)&\longleftarrow \Rep(R)\\
    T_\fpullback\otimes_R X&\mapsfrom X
    \end{align*}
   The unit and counit of the adjunction are given by using $1_T$ resp. by taking the tensor product further 
   \begin{align*}
       \eta:\;&X \to {_\fpullback}T_\fpullback\otimes_R X\\
    \epsilon:\; &T_\fpullback\otimes_R {_\fpullback}X \to T\otimes_T X\cong X
   \end{align*}
\end{example}

\begin{definition}
Let $\cat$ be a category and $\monT,\monR$ monads. Suppose there is a morphism of monads $\fpullback:\monT\leftarrow \monR$, that is, a natural transformation of endofunctors compatible with unit and multiplication. Then we have a restriction functor  
$$\fpullback^*: \cat_\monT\longrightarrow \cat_\monR$$
$$(X,\rho)\mapsto (X,\rho')$$
$$\rho':\monR(X)
\stackrel{\fpullback_X}{\longrightarrow}\monT(X)
\stackrel{\rho}{\longrightarrow} X$$ 
\end{definition}
\begin{proof}
We check that this defines a module, using naturality of $\fpullback$ and compatiblility with monad structure
\begin{align*}
&\monR(\monR(X))
\stackrel{R(\rho')}{\longrightarrow}
\monR(X)
\stackrel{\rho'}{\longrightarrow} 
X
\\
=\quad
&\monR(\monR(X))
\stackrel{R(\fpullback_X)}{\longrightarrow}
\monR(\monT(X))
\stackrel{\monR(\rho)}{\longrightarrow} \monR(X)
\stackrel{\fpullback_X}{\longrightarrow}
\monT(X)
\stackrel{\rho}{\longrightarrow} 
X
\\
=\quad
&\monR(\monR(X))
\stackrel{R(\fpullback_X)}{\longrightarrow}
\monR(\monT(X))
\stackrel{\fpullback_{T(X)}}{\longrightarrow}
\monT(\monT(X))
\stackrel{\monT(\rho)}{\longrightarrow} \monT(X)
\stackrel{\rho}{\longrightarrow} 
X
\\
=\quad
&\monR(\monR(X))
\stackrel{R(\fpullback_X)}{\longrightarrow}
\monR(\monT(X))
\stackrel{\fpullback_{T(X)}}{\longrightarrow}
\monT(\monT(X))
\stackrel{\mu^\monT_X}{\longrightarrow} \monT(X)
\stackrel{\rho}{\longrightarrow} 
X
\\
=\quad
&\monR(\monR(X))
\stackrel{\mu^\monR_X}{\longrightarrow} \monR(X)
\stackrel{\fpullback_X}{\longrightarrow}
\monT(X)
\stackrel{\rho}{\longrightarrow} 
X
\\
=\quad
&\monR(\monR(X))
\stackrel{\mu^\monR_X}{\longrightarrow} \monR(X)
\stackrel{\rho'}{\longrightarrow} 
X
\\
\end{align*}
and unitality similarly.
\end{proof}

The following result essentially says that any functor between categories of modules over monads that is compatible with the forgetful functors to $\cat$ comes from extension of scalars (or pullback) of monads. This is analog to the respective  statement for modules over rings and allows us to argue algebraically. 

\begin{lemma}\label{lm_monadRestriction}
Let $\cat$ be a category and $\monT,\monR$ monads. Suppose we have a functor
$$\funE: \cat_\monT\to \cat_\monR$$
compatible with the forgetful functors to $\cat$. Then there is a morphism of monads $\fpullback:\monT\leftarrow \monR$ such that $\funE=\fpullback^*$.
\end{lemma}
The proof proceeds as the corresponding algebra proof where a functor $\Rep(T)\to \Rep(R)$ is first presented as tensoring with a $R$-$T$-bimodule (Eilenberg–Watts theorem), then commuting with the forgetful functor means the bimodule is $T$ with right action the regular action, then applying the left action on the unit recovers $\fpullback:T\leftarrow R$. In this case, the functors $\monT,\monR$ are tensoring with $T,R$. \SimonRem{note that right exact, even exact faithful follows from the triangle setup}
\begin{proof}
To a functor assigns to a $\monT$-module $(X,\rho)$ a $\monR$-module $(\funE(X),\varrho^\funE_X(\rho))$. Because by assumption it commutes with the forgetful functor, we have $\funE(X)=X$. We now apply this in particular to induced modules $(X,\rho)=(\monT(Y),\mu^\monT_Y)$, then the functor must endow $\monT(Y)$ with $\monR$-module structure 
$$\varrho^\funE_{\monT(Y)}(\mu^\monT_Y):\; \monR(\monT(Y))\to \monT(Y)$$
Precomposing with the unit $\eta^\monT_Y:Y\to \monT(Y)$ produces a natural transformation
$$\fpullback_Y:\quad
\monR(Y)
\stackrel{\monR(\eta^\monT_Y)}{\longrightarrow}
\monR(\monT(Y))
\stackrel{\varrho^\funE_{\monT(Y)}(\mu^\monT_Y)}{\longrightarrow}
\monT(Y)$$

\bigskip

We now check that $\fpullback^*$ indeed recovers the functor $\funE$: Let $(X,\rho)$ be again an arbitrary module in $\cat^T$, then the functor $\funE$ maps it to $(X,\varrho_X^\funE(\rho))$.
On the other hand 
$\fpullback^*$ defined above sends is to $(X,\rho')$ with 
\begin{align*}
\rho'\quad
&\monR(X)
\stackrel{\monR(\eta_X^\monT)}{\longrightarrow}
\monR(\monT(X))
\stackrel{\varrho^\funE_{\monT(X)}(\mu_X^\monT)}{\longrightarrow}
\monT(X)
\stackrel{\rho}{\longrightarrow} 
X
\end{align*}
We have to show these two are equal. Recall that the action $\rho$ becomes a morphism between $\monT$-modules as follows (this is the counit of the adjunction $\cat^T\leftrightarrows \cat$)
$$\rho:(\monT(X),\mu_X^\monT)\longrightarrow (X,\rho)$$
Functoriality of $\funE$ then produces a morphism 
$$\funE(\rho):\;
(\monT(X),\varrho_{\monT(X)}^\funE(\mu_X^\monT))
\to 
(X,\varrho_{X}^\funE(\rho)
$$
Note that since by assumption the functor $\funE$ commutes with the forgetful functor also on the level of morphisms, we have $\funE(\rho)=\rho$ as morphisms in $\cat$. The fact that this is a morphism means the following square diagram commutes. Precomposing with $\monR(\eta_X^\monT)$ then proves $\rho'=\varrho_{X}^\funE(\rho)$:
\[\begin{tikzcd}[row sep=6ex, column sep=8ex]
\monR(X)
\arrow[rd, "\monR(\eta_X^\monT)"]
\\
&
\monR(\monT(X))
\arrow[r, "\monR(\rho)"]
\arrow[d, swap,"\varrho_{\monT(X)}^\funE(\mu_X^\monT)"]
&
\monR(X)
\arrow[d, "\varrho_{X}^\funE(\rho)"]
\\
&
\monT(X)
\arrow[r, "\rho"]
&
X
\end{tikzcd}\]
\end{proof}

We also have a left adjoint that can be explained precisely, just as in the case of rings

\begin{lemma}[Extension of scalars]\label{lm_monadExtensionOfScalars}
    Let $\fpullback:\monT\leftarrow \monR$ be a morphism of monads. Then then functor $\fpullback^*$ above has a left adjoint $\fpullback_*$. Explicitly, $\fpullback_*(X,\rho) $ given by the coequalizer of
\[\begin{tikzcd}[row sep=6ex, column sep=8ex]
\monT(\monR(X))
\arrow[r,  shift left=4pt,"\monT(\rho)"] 
\arrow[r, swap,shift right=4pt,"\mu^T_X\circ \monT(\fpullback_X)"]
& 
\monT(X)
\end{tikzcd}\]
The unit of this adjunction is given by the unit of $\monT$, the counit is given by the action of $\monT$. 
\end{lemma}

As already mention, by the uniqueness of adjoints the image of an induced module under the left adjoint is 
$$\fpullback_*:\; \funR(V)\mapsto \funT(V)$$
This can also be checked directly. As a consequence, we can compute the comparison morphism $\eta_{\funDD^\la(V)}$ from formula \eqref{formula_EcomparisonLa}, which is deduced from the adjunction's unit.
\begin{corollary}\label{cor:monadComparisonPullback}
After applying the forgetful functor, the comparison morphism  
$$
\forget_\monR(\fpullback^*(\funT(V)))
\cong
\forget_\monR(\fpullback^*(\fpullback_*(\funR(V))))
\xleftarrow{\forget_\monR(\eta_{\funR(V)})}
\forget_\monR(\funR(V))
$$
coincides with the the underlying morphism of monads
$$
\monT(V)
\stackrel{\,\fpullback_V}{\longleftarrow}
\monR(V)
$$
\end{corollary}

Quite useful is also the following formula for the left adjoint $\fpullback_*$ of quotients of induced modules, which can be directly computed. It is a monadic version of the algebraic statement that $T_\fpullback\otimes_R R/X\cong T/f(X)T$, where again $T,R$ are replaced by $T\otimes V,R\otimes V$ read as functors. 

\begin{lemma}\label{lm_monadExtensionOnRdurchX}
For any $\monT$-submodule $X$ of any induced module $\monR(V)$ the image of the left adjoint is 
$$\fpullback_*:\; \coker\big(X\stackrel{\iota}{\hookrightarrow}\funR(V)\big)
\longmapsto \coker\big(X\stackrel{\iota}{\hookrightarrow}\funR(V)\stackrel{\fpullback_V}{\rightarrow}\funT(V)\big)
$$
\end{lemma}
\begin{proof}
The adjunction follows from the following bijection of Hom sets
$$\Hom_{\catD_\monR}\left(\coker\big(X\stackrel{\iota}{\hookrightarrow}\monR(V)\big),\fpullback^*(Y)\right)
=
\Hom_{\catD_\monT}\left(\coker\big(X\stackrel{\iota}{\hookrightarrow}\monR(V)\stackrel{\fpullback_V}{\rightarrow}\monT(V)\big),Y\right)
$$
Namely, any morphism $g:\funT(V)\to Y$ killed by precomposing with $\fpullback_V\circ \iota$ can be used to define a morphism $h:\funR(V)\to Y$ killed by precomposing with $\iota$ by precomposing with $\fpullback$. Conversely any morphism $h:\funR(V)\to {_\fpullback}Y$ killed by precomposing with $\iota$ can be used to define a morphism $g:\funT(V)\to Y$ killed by precomposing with $\fpullback_V\circ \iota$.\SimonRem{More details}
\end{proof}

By going to the opposite category, all results in this section of course hold for comonads as well, which is the version we will need later on. Note that we keep the position of the stars in $f^*,f_*$ for simplicity and hope this does confuse experts.

\begin{equation}\label{formula_comonadTriangleSituation}
\xymatrix@C=4em@R=4em{
\cat^{\monT} 
\ar[rr]_{\funE}
\ar@<-3pt>[rd]_{\forget_\monT}
&&
\cat^{\monR}
\ar@<-3pt>[ll]_{\funE^\ra}
\ar@<-3pt>[ld]_{\forget_\monR}
\\
&
\catB
\ar[lu]_{\funT}
\ar[ru]_{\funR}
&
\\
} 
\end{equation}

\begin{definition}\label{def_comonadRestriction}
Let $\cat$ be a category and $\monT,\monR$ comonads. Suppose there is a morphism of comonads $\fpullback:\monT\rightarrow \monR$, that is, a natural transformation of endofunctors compatible with counit and comultiplication. Then we have a restriction functor by precomposition  
$$\fpullback^*: \cat^\monT\longrightarrow \cat^\monR$$
\end{definition}
\begin{lemma}\label{lm_comonadRestriction}
Let $\cat$ be a category and $\monT,\monR$ comonads. Suppose we have a functor
$$\funE: \cat^\monT\to \cat^\monR$$
compatible with the forgetful functors to $\cat$. Then there is a morphism of monads $\fpullback:\monT\rightarrow \monR$ such that $\funE=\fpullback^*$.
\end{lemma}
\begin{lemma}[Extension of scalars]\label{lm_comonadExtensionOfScalars}
    Let $\fpullback:\monT\rightarrow \monR$ be a morphism of comonads. Then then functor $\fpullback^*$ above has a right adjoint $\fpullback_*$. Explicitly, $\fpullback_*(X,\delta) $ given by the equalizer of

\[\begin{tikzcd}[row sep=6ex, column sep=8ex]
\monT(\monR(X))
\arrow[r, leftarrow, shift left=4pt,"\monT(\delta)"] 
\arrow[r, leftarrow, swap,shift right=4pt,"\monT(\fpullback_X)\circ \Delta^T_X"]
& 
\monT(X)
\end{tikzcd}\]
\end{lemma}

\begin{corollary}\label{cor:comonadComparisonPullback}
After applying the forgetful functor, the comparison morphism $\epsilon_{\funDD^{\ra}(V)}$ in formula \eqref{formula_EcomparisonRa}
$$
\forget_\monR(\fpullback^*(\funT(V)))
\cong
\forget_\monR(\fpullback^*(\fpullback_*(\funR(V))))
\xrightarrow{\forget_\monR(\epsilon_{\funR(V)})}
\forget_\monR(\funR(V))
$$
coincides with the the underlying morphism of comonads
$$
\monT(V)
\xrightarrow{\,\fpullback_V}
\monR(V)
$$
\end{corollary}

\subsection{Epireflectivity}\label{subsec:Nakayama}

Recall the situation \eqref{formula_monadTriangleSituation}, which we have shown in the previous section can be reduced to the following situation for $\fpullback:\monT\leftarrow \monR$ a morphism of monads and $\funE=\fpullback^*$ the restriction functor by Lemma \ref{lm_monadRestriction} and an explicit left adjoint $\funE^\la=\fpullback_*$ by  Lemma \ref{lm_monadExtensionOfScalars} - as in the example of algebra representations.

\begin{equation}
\xymatrix@C=4em@R=4em{
\cat_{\monT} 
\ar@<3pt>[rr]^{\fpullback^*}
\ar[rd]^{\forget_\monT}
&&
\cat_{\monR}
\ar[ld]^{\forget_\monR}
\ar[ll]^{\fpullback_*}
\\
&
\catB
\ar@<3pt>[lu]^{\funT}
\ar@<3pt>[ru]^{\funR}
&
\\
} 
\end{equation}

\begin{lemma}\label{lm_monadicNakayamaTrick}
Assume that $\funE=\fpullback^*$ is fully faithful, which is equivalent to the counit of the adjunction being an isomorphism by Lemma \ref{lm_AdjUnitInjective}. Then we claim that for any object $V\in\catB$ we have $\fpullback:\monT(V)\leftarrow \monR(V)$ epic.\end{lemma}

In the remainder of this section we prove this Lemma. Let us first remark:
\begin{itemize}
\item An embedding of categories admitting a left adjoint is called reflective and if the unit of the adjunction is epic, the embedding is called epireflective, see \cite{Riehl16} Definition 4.4.8. 
\item 
The statements is a slightly weaker consequence that the 
$\funE$ has a comparison morphism (Formula \eqref{formula_EcomparisonRa}) that is epic, meaning $\fpullback_{\monR(Y)}$ is epic for any ~$Y$.
\item A more sophisticated way to study this situation may be to also consider $E,E^\ra$ as a comonad, in this case an \emph{idempotent comonad} that defines the embedded subcategory by admitting a (unique) coaction. 
\end{itemize}

Let us also discuss that how this is a nontrivial statements for rings, relying on additional finiteness assumptions. Our general proof will be modeled after one particular proof of this ring-theoretic fact: 

\begin{example}["Which epimorphism of rings are surjective?"]

In Example \ref{exm_rings} we have discussed the restriction functor 
$$\fpullback^*:\Rep(T)\to \Rep(R)$$ 
for an algebra map $\fpullback:T\leftarrow R$. Assume that this functor is fully faithful, meaning: Every $R$-linear maps between $T$-modules is  automatically $T$-linear. Equivalently, the counit of the adjunction $T_\fpullback\otimes_R {_\fpullback}M\to M$ is an isomorphism, which is sufficient to check on the generator $M=T$. In this case $h$ is called an epimorphism of rings.

This is clearly true if $\fpullback$ is surjective, but in general the converse is not true: For example, the ring map $\fpullback:\mathbb{Q}\leftarrow \Z$ is surely not surjective but $\mathbb{Q}\otimes_\Z \mathbb{Q}\cong \mathbb{Q}$ and equivalently every $\Z$-linear map between $\mathbb{Q}$-modules is automatically $\mathbb{Q}$-linear. Similar behavior is encountered for every localization.

\begin{lemma}
   If $\fpullback^*$ is fully faithful and the $R$-module ${_\fpullback}T$ is a finitely generated, then $\fpullback$ is surjective, equivalently $\fpullback^*$ is epireflective.
\end{lemma}
\begin{proof}
Suppose $\fpullback$, which we write as a morphism of $R$-modules, has a nontrivial cokernel 
$$R\stackrel{\fpullback}{\longrightarrow}{_\fpullback}T
\longrightarrow \coker(\fpullback)
\longrightarrow 0$$

Applying the right exact functor $\fpullback_*$  with $\fpullback_*(R)=T$ gives 
$$T\stackrel{\fpullback_*(\fpullback)}{\longrightarrow}T_\fpullback\otimes_R {_\fpullback}T
\longrightarrow T_\fpullback\otimes_R \coker(\fpullback)
\longrightarrow 0$$

Note that $\fpullback_*(\fpullback):T\cong T_\fpullback\otimes_R R\to T_\fpullback\otimes_R {_\fpullback}T$ is equal to the unit $\eta_{T}:T\to T_\fpullback\otimes_R {_\fpullback}T$.
The assumption of $\fpullback^*$ fully faithful means by Lemma \ref{lm_AdjUnitInjective} precisely that the counit $\fpullback_*(\fpullback^*(X))\to X$ is iso, or explicitly $T_\fpullback\otimes_R {_\fpullback}T\cong T$, then we have the exact sequence
$$T\stackrel{=}{\longrightarrow} T
\longrightarrow T_\fpullback\otimes_R \coker(\fpullback)
\longrightarrow 0$$

This would easily conclude $\coker(\fpullback)=0$ from assuming $T_\fpullback$ is a faithfully flat $R$-module, but this is too restrictive in our situation. Instead we have assumed ${_\fpullback}T$ and hence $\coker(\fpullback)$ is a finitely generated $R$-module. We now use finite generation find a free presentation of $R$-modules
$$K\to R\otimes W\to {_\fpullback}T\to 0$$
note that in our more restrictive setting involving the adjunctions for $\monT,\monR$ we may directly choose $W=\forget_\monR( {_\fpullback}T)$. Since $\coker(\fpullback)$ is a quotient, we can use the same presentation with $K'\supset K$:
$$K'\to R\otimes W\to {_\fpullback}T\to 0$$
Then it is clear that 
$$(R\otimes W)/K'\neq 0
\quad\Longrightarrow\quad 
(R\otimes W)/K \otimes_R (R\otimes W)/K' \neq 0
$$

As a remark, note that it is not true that the left adjoint is conservative, i.e. $T_\fpullback\otimes_R M\neq 0$ for every $M$. This is not even true if $\fpullback$ is indeed surjective. Already in the example $\fpullback:\C\leftarrow \C\oplus \C$ projecting on the first factor we have $\fpullback^*:\Vect\to \Vect\oplus \Vect$ the embedding of the first summand and $\fpullback_*$ the functor projecting to the first summand. In some vague sense, the proof above shows that the left adjoint is conservative on the image of the functor itself.

\end{proof}
\end{example}

We now repeat this proof in the monadic setting. We note first: The counit $\epsilon$ can be expressed similarly to the algebra as the morphism between two coequalizers of $\monT(\monT(X))$ (once over $\monR$, once over $\monT$). The unit is 
$$\eta_X: X\stackrel{\eta^\monT_X}{\longrightarrow} \fpullback_*(\fpullback^*(X,\rho))=\mathrm{Coeq}(\monT(\monR(X))\rightrightarrows \monT(X))$$
which is clearly epic iff $\fpullback_X:\monR(X)\to \monT(X)$ epic. 

\bigskip

\begin{proof}[Proof of Lemma \ref{lm_monadicNakayamaTrick}] ~

1) Suppose for some object $V\in\catB$ we have a nontrivial cokernel   
$$\monR(V)\stackrel{\fpullback_V}{\longrightarrow}
\fpullback^*(\monT(V))
\longrightarrow \coker(\fpullback_V)
\longrightarrow 0$$
Applying the right exact functor $\fpullback_*$  with $\fpullback_*(\monR(V))=\monT(V)$ by Lemma  \ref{lm_monadExtensionOnRdurchX}
$$\monT(V)\stackrel{\fpullback_*(\fpullback)}{\longrightarrow} \fpullback_*(\fpullback^*(\monT(V))
\longrightarrow \fpullback_*(\coker(\fpullback))
\longrightarrow 0$$

Note that $\fpullback_*(\fpullback)$ is equal to the unit $\eta_{\monT(V)}$.
Using that $\fpullback_*$ was assumed fully faithful and hence the unit is iso we have 
$$\monT(V)\stackrel{=}{\longrightarrow} \monT(V)
\longrightarrow \fpullback_*(\coker(\fpullback))
\longrightarrow 0$$
Hence $\fpullback_*(\coker(\fpullback_V))=0$ and we want to conclude $\coker(\fpullback_V)=0$. 

\bigskip

2) The adjunction $\catB_\monR\leftrightarrows\catB$ in principles produces a free presentation of $\funT(V)$ and $\coker(\fpullback_V)$; in this sense this adjunction takes the role of relative finiteness. There are two complications compared to the algebra proof above:
\begin{itemize}
\item We need the presentations and their relation to be functorial. Hence we switch to the tensor category of right-exact endofunctors of $\catB$ and consider $\monR$ as an algebra object. In this view we can consider $\tilde{\monT}=\forget_\monR(\fpullback_*\fpullback^*)\funR$ as an endofunctor in $\catB$ with a left and right action of $\monR$; in the algebra example this is $T\otimes(-)$ with left and right action via $\fpullback$.
\item By composing from the left and right with the adjunction $\monR$, we obtain a free bimodule - note that the category of endofunctors is not braided and we need actions on both sides.
\end{itemize}

Now $\fpullback^*(\coker(\fpullback))$ is a quotient of $\tilde{T}$ in the category of $\monR$-bimodules in the category of endofunctors. 
Hence the tensor square over $\monR$ is not zero unless the quotient is zero. Hence $\fpullback$ is surjective as claimed.\SimonRem{More explicit}
\end{proof}

\section{The equivalence \texorpdfstring{$\catD^\k\cong \cU^\k$}{D with U}}\label{sec:QuantumGroupEquivalence}

\subsection{A relatively finite Schauenburg functor}\label{subsec:Schauenburg}

Assume that $\catD$ is a braided tensor category and $A$ a commutative algebra. By \cite{CLR23} Theorem 3.7 this lifts to the  Schauenburg functor $\Schauenburg$ to the Drinfeld center of $\catD_A$ relative to $\catC=\catD_A^\loc$.

\begin{equation}
\xymatrix@C=4em@R=4em{
\catD
\ar[rr]_{\Schauenburg}
\ar@<-3pt>[rd]_{A\otimes (-)}
&&
\cZ_\catC(\catD_A)
\ar@<-3pt>[ll]_{\Schauenburg^\ra}
\ar@<-3pt>[ld]_{\forget_c}
\\
&
\catD_A
\ar[lu]_{\forget_A}
\ar[ru]_{\coLaugwitz}
&
\\
} 
\end{equation}

Assume that the induction functor $A\otimes (-)$ is exact and fully faithful; for example this is proven in  \cite{CLR23} Lemma 3.4
 to be true under the following assumptions
\begin{assumption}[\cite{CLR23} Assumption 3.3.]\label{assumption_Fullyfaithful}
Assume that $\catD$ is a rigid, braided, locally finite abelian category with
trivial M\"uger center, the property that for any object X there exists a projective cover and an injective hull, and that $A$ is a commutative haploid algebra object, such that $\catD_A$ is rigid.
\end{assumption}
This has been much improved \cite[Theorem 3.20]{CMSY24}. That Theorem concludes that the induction functor is fully faithful provided a mild left exactness condition holds, provided that $\catD_A$ is rigid and provided that $\catD$ is an $r$-category and that induction from $\catD$ to $\catD_A$ respects the duality structure:
\begin{theorem}\textup{\cite[Theorem 3.20]{CMSY24}}
Let $\mathcal{C}$ be a locally finite $\Bbbk$-linear abelian braided
$r$-category with duality functor $D$, and let $A$ be a commutative
algebra in $\mathcal{C}$ such that the unit morphism
$\iota_A \colon \mathbf{1} \to A$ is injective. Also assume that the
following three conditions hold:
\begin{enumerate}
\item[(a)] The tensor category $\mathcal{C}_A$ 
in $\mathcal{C}$ is rigid.
\item[(b)] $I(DX) \cong I(X)^{*}$ as objects of
$\mathcal{Z}(\mathcal{C})^{\mathrm{loc}}_{A}$ for every object
$X \in \mathcal{C}$, where $I \colon \mathcal{C} \to
\mathcal{Z}(\mathcal{C})^{\mathrm{loc}}_{A}$ is the induction
functor.
\item[(c)] For any $\mathcal{C}$-subobject $s \colon S \hookrightarrow A$
such that $S$ is not contained in $\operatorname{Im}\iota_A$, there
exists an object $Z \in \mathcal{C}$ such that
$c_{Z,S} \neq c_{S,Z}^{-1}$ and the map
$s \otimes \mathrm{id}_Z \colon S \otimes Z \to A \otimes Z$ is
injective.
\end{enumerate}
Then the braided tensor functor
$I \colon \mathcal{C} \to \mathcal{Z}(\mathcal{C})^{\mathrm{loc}}_{A}$
is fully faithful.
\end{theorem}

Note that vertex tensor categories of self-contragredient vertex operator algebras are always $r$-categories. Rigidity and hence exactness of the induction functor is then
\cite[Theorem 3.21]{CMSY24} where in addition the splitting property of $\catD_A$ is required, that is that every simple object in $\catD_A$ is already local. 

If now $A\otimes (-)$ is exact and faithful by assumption, and so is $\forget_c$ without assumptions, then these functors together with their right adjoints are comonadic adjunctions. They are strict tensor functors, hence they define lax comonads, compare the list in Example \ref{exm_fourmonoidalmonads}. For the relative center, this comonadic description is clearly visible, and matches the  description of the Drinfeld center as the central comonad~$\comonZ$, see Example \ref{exm_centralcomonad}. For the algebra extension, the comonad $\comonC$ is precisely the reconstruction by descent in Section \ref{sec_deux} with $\forget_\comonC=\Ind_A$ and $\funC=\forget_A$
and  
$$\Ind_A':\catD\stackrel{\sim}{\longrightarrow}
(\catD_A)^\comonC$$
where $\Ind_A'(V)={_A}A\otimes V$ comes with $\comonC$-coaction ${_A}A\otimes V\mapsto {_A}A\otimes A \otimes V$ by inserting the algebra unit. Altogether the previous diagrams reads as follows, with $\catB=\catD_A$

\begin{equation}
\xymatrix@C=4em@R=4em{
\catB^\comonC
\ar[rr]_{\tilde{\Schauenburg}}
\ar@<-3pt>[rd]_{\forget_\comonC}
&&
\cZ_\catC(\catB)=\catB^\comonZ
\ar@<-3pt>[ll]_{\tilde{\Schauenburg}^\ra}
\ar@<-3pt>[ld]_{\forget_c}
\\
&
\catB
\ar[lu]_{\funC}
\ar[ru]_{\coLaugwitz}
&
\\
} 
\end{equation}

Here $\tilde{\Schauenburg}$ is actually the identity functor on objects in $\catB$, and it turns the $\comonC$-coaction on an object into a halfbraiding of the object, using the braiding in $\catD$. 

\bigskip

We now summarize results from the previous section in its comonad version
    \begin{enumerate}[(a)]
    \item By Lemma \ref{lm_comonadRestriction} there there is a morphism of comonads 
    \begin{align*}
    \fpullback:\comonC
    &\xrightarrow{\quad h\;\quad} \comonZ
    \\\Ind_A(\forget_A(X))
    =\forget_\comonC(\Ind_C(X))
    &\xrightarrow{\quad h_X\;\quad}  \forget_c(\coLaugwitz(X))),
    \end{align*}
    with $X\in \cB$ arbitrary,
    that realizes the Schauenburg functor as 
    $\tilde{\Schauenburg}=\fpullback^*$.
    \item As one important consequence, if $\fpullback_X$ is an isomorphism for any $X$, then $\tilde{\Schauenburg}=\fpullback^*$ is an equivalence of categories. 
    \item By Lemma \ref{lm_comonadExtensionOfScalars} there exists a right adjoint $\tilde{\Schauenburg}^\ra=\fpullback_*$, regardless of infiniteness and explicit, which is constructed analogously to a (co-)extension of (co-)rings
    \item By Lemma \ref{lm_monadicNakayamaTrick}, applied to the opposite category to hold for comonads, the assumption of  $\tilde{\Schauenburg}$ fully faithful implies 
    all $\fpullback_X$ are monomorphisms.
    \end{enumerate} 

Recall from Section \ref{subsec:troi} the comparison map obtained from the Schauenburg functor's adjunction counit
\[\quad \epsilon_{\coLaugwitz(X)}: \;\tilde{\Schauenburg}(\forget_A(X))\to \coLaugwitz(X)\]
From Lemma \ref{cor:comonadComparisonPullback} we know that after applying the forgetful functor $\forget_c$, this coincides simply with the realizing morphism of comonads $h$:
\[h=\forget_c(\epsilon_{\coLaugwitz(X)}):\; \Ind_A(X)\to \forget_c(\coLaugwitz(X))\]
We have seen that this in a monomorphism, hence it is sufficient to compare both sides in some quantifiable sense.

\begin{theorem}\label{thm:Schauenburg}
Suppose that the Schauenburg functor is exact and fully faithful, which is true for example under Assumption \ref{assumption_Fullyfaithful}. Suppose that $\cB$ is rigid and locally finite. Suppose that 
for all simple objects in $X$ in $\cB$ we have 
$$\Ind_A(X)\cong \forget_c(\coLaugwitz(X)$$
Then the Schauenburg functor is an equivalence of braided tensor categories. 
\end{theorem}
\begin{proof}
We have shown that $h:\Ind_A(X)\to \forget_c(\coLaugwitz(X))$ is a monomorphism, hence in locally finite category an isomorphism iff both sides are isomorphic. It is sufficient to check this for simple composition factors of $X$ because we know exactness of $\funC=\forget_A$ and of $\coLaugwitz$ by the exact tensor product in $\cB$. \SimonRem{Better argument beyond Hopf algebra case?}
\end{proof}

\begin{example}\label{exm:SchauenburgSplit}
We now assume in addition that $\catD_A\cong (\catD_A^\loc)_\NicholsOf$ for an algebra $\NicholsOf$, then we can rewrite $\cZ_\catC(\catD_A)=\YD{\NicholsOf}(\catC)$ and have an explicit description  in Section \ref{subsec:QuantumGroup} of the right adjoint $\coLaugwitz(X)={_{\text{ad}}^{\text{reg}}}(\NicholsOf\otimes X)$ in terms of adjoint action on $\NicholsOf$ with a double braiding on $X$ as well as the regular coaction solely on $\NicholsOf$. 
\begin{align*}
\intertext{
Then the comparison morphism is more explicitly a morphism in $\catD$}
\epsilon_{\coLaugwitz(X)}:\;
\forget_A(X)
&\longrightarrow
{_{\text{ad}}^{\text{reg}}}(\NicholsOf\otimes X)
\\
\intertext{and after applying $\forget_\comonC=\Ind_A$ we have the morphism $h$ in $\catD_A$}
h_X:\;
\Ind_A(\forget_A(X))
&\longrightarrow
{_{\text{ad}}}(\NicholsOf\otimes X)
\\
\intertext{and with the (faithful) splitting functor we can go even further to a morphism in  $\catC$}
\mathrm{split}(h_X):\;
\mathrm{split}(\Ind_A(\forget_A(X)))
&\longrightarrow
\;\;\NicholsOf\otimes X
\end{align*}
If $\cC$ has an exact tensor product, then it is obviously enough to test that the two sides are isomorphic for $X=\1={_A}A$, meaning we have to  test if $\Ind_A(A)\cong {_{\text{ad}}}\NicholsOf$.
\end{example}

\noindent
We have some remarks on this result:
\begin{itemize}
\item It is not really necessary to show an isomorphism of objects, it is sufficient to for example compare composition factors or some other quantity, such as  Frobenius Perron dimension of objects. 
\item If $\cC$ is the category of local modules in $\cB$, then the equivalence should hold without additionally comparing the objects, as we have shown in the finite case in \cite{CLR23} Theorem 3.7 using Frobenius-Perron dimension. 
\item If we forget about monads, we are actually just showing the isomorphism property of the comparison map \eqref{formula_EcomparisonRa} built from the adjunctions counit. It is known that if unit and also counit is an isomorphism, then we have an equivalence of categories. Note however, that with the monadic argument we only check this for objects induced from $\cB$, and we prove the existence of an adjoint in the infinite case in the first place. 
\item In cases where the adjunction $\cD\to\cD_A$ is not comonadic, for example in the $(p,q)$-model, we can still compare the comonadic category via the Schauenburg functor and discuss the difference to the actual category. In \cite{Riehl16} it is explained how the category is between between the Klesili and Eilenberg Category of the comonad.
\end{itemize}
\SimonRem{
\begin{itemize}
    \item The Schauenburg functor is a lax monoidal functor, preserving coalgebras,  which means that $h$ preserves the lax structure on the comonads, which we have not entered in our picture.  
    \item We could rewrite the braiding of $\cD$ to a $r$-matrix of the comonad $\comonC$, to incorporate this crucial structure. One goal could be to make $\tilde{\Schauenburg}$ more explicit. 
    \item We could write $\tilde{\Schauenburg},\tilde{\Schauenburg}^\ra$ itself as an idempotent monad to maybe gain more insight.
\end{itemize}
}

\subsection{Proof of Theorem 1.1}\label{subsec:QuantumGroupEquivalence}

We have already established in Theorem \ref{thm:BorelEquivalence} an equivalence of tensor categories 
$$(\cD^\k)_A\cong \cB^\k$$
where $\cB^\k$ is the  category of modules over $\NicholsOf(M)$ in $\cC^\k$ for $M=\VirPhi_{1,2}\boxtimes \Pi_{1}(-t/2)$ for $v\geq 3$ resp. $M=\VirPhi_{1,1}\boxtimes \Pi_{2}(-t)$ for $v=2$. \\

We have proven in Section \ref{subsec:sl2Amodules}
that $\catD^\k$ and $(\catD^\k)_A$ are rigid and that the Schauenburg functor is exact and fully faithful.
We apply Theorem \ref{thm:Schauenburg} in our situation to prove Theorem \ref{thm:main}.\\

To fulfill the last assumption $\Ind_A(\forget_A(X))\cong \forget_c(\coLaugwitz(X))$ to apply the theorem, we refer to the specific formulae in Example \ref{exm:SchauenburgSplit}, where it is also shown that it is sufficient to check the case $X=\1$. Then  we have in \ref{cor:InductionA} computed the induction 
$$\Ind_A(A)\cong \1\oplus K$$
as well as in Example \ref{exm:ourNicholsAdjoint} the adjoint representation of our Nichols algebra  
$${_{\text{ad}}}\NicholsOf(M)\cong \1\oplus K$$
Since these coincide, we can apply Theorem \ref{thm:Schauenburg} and this concludes the proof of Theorem \ref{thm:main}.\flushright \qedsymbol

\section{Bibliography}
\bibliographystyle{alphabetic}

\end{document}